\documentclass[12pt]{article}
\usepackage{amsmath}
\usepackage{amssymb}
\usepackage{theorem}
\usepackage{enumerate} 
\usepackage{graphicx, tikz} 
\usepackage{hyperref}

\numberwithin{equation}{section}

\newtheorem{thm}{Theorem}[section]
\newtheorem{lemma}[thm]{Lemma}
\newtheorem{sublemma}[thm]{Sublemma}
\newtheorem{prop}[thm]{Proposition}
\newtheorem{cor}[thm]{Corollary}
{\theorembodyfont{\rmfamily}
\newtheorem{defn}[thm]{Definition}

\newtheorem{rmk}[thm]{Remark}
}

\newcommand{\hQ}{{\widehat Q}}

\newcommand{\hB}{{\widehat B}}
\newcommand{\hA}{{\widehat A}}
\newcommand{\hE}{{\widehat E}}

\newcommand{\hPsi}{{\widehat \Psi}}

\newcommand{\tQ}{{\widetilde Q}}

\newcommand{\qed}{\hfill \mbox{\raggedright \rule{.07in}{.1in}}}
 
\newenvironment{proof}{\vspace{1ex}\noindent{\bf
Proof}\hspace{0.5em}}{\hfill\qed\vspace{1ex}}
\newenvironment{pfof}[1]{\vspace{1ex}\noindent{\bf Proof of
#1}\hspace{0.5em}}{\hfill\qed\vspace{1ex}}

\def\R{\mathbb{R}}

\def\Z{\mathbb{Z}}
\def\C{\mathbb{C}}

\def\cB{\mathcal{B}}

\newcommand{\eps}{{\epsilon}}
\newcommand{\Leb}{\operatorname{Leb}}

\newcommand{\supp}{\operatorname{supp}}

\newcommand{\diam}{\operatorname{diam}}

\title{Strong mixing for the periodic Lorentz gas flow with infinite horizon
}

\author{
{Fran\c{c}oise P\`ene}
\thanks{Univ Brest, Universit\'e de Brest, LMBA,
UMR CNRS 6205, 
6 avenue Le Gorgeu, 29238 Brest cedex, France}
\and
Dalia Terhesiu
\thanks{Mathematisch Instituut,
University of Leiden, Niels Bohrweg 1, 2333 CA Leiden, Netherlands}
}

\begin{document}

 \maketitle
 
\begin{abstract}
We establish strong mixing for the $\mathbb Z^d$-periodic, infinite horizon, Lorentz gas flow for continuous observables with compact support.
The essential feature of this natural class of observables is that their support may contain points with  infinite  free flights. 
Dealing with such a class of functions is a serious challenge
and there is no analogue of it in the finite horizon case.
The mixing result for the aforementioned class of functions
is obtained via new results: 1) mixing for
continuous observables with compact support consisting of configurations 
at a bounded time from the closest collision;
 2) a tightness-type result
that allows us to control the configurations with long free flights.
To prove 1), we establish a mixing local limit theorem
for the Sinai billiard flow with infinite horizon, previously an open question.
As far as we know, our approach to the tightness result has no analogue in the literature. 
\end{abstract}
 
\section{Introduction and Main result}\label{sec:intro}
We are interested in mixing for the continuous time  dynamics of the $\mathbb Z^d$-periodic Lorentz gas ($d\in\{1,2\}$). This model has been introduced by
Lorentz in \cite{Lorentz05} to model the diffusion of electrons in a low conductive metal. 
It describes the behaviour of a point particle moving at unit speed in the plane $\mathcal D_2:=\mathbb R^2$ (when $d=2$) or on the tube $\mathcal D_1:=\mathbb R\times \mathbb T$ (when $d=1$, writing as usual $\mathbb T:=\mathbb R/\mathbb Z$ for the one-dimensional torus) between a $\mathbb Z^d$-periodic locally finite configuration of convex obstacles with disjoint closures and $\mathcal C^3$ boundary (with non null curvature), with elastic collisions on them (pre-collisional and post-collisional angles being equal). We write $\Omega_d$ for the set of possible positions, that is the set of positions in $\mathcal D_d$ that are not inside an obstacle.

The set of configurations is the set $\widetilde{\mathcal M}$  of
couples of position and unit velocity $(q,\vec v)\in \Omega_d\times \mathbb S^1$, identifying pre-collisional and post-collisional vectors at a collision time (rigorously, $\widetilde{\mathcal M}$ is the quotient of $\Omega_d\times \mathbb S^1$ by the equivalence relation identifying pre- and post-collisional vectors).   
The Lorentz gas flow $(\Phi_t)_t$ maps a configuration $(q,\vec v)$ (corresponding to a couple position and velocity at time 0) to the configuration $\Phi_t(q,\vec v)=(q_t,\vec v_t)$ corresponding to the couple position and velocity at time $t$ of
a particle that was at time $0$ at position $q$ with velocity $\vec v$. 
This flow $(\Phi_t)_t$ preserves the infinite Lebesgue measure $\widetilde\nu$ on $\Omega_d\times\mathbb S^1$, normalized so that $\widetilde\nu\left((\Omega_2\cap[0,1[^2)\times\mathbb S^1\right)=1$ if $d=2$ and so that $\widetilde\nu\left((\Omega_1\cap([0,1[\times\mathbb T))\times\mathbb S^1\right)=1$ if $d=1$.
It is natural to consider also the dynamics at collision times. The space $\widetilde M$ for this dynamics is the set
of configurations $(q,\vec v)\in\widetilde{\mathcal M}$ with $q\in\partial \Omega_d$. The collision map $\widetilde T:\widetilde M\to\widetilde M$, that maps a configuration at a collision time to the configuration at the next collision time, is referred to as the Lorentz gas map and preserves an infinite measure $\widetilde\mu$ absolutely continuous with respect to the Lebesgue measure. Let us write $\widetilde W_t:\widetilde{\mathcal M}\rightarrow \mathcal D_d$ for the map corresponding to the displacement up to time $t$~:
\[
\forall (q,\vec v)\in\widetilde{\mathcal M},\quad 
\Phi_t(q,\vec v)=(q_t,\vec v_t)\quad\Rightarrow\quad \widetilde W_t(q,\vec v)=q_t-q\, .
\] 
If $d=1$ we set $\widetilde W'_t:\widetilde{\mathcal M}\rightarrow \mathbb R$ for the first coordinate of $\widetilde W_t$. 
If $d=2$, $\widetilde W_t$ takes its values in $\mathbb R^2$, we then just set $\widetilde W'_t=\widetilde W_t$. In both cases,  $\widetilde W'_t$ is the natural projection of $\widetilde W_t$
on $\mathbb R^d$.

 
When every trajectory touches eventually at least one obstacle, 
 we speak of \emph{finite horizon} Lorentz gas. In the finite horizon case, it follows from \cite{BS81,BCS91} that 
 $\widetilde W_t$ satisfies a standard central Limit Theorem meaning that $(\widetilde W'_t/\sqrt{t})_t$ converges strongly in distribution\footnote{In this article, the strong convergence in distribution means the convergence in distribution with respect to any probability measure absolutely continuous with respect to the Lebesgue measure.}, as $t\rightarrow +\infty$, to a centered Gaussian random variable with non degenerate variance matrix given by an infinite sum.\\
 When there exists at least a trajectory that never touches an obstacle, we speak of \emph{infinite horizon} Lorentz gas. In this article, we focus on the "fully dimensional" infinite horizon case, meaning that there exist at least $d$ non parallel unbounded trajectories touching no obstacle. 
 In this case it follows from \cite{SV07} by Sz\'asz and Varj\'u (see subsection~\ref{subsec:mllt}, in particular Proposition~\ref{prop:cltflow}
 for details) that\footnote{The notation $\Longrightarrow \mathcal N(0,C)$ means the strong convergence in distribution to a Gaussian random variable of distribution $\mathcal N(0,C)$, that is centered with variance matrix $C$.}
 \begin{equation}\label{eq:Sigma}
  \frac{\widetilde W'_t}{\sqrt{t\log t}}\Longrightarrow \mathcal N(0,\Sigma)\, ,\quad \mbox{as }t\rightarrow +\infty\, ,
 \end{equation}
 where $\Sigma$ is a $d$-dimensional definite positive 
 symmetric matrix which, furthermore, is given by an explicit formula in terms of the configuration of obstacles  (recalled at the beginning of Section~\ref{sec:proofjCLT}).

We are interested here in the question of strong mixing. 
We recall that an infinite measure preserving system $(\hat X,\hat T,\hat\mu)$ is said to be strongly mixing if there exist
a sequence $\widehat{a}_n\to\infty$ and a class of integrable functions $f,g$ so that
\begin{equation}\label{def:mix}
 \widehat{a}_n\int_{\hat M}f.g\circ \hat T^n\, d\hat\mu
\to\int_{\hat M}f\, d\hat\mu
\int_{\hat M}g\, d\hat\mu\, ,
\end{equation} 
as $n\rightarrow +\infty$. 
The sequence $\widehat{a}_n$ gives the speed of convergence to $0$ of $\int_{\hat M}f.g\circ \hat T^n\, d\hat\mu$. The first such rate was obtained in~\cite{Thaler00} for a very restrictive
class of intermittent maps preserving an infinite measure. This was later generalized to larger
classes of such maps in~\cite{MT12} and~\cite{Gouezel11}.
For other notions of mixing in the infinite measure set up (such as local-global and global-global) 
were introduced in~\cite{Lenci10} (see also~\cite{DNL21,DNP22} and references therein).

In the set up of the \emph{discrete} time Lorentz gas $(\widetilde M,\widetilde T,\widetilde\mu)$, mixing in the sense of~\eqref{def:mix} is well understood in both finite and infinite horizon case and it is a direct consequence of a mixing local limit theorem (MLLT) for the cell change (see e.g. \cite[Section 3]{Pene18b}).
For the finite horizon case, we refer to~\cite{SV04} for the key LLT (which can be generalized in MLLT and thus provides mixing) and
to~\cite{Pene18} for expansions of any order. In the
much more difficult set up of infinite horizon case, we refer
to~\cite{SV07} for LLT (which, again, leads to MLLT and  mixing) and to~\cite{PeneTerhesiu21} for error terms. 
There is a plethora of limit theorems known in the discrete time set up with finite horizon case. Some results are also known for the discrete time Lorentz gas with infinite horizon case (in particular,~\cite{SV07,ChDo09} and more recently,~\cite{PeneTerhesiu21}).

Mixing for \emph{continuous} time Lorentz gas
$(\widetilde{\mathcal M},(\Phi_t)_t,\widetilde\nu)$ is seriously more challenging.
Even in the set up of the \emph{finite} horizon Lorentz gas flow, mixing in the sense of~\eqref{def:mix} was open
until the work of~\cite{DN20} and very recently, expansions of any order have been obtained in \cite{DNP22}. Strictly speaking, the work~\cite{DN20} focused on a mixing local limit theorem (MLLT) for the Sinai billiard flow with \emph{finite} horizon, but as in \cite{DNP22}, mixing in the sense of~\eqref{def:mix} and MLLT are equivalent.
For related, but weaker, results on MLLT  for group extensions
of suspension flows with bounded roof function, not applicable as such to Sinai billiards we refer to~\cite{AT20}.

Nothing is known about the mixing for the Lorentz gas flow with \emph{infinite} horizon. In this paper we address this open question and establish
\begin{thm}\label{cor:mixingrate}
	For any 
	continuous compactly supported functions $f,g:\Omega_d\times\mathbb S^1 \rightarrow\mathbb R$,
	\begin{equation}\label{mixing2}
	\int_{\widetilde{\mathcal M}}f.g\circ 
	\Phi_t
	\, d\widetilde\nu\sim \frac {\int_{\widetilde{\mathcal M}}f\, d\widetilde\nu \int_{\widetilde{\mathcal M}}g\, d\widetilde\nu}{(2\pi t\log t\det(\Sigma))^{\frac d2}}\, ,\quad\mbox{as }t\rightarrow +\infty \, ,
	\end{equation}
where $\Sigma$ is the variance matrix appearing in~\eqref{eq:Sigma}.
\end{thm}
Theorem~\ref{cor:mixingrate} gives mixing for observables 
with support that may contain configurations with infinite free flights.
In the set up of the Lorentz gas flow with infinite horizon, this class of observables is the natural one. Theorem~\ref{cor:mixingrate} can be rephrased in terms of vague convergence (see comments after Corollary~\ref{cor:easymix}).
The main ingredients, which are new and important results on their own, used in the proof of Theorem~\ref{cor:mixingrate} are 
\begin{enumerate}
 \item Strong mixing for observables $f,g$ with 
 supports uniformly 'close' to a collision time (Corollary~\ref{cor:easymix} of Proposition~\ref{prop:MLLTkappa0}), i.e. the supports of $f,g$ are at a bounded time of the closest collision time (either in the past or in the future). This mixing result is an easy consequence of a MLLT for the Sinai billiard flow with infinite horizon.
 The  present MLLT, Proposition~\ref{prop:MLLTkappa0} and its variant Theorem~\ref{thm:main}, are 
 first main results and are established
 via two joint local limit results for the Sinai billiard map on the cell change function and flight time together: a) Joint MLLT, Lemma~\ref{lem:jointllt}; b) Joint Local Large Deviation,
 Lemma~\ref{lem:lld}.
 
 \item A tightness type result, Theorem~\ref{lem:tight}, that allows $f,g$ to have  any compact support in $\widetilde{\mathcal M}$. In particular, the supports of $f,g$ can contain 
 configurations of particles that will never hit an obstacle.
 The proof of Theorem~\ref{lem:tight} provided in Section~\ref{sec:proofmixing} exploits a very delicate decomposition
 of the type of possible free flights along with Joint MLLT \emph{with good error terms} for the Sinai billiard map
 (as in Lemma~\ref{lem:jointllt0}  and Corollary~\ref{coro:llt0cor}), 
 combined with a series of subtle new estimates (including a large deviation estimate).
 We emphasize that the Joint MLLT with \emph{error terms} (not only the JMLLT mentioned in the above item, and even with sharper error terms than the one obtained in~\cite{PeneTerhesiu21} in the non joint MLLT) together with several other new technical estimates are required ingredients for the proof of Theorem~\ref{lem:tight}.
\end{enumerate}

We conclude the introductory section with a very brief summary of the various  results along with an outline of the paper.

In Section~\ref{sec:MLLT}, we introduce most of the required notations, and state MLLTs for the Sinai billiard flow
as in Proposition~\ref{prop:MLLTkappa0} for the cell change function, and
as in Theorem~\ref{thm:main} for the flight function. In Section~\ref{sec:MLLT}
we also record a consequence of Proposition~\ref{prop:MLLTkappa0},
namely  Corollary~\ref{cor:easymix} that proves mixing for  continuous observables with compact support consisting of configurations 
at a bounded time from the closest collision; in short, this gives mixing
for continuous observables supported on a region on which the free flights (either in the past or in the future) are uniformly bounded.

In Section~\ref{sec:MLLT+LLD}, we state the joint limit results (Joint CLT, Joint MLLT with \emph{error terms}, Joint LLD) for the Sinai billiard map, 
for the couple formed by the  cell change function with the flight time. 
Using the statement of these key technical ingredients, in Sections~\ref{sec:proofMLLT} and~\ref{sec:proofmixing} we prove
Theorem~\ref{thm:main}, and Theorem~\ref{cor:mixingrate}.

The proofs of the technical key results stated in Section~\ref{sec:MLLT+LLD} are included in Sections~\ref{sec:proofjCLT},~\ref{sec:proofjMLLT} and in Appendix~\ref{sec:jLLD}. 
While the joint LLD, Lemma~\ref{lem:lld}, follows by slightly modifying  the proof of LLD for the cell change function obtained in~\cite{MPT},
all the other technical results obtained in this paper,
namely, the Joint CLT and the Joint MLLT with \emph{error terms}
stated in Section~\ref{sec:MLLT+LLD} are new and require serious new ideas and work.

In Section~\ref{sec:proofmixing} we state and prove the tightness result Theorem~\ref{lem:tight}. 
At the beginning of Section~\ref{sec:proofmixing}, we use the statement of 
Theorem~\ref{lem:tight} to complete the proof of Theorem~\ref{cor:mixingrate}.
The role and the novelty of Theorem~\ref{lem:tight} has been already summarized in item 2 above. Finally, we mention that, in Section~\ref{sec:proofmixing},
as a by product of certain technical lemmas we obtain a large deviation result, namely Proposition~\ref{prop:L}, which is of independent interest.

\vspace{-2ex}
\paragraph{Notation}
We use ``big O'' and $\ll$ notation interchangeably, writing $b_n=\mathcal O(c_n)$ or $b_n\ll c_n$
if there are constants $C>0$, $n_0\ge1$ such that
$b_n\le Cc_n$ for all $n\ge n_0$.
As usual, $b_n=o(c_n)$ means that there exists $\varepsilon_n$ such that, for all $n$ large enough, $b_n=c_n\varepsilon_n$ and $\lim_{n\rightarrow +\infty}\varepsilon_n=0$
and $b_n\sim c_n$ means that $b_n=c_n+o(c_n)$. 
Unless otherwise specified, given $x\in\R^d$, we let $|x|$ be the usual Euclidean norm of $x$. Throughout this article, when $d=1$, we identify  $\mathbb Z^1$ (resp. $\mathbb R^1$) with $\mathbb Z\times\{0\}$ (resp. $\mathbb R\times\{0\}$). In particular, for any $(q,z)\in\Omega_1\times\mathbb R$, the notation $q+z$ means $q+(z,0)$.

\section{MLLT for the Sinai billiard flow and mixing for the $\mathbb Z^d$-extension flow}
\label{sec:MLLT}
\subsection{Notations and previous results}\label{sec:obstacles}
Let $d\in\{1,2\}$.
The domain $\Omega_d$ of the $\mathbb Z^d$-periodic Lorentz gas is given by $\Omega_d:=\mathcal D_d\setminus
\bigcup_{i=1}^I\bigcup_{\ell\in\mathbb Z^d}(\mathcal O_i+\ell)$  where  $\mathcal O_1,...,\mathcal O_I$ is a nonempty finite family of convex open sets with $\mathcal C^3$ boundary of non null curvature such that the obstacles $\mathcal O_i+\ell$ have pairwise disjoint closures. We recall that we are interested in the fully dimensional infinite horizon and so assume throughout that
the interior of the billiard domain $\Omega_d$ contains at least $d$ unbounded corridors (made of unbounded parallel lines) the direction of which are not parallel to each other.
\subsubsection*{Sinai billiard}
Quotienting the system $(\widetilde{\mathcal M},(\Phi_t)_t)$ by $\mathbb Z^d$ (for the position), we obtain
the 
Sinai billiard flow $(\mathcal M,(\phi_t)_t)$ (see \cite{Sinai70}) which describes the evolution of point particles moving at unit speed in $\Omega:=\Omega_d/\mathbb Z^d=\mathbb T^2\setminus \bigcup_{i=1}^I\overline{\mathcal O_i}$  with elastic reflection off $\partial\Omega$ (where $\overline{\mathcal O_i}$ is the image of $\mathcal O_i$ by the canonical projection $p_d:\mathcal D_d\rightarrow \mathbb T^2$). The flow $(\phi_t)_t$ preserves the probability measure $\nu$ on $\mathcal M=(\Omega\times\mathbb S^1)/\equiv$ that is proportional to the Lebesgue measure,
where $\equiv$ is the equivalence relation identifying pre- and post-collisional vectors. 
The Poincar\'e map
of $\phi_t$ with Poincar\'e section $\partial\Omega\times\mathbb S^1$ is the Sinai billiard map $(M,T,\mu)$,
where the two-dimensional phase space $M=\{(q,\vec v)\in\mathcal M:q\in\partial\Omega\}$ (position in $\partial\Omega$ and unit post-collisional velocity vector) is identified with  $\partial\Omega\times (-\pi/2,\pi/2)$ (we parametrise here the post-collisional velocity vector by its angle with the normal to $\partial\Omega$). 
This map $T$ sends a post-collisional vector to the post-collisional vector corresponding to the next collision.
This map preserves the probability measure $\mu$ with density 
$\cos\varphi/(2|\partial\Omega|)$ at the point $(q,\varphi)\in\partial\Omega\times[-\frac \pi 2,\frac\pi 2]$. 
The flight time between consecutive collisions
is the return time of $(\phi_t)_t$ to $M$ and we denote it by $\tau:M\to\R_{+}$.
In this notation, we 
have the following identification
\begin{align*}
T(x)&=(\phi_\tau)(x)=\phi_{\tau(x)}(x)\, .
\end{align*}
We set $\tau_n:=\sum_{k=0}^{n-1}\tau\circ T^k$, with the usual convention $\tau_0:=0$. 
For any $x\in \mathcal M$, we set $N_t(x)\in\mathbb N_0$ for the
\emph{collisions number} in the time interval $(0,t]$ starting from the configuration $x$. We observe that, for $x\in M$, this quantity satisfies
\begin{equation}\label{eq:lapn}
\tau_{N_t(x)}(x)\le t<\tau_{N_t(x)+1}(x)\, .
\end{equation}
Furthermore, for all $x\in M$ and all $u\in[0,\tau(x))$ and any $t\in [0,+\infty)$, $N_{t}(\phi_u(x))=N_{t+u}(x)$.
With these notations, the Sinai billiard flow $(\mathcal M,(\phi_t)_t,\nu)$ is isomorphic to the
suspension flow  $(\widehat{\mathcal M},(\widehat{\phi}_t)_t,\widehat\nu)$, given by
\begin{align*}
\widehat{\mathcal M}&=\{(x,u)\in M\times[0,+\infty):0\le u< \tau(x)\}\\
\widehat\phi_s(x,u)&=(T^{N_{s+u}(x)}(x),s+u-\tau_{N_{s+u}(x)}(x))\\
\widehat\nu&=(\mu\times \Leb)/{\mu(\tau)},\quad\mbox{where  }\mu(\tau):=\int_M\tau\,d\mu
\, ,
\end{align*}
via the isomorphism $(x,u)\in\widehat{\mathcal M}\mapsto \phi_u(x)\in\mathcal M$ (this map is injective, its image is the set of configurations in $\mathcal M$ that do not belong to an infinite free flight).

\subsubsection*{$\mathbb Z^d$-extension and cell change function}
We recall that the $\mathbb Z^d$-periodic Lorentz gas map $(\widetilde M,\widetilde T,\widetilde\mu)$ can be represented by the $\mathbb Z^d$-extension of the Sinai
billiard map $(M,T,\mu)$ by the \emph{cell change} function $\kappa$ that can be defined as follows. 
For any $\ell\in\mathbb Z^d$, we call $\ell$-cell
the set $\mathcal C_\ell$ of configurations $(q,v)\in\widetilde M$
such that $q\in\bigcup_{i=1}^I(\partial O_i+\ell)$. 
Because of the $\mathbb Z^d$-periodicity of the model, there exists $\kappa:M\rightarrow\mathbb Z^d$, called the \emph{cell change function}, such that 
\begin{equation}\label{eq:defkappa}
\widetilde x=(q,\vec v)\in\mathcal C_\ell\quad\Rightarrow\quad \widetilde T(x)\in\mathcal C_{\ell+\kappa(p_d(q),\vec v)}\, .
\end{equation}
Note that, for any $\widetilde x\in\widetilde M$, there exists a unique $x=((q,\vec v),\ell)\in M\times\mathbb Z^d$ such that\footnote{Recall that, if  $d=1$,
	we identify $\mathbb Z^1$ with $\mathbb Z\times\{0\}$, meaning that  for any $q'\in\mathcal D_1$ and any $\ell\in\mathbb Z^1$, the notation $q'+\ell$ means $q'+(\ell,0)$.} 
$\widetilde x=(p_{d,0}^{-1}(q)+\ell,\vec v)$, where $p_{d,0}$ denotes the restriction of $p_d$ to $\bigcup_{i=1}^I\partial \mathcal O_i$ ($\vec v$ is the velocity of $\widetilde x$, setting $\widetilde q$ for the position of $\widetilde x$, $(q,\ell)$ is such that $q=p_{d}(\widetilde q)$ and $\widetilde x\in\mathcal C_\ell$). 
Formula~\eqref{eq:defkappa} can be rewritten under the form
\begin{equation}\label{Zdextension}
\forall ((q,\vec v),\ell)\in M\times\mathbb Z^d,\ 
T(q,\vec v)=(q',\vec v')\ \Rightarrow\ \widetilde T(p_{d,0}^{-1}(q)+\ell,\vec v) = \left(p_{d,0}^{-1}(q')+\ell+\kappa(q,\vec v),\vec v'\right)\, .
\end{equation}
This gives the identification of $(\widetilde M,\widetilde T,\widetilde\mu)$
by the $\mathbb Z^d$-extension of $(M,T,\mu)$ by 
$\kappa:M\rightarrow \mathbb Z^d$. 
A direct and classical induction ensures that, for any
$ ((q,\vec v),\ell)\in M\times\mathbb Z^d$ and any $n\in\mathbb N$,
\begin{equation}\label{Zdextensionforn}
T^n(q,\vec v)=(q'_n,\vec v'_n)\ \Rightarrow\  \widetilde T^n(p_{d,0}^{-1}(q)+\ell,\vec v)= \left(p_{d,0}^{-1}(q'_n)+\ell+\kappa_n(q,\vec v),\vec v'_n\right)\, ,
\end{equation}
where we set $\kappa_n:=\sum_{j=0}^{n-1}\kappa\circ T^j$.

\subsection{Mixing for the Lorentz gas seen as a suspension flow}
We will use crucially the fact established in the previous section that $(\widetilde{\mathcal M},(\Phi_t)_t,\widetilde\nu)$ can be represented
as a suspension flow by $(x,\ell)\mapsto \tau(x)$ over   $(\widetilde M,\widetilde T,\widetilde\mu)$ which itself can be represented as a $\mathbb Z^d$ extension of $(M,T,\mu)$
by $\kappa$. Thus, we can represent $\widetilde {\mathcal M}$
by $\widehat{\mathcal M}\times\mathbb Z^d$. 
In this part, we state a mixing local limit theorem for
$\kappa_n$ and see how we can  use it to easily derive
Theorem~\ref{cor:mixingrate} in the case of functions $f,g$ with support at a bounded time from a collision, i.e. for functions that are compactly supported in 
$\widehat{\mathcal M}\times\mathbb Z^d$. As detailed in Section~\ref{sec:proofmixing}, these functions form a much more restrictive class than the ones of Theorem~\ref{cor:mixingrate}.
To state these results,  we shall introduce two classes  of
sets $\mathcal F$ (resp. $\widetilde{\mathcal F}$) that will correspond to the set of measurable sets of configurations in $\mathcal M$ (resp.  $\widetilde{\mathcal M}$) with previous collision in some fixed subset of $M$ (resp. some fixed cell of $\widetilde M$), at some time in a fixed bounded time interval.
\begin{defn}\label{defsetF}
Let $\mathcal F$ be the  class of measurable subsets $A$ of $\mathcal M$ of the form $A=\phi_I(A_0)=\{\phi_u(x),\, x\in A_0,\, u\in I\}$ that are represented in $\widehat{\mathcal M}$ 
by $A_0\times I\subset\widehat{\mathcal M}$ (implying that $I\subset [0,\inf_{A_0}\tau)$), with $A_0\subset M$ a measurable set satisfying $  \mu(\partial A_0)=0$ and with $I$ a bounded interval.\\
Let $\widetilde{\mathcal F}$ be the set of subsets of
$\widetilde{\mathcal M}$ corresponding to $A_0\times I\times \{\ell\}\subset\widehat{\mathcal M}\times\mathbb Z^d$, with
$\phi_I(A_0)\in\mathcal F$ and $\ell\in\mathbb Z^d$, that is
sets of the form
\[\left\{\Phi_u(p_{d,0}^{-1}(q)+\ell,\vec v)\, :\, (q,\vec v)\in A_0, u\in I\right\}\, \mbox{ with }\phi_{I}(A_0)\in\mathcal F,\ \ell\in\mathbb Z^{d}\, .
\] 
\end{defn}
We state now a MLLT for $\kappa_{N_t}$ defined on $\mathcal M$
by
\[
\forall (x,u)\in\widehat{\mathcal M},\ \forall t\in[0,+\infty),\quad\kappa_{N_t}(\phi_u(x)):=\kappa_{N_t(\phi_u(x))}(x)=\kappa_{N_{t+u}(x)}(x)\, .
\]
This observable $\kappa_{N_t}$ will be understood as the cell change during the time interval $(0,t]$.
\begin{prop}\label{prop:MLLTkappa0}
Let $A,B\in\mathcal F$ and let $K$ be a bounded subset of $\mathcal D_d$ 
with 
$Leb(\partial K)=0$. 
Then
\begin{align}
\forall \ell\in\mathbb Z^d,\quad (t\log t)^{\frac d2}\nu&\left(A\cap\{\phi_t \in B,\, 
\kappa_{N_t}=\ell\}\right)\sim	\widetilde g_d\left(0\right)\nu(A)\nu(B)\, ,\label{eq:mixingflow000}
\end{align}
as $t\to\infty$, where $\widetilde g_d$ is the density of the $d$-dimensional Gaussian distribution $\mathcal N(0,\Sigma)$
appearing in~\eqref{eq:Sigma}.
\end{prop}
This result is contained in a more general MLLT stated in Proposition~\ref{LLTkappaNt} (applied with $w_t=\ell$, $w=0$, $K=\{0\}$).
An immediate consequence of Proposition~\ref{prop:MLLTkappa0} is the following light version
of Theorem~\ref{cor:mixingrate} for compactly supported observables in the 'extended suspension' 
$\widehat{\mathcal M}\times\Z^d$; in particular, the supports of these functions only 
contain configurations that have hit or will hit an obstacle in a bounded time.
\begin{cor}\label{cor:easymix}
	Let $n\in\mathbb N$, setting
	\[
	E_{\pm n}=\left\{\Phi_{\pm u}(q+\ell,\vec v)\in \widetilde{\mathcal M}\, :\, q\in\bigcup_{i=1}^I\partial O_i,\  u\in[0,n],\ \ell\in\mathbb Z^d,\ |\ell|\le |n|\right\}\, ,
	\]
then, for any $f,g:\widetilde{\mathcal M}\rightarrow \mathbb R$ that are $\mu$-a.e. continuous functions and supported respectively in $E_{-n}$ and in $E_n$,
	\begin{equation}\label{mixKrick}
	\int_{\widetilde{\mathcal M}}f.g\circ 
	\Phi_t
	\, d\widetilde\nu\sim \frac {\int_{\widetilde{\mathcal M}}f\, d\widetilde\nu \int_{\widetilde{\mathcal M}}g\, d\widetilde\nu}{(2\pi t\log t\det(\Sigma))^{\frac d2}}\, ,
	\end{equation}
as $t\rightarrow +\infty$.
\end{cor}
\begin{proof}
	Let $A,B$ be two sets belonging to $\widetilde{\mathcal F}$ corresponding to respectively $A_0\times I\times\{\ell_0\}$ and $B_0\times J\times\{\ell'_0\}$
	in $\widehat{\mathcal M}\times\mathbb Z^d$. We observe that
\[
	\widetilde\nu\left(A\cap\Phi_{-t}(B)\right)=\nu(\phi_I(A_0)\cap \{\phi_t\in \phi_J(B_0),\, \kappa_{N_t}=\ell'_0-\ell_0\})\, .
\]
	Thus, it follows from~\eqref{eq:mixingflow000} that
\[
(t\log t)^{\frac d2}\widetilde\nu\left(A\cap\Phi_{-t}(B)\right)\sim \widetilde g_d(0)\nu(\phi_I(A_0))\nu(\phi_J(A_0))=\widetilde g_d(0)\widetilde\nu(A)\widetilde\nu(B)\, .
\]	
This result extends directly to any finite union  $A,B\subset\widetilde{\mathcal M}$ of sets belonging to  $\widetilde{\mathcal F}$, implying 
	Krickeberg mixing as defined in~\cite{Krickeberg67}
	for the family of sets $(E_n)_{n\ge 1}$.  
	It follows from~\cite[Section 2]{Krickeberg67} (see also,~\cite[Section 9]{MT17} for the Krickeberg argument written for suspension flows)
	that~\eqref{mixKrick} holds true for any $f,g$ supported in some
	$E_n$ and $\mu$-almost everywhere continuous.
	To end the proof of Corollary~\ref{cor:easymix}, we notice that, $\Phi$ being invertible, if $f$ is supported in $E_{-n}$, then $f\circ\Phi_{-n}$ is supported on $E_n$ and we finally conclude with the use the following formula
	\[
	\int_{\widehat{\mathcal M}}f.g\circ 
	\Phi_t\, d\widetilde\nu=\int_{\widehat{\mathcal M}}f\circ \Phi_{-n}.g\circ 
	\Phi_{t-n}\, d\widetilde\nu\, ,
	\]
	since $(t-n)\log(t-n)\sim t\log t$.
\end{proof}

The mixing result in Corollary~\ref{cor:easymix} can be rephrased in terms of the vague convergence of
the family of $\mu_t$ to $\mu\otimes\mu$ where $\mu_t$ is  the measure on $(\widehat{\mathcal M})^2$
defined by $\mu_t(A'\times B')=\mu(A'\cap \Phi_{-t}B')$ for  $A',B'\in\widetilde{\mathcal F}$ (this is a consequence of the Portmanteau theorem  as in, for instance,~\cite{Resnick}),
and the same applies for Theorem~\ref{cor:mixingrate}.

\begin{rmk}
	
	We remark that mixing 
	of the type of Corollary~\ref{cor:easymix} has been previously obtained
	in~\cite{T22} for $\Z$-extensions of Gibbs Markov semiflows with roof and displacement functions in the domain of a nonstandard CLT. The method of proof in~\cite{T22} is very 
	different; in particular, it does not go via a MLLT for the base map.
\end{rmk}

\subsection{MLLT for the infinite horizon Sinai flow}
\label{subsec:mllt}
In this section we state the MLLT for a natural cocycle of
the Sinai billiard flow, which corresponds to the displacement.

\subsubsection*{Free flight}
Due to the $\mathbb Z^d$-periodicity, the \emph{free flight}
$\widetilde V: \widetilde M\rightarrow \mathcal D_d$ which is defined by 
\begin{equation}\label{deftildeV}
\forall (q,\vec v)\in \widetilde M,\quad 
\widetilde T(q,\vec v)=(\widetilde q,\vec v_1)\quad \Rightarrow\quad 
\widetilde V(q,\vec v)=\widetilde q-q
\end{equation}
goes to the quotient by $\mathbb Z^d$, i.e. there exists $V:M\rightarrow \mathcal D_d$ such that 
\begin{equation}\label{deftildeV0}
\widetilde V(q,\vec v)=V\left(p_d(q),\vec v\right)\, .
\end{equation}
When $d=2$, this quantity is related to the flight time $\tau$ via the following identity
\begin{equation}\label{linktauV}
\mbox{if }d=2\, ,\quad \tau=|V|\, .
\end{equation}
Let us show that the free flight $V$ is cohomologous to the cell change $\kappa$.
It follows from \eqref{Zdextension},~\eqref{deftildeV} and~\eqref{deftildeV0} that, for all $x=(q,\vec v)\in M$, if $T(q,\vec v)=(q',\vec v')$, then
\begin{align}
\nonumber V(x)&= \widetilde V\left(p_{d,0}^{-1}(q),\vec v\right)\\
&=p_{d,0}^{-1}(q')+\kappa(q,\vec v)-p_{d,0}^{-1}(q)=\kappa(x)+H_0(T(x))-H_0(x)\, ,\label{coboundPsi}
\end{align}
with $H_0(q,\vec v)=p_{d,0}^{-1}(q)$.
Proceeding as for $\widetilde W_t$ in Section~\ref{sec:intro}, if $d=1$ we set $V':\mathcal M\rightarrow \mathbb R$ for the first coordinate of $V$, and if $d=2$,
$V$ takes its values in $\mathbb R^2$, we then just set $V'=V$. The following nonstandard CLT
 was proved  in~\cite{SV07}
for $V'$:
\begin{equation}\label{eq:cltV}
a_n^{-1}\sum_{j=0}^{n-1} V'\circ T^j\implies \mathcal N(0,\Sigma_0)\, ,
\end{equation}
where $a_n=\sqrt{n\log n}$ and where $\Sigma_0\in\R^{d\times d}$ is a positive-definite symmetric $d$-dimensional matrix (see~\eqref{defSigma0dim2} and~\eqref{defSigma0dim1}) for precise formulas).
An important ingredient of~\cite{SV07} is that $V$
lies in the domain of a nonstandard CLT; that is, there exists
$c > 0$ such that 
\begin{equation}\label{tailV}
\mu(|V | > t)\sim ct^{-2}\, .
\end{equation}

\subsubsection*{Displacement function $W_t$}
We have already defined in Section~\ref{sec:intro} the displacement function $\widetilde W_t:\widetilde{\mathcal M}\rightarrow\mathcal D_d$ and $\widetilde W'_t:\widetilde{\mathcal M}\rightarrow\mathbb R^d$ its projection on $\mathbb R^d$. 
Due to the $\mathbb Z^d$-periodicity of our model, both displacement functions
go the quotient by $\mathbb Z^d$, i.e. there exists 
$W_t:\mathcal M\rightarrow \mathcal D_d$ and $W'_t:\mathcal M\rightarrow\mathbb R^d$ such that
\[
\forall (q,\vec v)\in\widetilde{\mathcal M},\quad
\widetilde W_t(q,\vec v)=W_t(p_d(q),\vec v)\quad\mbox{and}\quad
\widetilde W'_t(q,\vec v)=W'_t(p_d(q),\vec v)\, .
\]
Observe that $W_t$ is a cocycle:
\begin{equation}\label{Wt_cocycle}
\forall x\in \mathcal M,\ \forall t,s\ge 0,\quad  W_{t+s}(x)=W_s(x)+W_t(\phi_s(x))\, 
\end{equation}
and that
\begin{equation}\label{eq:xi}
\forall x=(q,\vec v)\in M, \quad V(x)=W_{\tau}(x):=W_{\tau(x)}(x)\, .
\end{equation}
Thus the nonstandard CLT for $V'$ stated in~\eqref{eq:cltV} implies a nonstandard CLT for $W'_t$  via the relation~\eqref{eq:xi} 
together with the classical scheme of lifting 
limit theorems from the induced map to the original system (map or flow) ~\cite{MT04, Gouezel07}. This leads to the following result where we use the notation $a_t:=\sqrt{t\log t}$.
\begin{prop}[CLT~\cite{SV07}] \label{prop:cltflow}
As $t\rightarrow+\infty$, $	a_t^{-1}W'_t\implies \mathcal N(0,\Sigma)$
	where $\Sigma\in\R^{d\times d}$, $\Sigma=\Sigma_0/\mu(\tau)^{1/2}$
	with $\Sigma_0$ as in~\eqref{eq:cltV}.
\end{prop}
Let $v_0:\mathcal M\to\mathbb S^1$ be the velocity map which is given by $v_0(q,\vec v)=\vec v$. Note that
\begin{equation}\label{eq:Wt}
\mbox{if }d=2,\quad W_t:=\int_0^tv_0\circ\phi_s\, ds \, .
\end{equation}
If $d=1$,
then $W_t$ is the equivalent class in $\mathcal D_1$ (that is, $W_t$ is the canonical projection) of the  ergodic integral 
$\int_0^t v_0\circ\phi_s\,ds$.

\subsubsection*{MLLT for the displacement function}
Let us see that, due to~\eqref{eq:xi},  the coboundary equation~\eqref{coboundPsi} for $V-\kappa$ leads to a similar equation involving $W$.
We consider the function $H_1:\mathcal M
\rightarrow \mathcal D_d$ mapping $\mathbf x\in\mathcal M$
to the position of its representant in $\mathcal D_d$ with previous collision in $\mathcal C_0$, that is
\begin{equation}\label{def:H1}
\forall ((q,\vec v),u)\in\widehat{\mathcal M},\quad 
H_1(\phi_u(q,\vec v))=\mathfrak p\left(\Phi_u\left(p_{d,0}^{-1}(q),\vec v\right)\right)
=p_{d,0}^{-1}(q)+W_u(q,\vec v)\, ,
\end{equation}
where $\mathfrak p:\widetilde{\mathcal M}\rightarrow\Omega_d$ is the natural projection.
In other words, if $d=2$, then 
\begin{equation}\label{eq:H1}
H_1(\phi_u(q,\vec v))=H_0(q,\vec v)+W_u(q,\vec v)=p_{d,0}^{-1}(q)+u\vec v\, ;
\end{equation} 
if $d=1$, $H_1(\phi_u(q,\vec v))$ is the class of $p_{d,0}^{-1}(q)+u\vec v$ in $\mathcal D_1$. 
Recall that we set $N_t:\mathcal M\rightarrow \mathbb N_0$ for the \emph{collisions number} in the time interval $(0,t]$ (see in particular \eqref{eq:lapn}). 
The above defined function $H_1$ satisfies the following important property:
\begin{equation}\label{eq:Wcobound}
\forall (x=(q,\vec v),u)\in \widehat {\mathcal M}\, ,\quad
W_t(\phi_u(x))=\kappa_{N_{t+u}(x)}(x)+H_1(\phi_t(\phi_u(x)))-H_1(\phi_u(x))\, .
\end{equation}
Indeed, setting $N:=N_{t+u}(x)$, we notice that \begin{equation}\label{phitphiu}
\phi_t(\phi_u(x))=\phi_{u'}(q', \vec w),\quad \mbox{with }u':=u+t-\tau_{N}(x),\ (q',\vec w):=T^{N}(x)=\phi_{\tau_N}(x)\, ,
\end{equation}
and $(T^N(x),u')$ is in $\widehat{\mathcal M}$.
Therefore, it follows from~\eqref{Wt_cocycle} and~\eqref{eq:xi} that
\begin{align*}
W_t(\phi_u(x))&=W_{t+u}(x)-W_u(x)=W_{u'}(\phi_{\tau_N}(x))+W_{\tau_N}(x)-W_u(x)\\
&=W_{u'}(T^N(x))+W_{\tau_N}(x)-W_u(x)\\
&=W_{u'}(T^N(x))+V_{N}(x)-W_u(x)\, .
\end{align*}
Finally, using~\eqref{coboundPsi} and~\eqref{eq:H1}, we obtain that
\begin{align*}
W_t(\phi_u(x))
&=W_{u'}(T^N(x))+\kappa_{N}(x)+H_0(T^N(x))-H_0(x)-W_u(x)\\
&=H_1(\phi_t(\phi_u(x)))+\kappa_{N}(x)-H_1(\phi_u(x))\, ,
\end{align*}
as announced.
\\[3mm]
Recall that $a_t = \sqrt{t \log t}$.

 \begin{thm}[MLLT for $W_t$]
 \label{thm:main}
 Let $A,B\in\mathcal F$ and let $K$ be a bounded subset of $\mathcal D_d$ 
   with 
 $Leb(\partial K)=0$. Let $w\in\R^d$ and let $w_t\in\mathbb R^d$
 such that $\lim_{t\rightarrow +\infty}w_t/a_t=w$. 
 Then\footnote{Again, in this formula, if $d=1$, the notation $w_t+K$ means
	$(w_t,0)+K$ and $\mathbb Z^d$ means $\mathbb Z\times\{0\}$.}
\begin{align}\nonumber
a_t^d\nu&\left(A\cap\{\phi_t \in B,\,
W_t\in w
_t+K\}\right)\\
&\sim	\widetilde g_d\left(w
\right)\int_{A\times B}\#((K+w_t+H_1(\mathbf{x})-H_1(\mathbf{y}))\cap\mathbb Z^d)\, d\nu(\mathbf{x})\, d\nu(\mathbf{y})\, ,\label{eq:mixingflow}
\end{align}
as $t\to\infty$, where $\widetilde g_d$ is the density of the $d$-dimensional Gaussian distribution $\mathcal N(0,\Sigma)$
appearing in Proposition~\ref{prop:cltflow} and where $H_1$
is the function that has been defined in~\eqref{def:H1}.
\end{thm}
The proof of Theorem~\ref{thm:main} is provided in Section~\ref{sec:proofMLLT}, and will appear as a consequence of an analogous result (Proposition~\ref{LLTkappaNt}) stated for $\kappa_{N_t}$ instead of $W_t$.

\section{Statements of the Joint LLT with error term and the joint LLD for the billiard map}\label{sec:MLLT+LLD}
Let $d\in\{0,1,2\}$. 
In this section we state the main technical results that will be used in the proofs of Theorem~\ref{thm:main} (MLLT for the Sinai flow) and Theorem~\ref{cor:mixingrate} (mixing for the Lorentz gas), including those used in the proof of  the key tightness-type result 
Theorem~\ref{lem:tight} (stated in Section~\ref{sec:proofmixing}).
We are interested in joint MLLT  and LLD for the pair
 \[
\widehat{\Psi}=\widehat{\Psi}^{(d)}:=(\kappa,\widetilde\tau):M\to\mathcal D_d\times\mathbb R\, ,\quad\mbox{with }\widetilde\tau:=\tau-\mu(\tau)\, ,\quad\mbox{if }d\in\{1,2\}
\]
or for
\[
\widehat \Psi=\widehat\Psi^{(0)}:=\widetilde\tau\, ,\quad\mbox{if }d=0\, .
\]
Note that $\int_M \widehat\Psi\,d\mu{=0}$. When $d\in\{1,2\}$, our joint limit results are related to the fact that
that the sums of $(\kappa\circ T^k,\widetilde\tau\circ T^k)_k$ satisfies a CLT with nonstandard normalization $a_n^{-1}$. 
In particular, as clarified in the proof of Sublemma~\ref{sub:asl} below, the vector $\widehat\Psi$ is so that $\mu(|\widehat\Psi|>t)\sim ct^{-2}$.
As usual, we write $\widehat\Psi_n=\sum_{j=0}^{n-1}\widehat\Psi\circ T^j$
and similarly for $V_n,\,\tau_n, \widetilde\tau_n$.
We start with a nondegenerate CLT with nonstandard scaling for $\widehat\Psi_n$.

\begin{lemma}[Joint CLT for the billiard map] \label{lem:clt}
 $a_n^{-1}\widehat\Psi_n\implies \mathcal N(0,\Sigma_{d+1})$ as $n\to\infty$, 
where $\Sigma_{d+1}\in\R^{(d+1)\times(d+1)}$ is positive-definite (see~\eqref{Sigmad+1} for an explicit formula).
\end{lemma}
This result is proved in Section~\ref{sec:proofjCLT} by adapting the proof of the CLT for $V_n$ established in~\cite{SV07} via~\cite{BalintGouezel06}, writing $\widehat\Psi_n$ as a function of the two dimensional cell change plus a Lipschitz function.

Let us state a MLLT for $\widehat\Psi_n$ with a uniform error term.
We write $\Lambda_{d+1}$ for the Haar measure on $\mathbb Z^d\times\mathbb R$ given by the product of the counting measure on $\mathbb Z^d$ and of the Lebesgue measure on $\mathbb R$.
\begin{lemma}[Joint MLLT for the billiard map]
Let $p>2$ and $R>0$. 
	\label{lem:jointllt0}
	We take $\mathfrak a_n$ such that $\mathfrak a_n^2 = 2n\log (\mathfrak a_n) \sim n \log n =  a_n^2$.
Assume $G,H:M\rightarrow\mathbb R$ are two bounded dynamically H\"older continuous functions and that $h:\mathbb Z^d\times\mathbb R\rightarrow\mathbb R$ is integrable with compactly supported
Lipschitz Fourier transform $\widehat h:\mathbb T^d\times \mathbb R\rightarrow\mathbb C$.
There exists $a_0>0$ (depending only on $p$ and on the H\"older exponent of $G$ and $H$) such that, for all $k_n<n/4$, 
\begin{align*}
	\mathbb E_{\mu}&\left[G. 
	h(\widehat\Psi_n-L
	).H\circ T^n\right]\\
	&
	= \mathfrak a_n^{-d-1}  \mathbb E_{\mu}\left[H\right]\mathbb E_{\mu}\left[G\right]\left(g_{d+1}\left(\frac L{\mathfrak a_n}
	\right)\int_{\mathbb Z^d\times\mathbb R}h\, d\Lambda_{d+1} +\mathcal O\left((\log n)^{-1}+\frac{k_n}n
	\right)\right)\\
	&\quad+\mathcal O\left(
	e^{-a_0k_n}
	\Vert G\Vert_{Holder} \Vert H\Vert_{Holder}
	+
	a_n^{-d-2}
	(k_n\Vert G\Vert_{L^1}\Vert H\Vert_{L^{p}}+
	\Vert H\Vert_{L^1}
	 \Vert \widehat\Psi_{2k_n}.G\Vert_{L^1(\mu)})\right),
\end{align*}
uniformly in $L
\in\mathbb Z^d\times\mathbb R$, in $(n,k_n)$ as above, and in $h$ such that  $\text{supp}(\widehat h)\subset B(0,R)$ and $\Vert\widehat h\Vert_{Lipschitz}\le R$,  where $g_{d+1}$ is the density for the $(d+1)$-dimensional Gaussian in Lemma~\ref{lem:clt}.\\
\end{lemma}
This result is proved in Section~\ref{sec:proofjMLLT}. The scheme of its proof follows the one of the MLLT established in~\cite[Theorem 2.2]{PeneTerhesiu21}
but there are at least two main differences.
First, here we need to obtain a Joint MLLT, which is different from
~\cite[Theorem 2.2]{PeneTerhesiu21} which was a MLLT with error terms for the cell change function. The new ingredients needed to deal with the Joint MLLT (with error) are summarized in Section~\ref{sec:proofjCLT}.
Second, the error term obtained in~\cite{PeneTerhesiu21}  is not sharp enough for the present purposes. To establish Lemma~\ref{lem:jointllt0}, we need to be much more careful with the error terms all throughout the proof and this requires 
 entirely new estimates, all obtained in Section~\ref{sec:proofjMLLT}.


A consequence of Lemma~\ref{lem:jointllt0} is 
\begin{cor}\label{coro:llt0cor}
Under the assumptions of Lemma~\ref{lem:jointllt0}, then
\begin{align*}
\mathbb E_{\mu}&\left[G. 
h(\widehat\Psi_n-L
).H\circ T^n\right]-\mathbb E_\mu[G]\mathbb E_\mu[H] \mathbb E_\mu[h(\widehat\Psi_n-L
)]\\
	&=\mathcal O\left(
e^{-a_0k_n}
\Vert G\Vert_{Holder} \Vert H\Vert_{Holder}
+
a_n^{-d-2}
(k_n\Vert G\Vert_{L^1}\Vert H\Vert_{L^{p}}+
\Vert H\Vert_{L^1}
\Vert \widehat\Psi_{2k_n}.G\Vert_{L^1(\mu)})
\right)\, .
\end{align*}
\end{cor}
\begin{proof}
We observe that
\begin{align*}
\mathbb E_{\mu}&\left[G. 
h(\widehat\Psi_n-L
).H\circ T^n\right]-\mathbb E_\mu[G]\mathbb E_\mu[H] \mathbb E_\mu\left[h(\widehat\Psi_n-L
)\right]\\
&=\mathbb E_{\mu}[(G-\mathbb E_\mu[G]). 
h(\widehat\Psi_n-L
).H\circ T^n]
)]+\mathbb E_{\mu}[G]\, \mathbb E_{\mu}[ 
h(\widehat\Psi_n-L
).(H-\mathbb E_{\mu}[H])\circ T^n]
)]\, ,
\end{align*}
and we apply Lemma~\ref{lem:jointllt0} to the two terms of the right hand side of the above equality,
since $\Vert\widehat\Psi_{2k_n}\Vert_{L^1(\mu)}=\mathcal O(k_n)$,  the function $\widehat\Psi$ being integrable.
\end{proof}

The following MLLT for $\widehat\Psi_n$ will be shown (in Section~\ref{sec:proofjMLLT}) from Lemma~\ref{lem:jointllt0}. 
\begin{lemma}
\label{lem:jointllt}
Let $A_0,B_0\subset M$ be measurable sets such that $\mu(\partial A_0)=\mu(\partial B_0)=0$. Let $K\subset
\mathbb Z^d\times\mathbb R
$ be a bounded set with $\Lambda_{d+1}(\partial K)=0$ (boundary in $\mathbb Z^d\times\mathbb R$).
Then, for any $L>0$,
\begin{equation}\label{eq:jointLLDunif}
a_n^{d+1}\mu\left(A_0\cap T^{-n}(B_0)\cap\{\widehat\Psi_n(x)\in z +K\}\right)- g_{d+1}\left(\frac{z}{a_n}\right)\mu(A_0)\mu(B_0)\Lambda_{d+1}(K)
\end{equation}
converges to 0 uniformly in $z\in\mathbb Z^d\times\mathbb R\, :\, |z|\le L a_n$, as $n\rightarrow +\infty$.
\end{lemma}

The joint LLD we shall need is
\begin{lemma}[Joint LLD for the billiard map] \label{lem:lld}
Let $U\subset\R^{d+1}$ be an open ball.
Then
\begin{equation}\label{eq:LLD} 
\mu(\widehat\Psi_n\in z+U)\ll \frac{n}{a_n^{d+1}} \frac{\log(2+|z|)}{1+|z|^2}
\end{equation} 
uniformly in $ n\ge1$ and $z\in\R^{d+1}$. 
\end{lemma}
The proof of this result, given in Section~\ref{sec:pfkey} is a more or less obvious adaptation of the proof of~\cite{MPT} with the additional complication
that $\widehat\Psi_n,\tau_n$ are non-lattice valued. As already mentioned in the introduction,
this is the only result of the current paper that does not require any novelty, but just a straighforward adaptation with some care.

 \section{Proof of MLLT for the Sinai flow (Proposition~\ref{prop:MLLTkappa0} and  Theorem~\ref{thm:main})}\label{sec:proofMLLT}
In this section we assume $d=1$ or $d=2$. 
In this section, we complete the proof of Theorem~\ref{thm:main} by stating and proving the following result (assuming, for the moment, the statement of the results
stated in Section~\ref{sec:MLLT+LLD}). 
\begin{prop}\label{LLTkappaNt}
Let $A,B\in\mathcal F$.
Let $K$ be a bounded subset of $\mathbb R^{d}$, let $w\in\mathbb R^{d}$ and 
$w_t\in
\mathbb R^{d}$ be such that $\lim_{t\rightarrow +\infty}w_t/a_t= w$. Then
\begin{equation}\label{LLTFLOWkappa}
a_t^{d}\ \nu\left(A\cap\{\phi_t \in B,\,
\kappa_{N_t}\in w
_t+K\}\right)
\sim  \widetilde g_d\left(w
\right)\nu(A)\nu(B)\#((K+w_t)\cap\mathbb Z^d)\, ,
\end{equation}
where $\widetilde g_d$ is the density of the Gaussian limit of
Proposition~\ref{prop:cltflow} and with the following natural convention
\[
\forall (x,u)\in\widehat{\mathcal M},\quad (\kappa_{N_t})(\phi_u(x)):=\kappa_{N_{t+u}(x)}(x)\, .
\]
\end{prop}

Using Proposition~\ref{LLTkappaNt} we complete

\begin{pfof}{Theorem~\ref{thm:main}}
We first recall that we need to show that
\begin{align}
a_t^{d}\nu&\left( A\cap\{\phi_t \in B,\, 
W_t\in w
_t+K\}\right)\sim	\widetilde g_d\left(w
\right)
\mathcal I(A\times B,w_t)\, ,\label{LLTFLOWV}
\end{align}
where $K$ is so that $Leb(\partial K)=0$, where $H_1$ is as in~\eqref{def:H1} and where we set
\[
\mathcal I(A\times B,w_t):=
\int_{A\times B}\#((K+w_t+H_1(\mathbf{x})-H_1(\mathbf{y}))\cap\mathbb Z^d)\, d\nu(\mathbf{x})\, d\nu(\mathbf{y})\, .
\]
For any positive integer $m$, we partition $A$ (resp. $B$) in a finite number of atoms $A_{k,m}\in\mathcal F$ and  $B_{k,m}\in\mathcal F$ of diameter at most $1/m$, and consider the sets
\[
K_{i,j,m}^-:=\left\{z\in \mathbb R^d\, :\, \forall (x,y)\in A_{i,m}\times B_{j,m},\ z+H_1(y)-H_1(x)\in K\right\}
\]
and 
\[
K_{i,j,m}^+:=\left\{ z\in\mathbb R^d\, :\, \exists (x,y)\in A_{i,m}\times B_{j,m},\ z+H_1(y)-H_1(x)\in K\right\}
\, .
\]
Note that $H_1$ is Lipschitz continuous
in $(x,u)$ and bounded on $A$ and $B$.
It follows from~\eqref{eq:Wcobound} that
\begin{align}
\nu&\left( A\cap\{\phi_t \in B,\, 
W_t\in w
_t+K\}\right)\nonumber\\
&= \nu\left(
A\cap\{
\phi_t 
\in B,\, 
\kappa_{N_t}+H_1\circ\phi_t-H_1\in w
_t+K\right)\nonumber\\
&\le  \sum_{i,j}\nu\left( A_{i,m}\cap\{\phi_t\in B_{j,m},\, 
\kappa_{N_t}\in w_t+ K_{i,j,m}^+\}\right)\label{MAJO1}
\end{align}
and, analogously,
\begin{align}
\nu&\left( A\cap\{\phi_t \in B,\, 
W_t\in w
_t+K\}\right)\ge \sum_{i,j} \nu\left( A_{i,m}\cap\{\phi_t \in B_{j,m},\, 
\kappa_{N_t}\in w_t+K_{i,j,m}^-\}\right)\, .\label{MAJO2}
\end{align}
By Proposition~\ref{LLTkappaNt},
\begin{align}\nonumber
\nu&\left( A_{i,m}\cap\{ \phi_t \in B_{j,m},\, 
\kappa_{N_t}\in w_t+ K_{i,j,m}^\pm\}\right)\\
&\sim
a_t^{-d} \widetilde g_{d}(w)
\#((w_t+K_{i,j,m}^\pm)\cap\mathbb Z^d)\nu(A_{i,m})\nu(B_{j,m})\, .\label{MAJO3}
\end{align}
Furthermore,
\begin{align}
\#((w_t+K_{i,j,m}^-)\cap\mathbb Z^d)\nu(A_{i,m})\nu(B_{j,m})&\le
\mathcal I(A_{i,m}\times B_{j,m},w_t)\label{MAJO4}\\
\mathcal I(A_{i,m}\times B_{j,m},w_t)&\le \#((w_t+K_{i,j,m}^+)\cap\mathbb Z^d)\nu(A_{i,m})\nu(B_{j,m})\, .\label{MAJO5}
\end{align}
Let $(x,y,z)\in\bigcup_{i,j}\left(A_{i,m}\times  B_{j,m}\times\left(
\mathbb Z^d\cap(w_t+(K_{i,j,m}^+\setminus K_{i,j,m}^-))\right)\right)$. 
Then there exist $\mathbf{x},\mathbf{x'}\in A_{i,m}$ and $\mathbf{y},\mathbf{y'}\in B_{j,m}$ such that $z-w_t+H_1(\mathbf{y})-H_1(\mathbf{x})\in K$ but $z-w_t+H_1(\mathbf{y'})-H_1(\mathbf{x'})\not\in K$.
Recall that $H_1$ is Lipschitz. 
Thus the above conditions means that $z\in\mathbb Z^d$ and that $(x,y)$ is at distance smaller than $1/m$ of $\mathcal E_{w_t-z}:=\{(\mathbf x,\mathbf y)\ :\, H_1(\mathbf{y})-H_1(\mathbf{x})\in
   w_t-z+\partial K\}$.
This $z$ should be one of the elements of $\mathbb Z^d$ contained in the ball of radius $\sup_A| H_1|+\sup_B| H_1|+\sup_{s\in K}|s|$ around $w_t$.
But, for each such $z$, the measure of this neighbourhood of $\mathcal E_{w_t-z}$ satisfies 
\begin{align*}
\nu^{\otimes 2}\left((\mathcal E_{w_t-z})^{[1/ m]} \right)
&\le \sup_{|u|\le 3\sup_{A\cup B}|H_1|+\sup_{s\in K}|s|} \nu\left( H_1^{-1}\left((u+\partial K)\right)^{[1/m]} \right)\\
&\le \sup_{|u|\le 3\sup_{A\cup B}|H_1|+\sup_{s\in K}|s|} Leb\left(\left((u+\partial K)\right)^{[1/m]} \right)\, ,
\end{align*}
which converges to $0$ as $m\rightarrow +\infty$ since $Leb(\partial K)=0$.
Since the number of possible $z$ is uniformly bounded, we have proved
that 
\begin{equation}\label{MAJO6}
\lim_{m\rightarrow +\infty}\sup_t\sum_{i,j}
\#((w_t+(K_{i,j,m}^+\setminus K_{i,j,m}^-))\cap\mathbb Z^d)\nu(A_{i,m})\nu(B_{j,m})=0\, .
\end{equation}
The desired conclusion~\eqref{LLTFLOWV} follows from~\eqref{MAJO1},~\eqref{MAJO2},~\eqref{MAJO3},~\eqref{MAJO4},~\eqref{MAJO5} and~\eqref{MAJO6}.
\end{pfof}

\subsection{Proof of Proposition~\ref{LLTkappaNt}}

Recall that $A=\phi_I(A_0)$ and $B=\phi_J(B_0)$ with $A_0,B_0\subset M$
such that $\mu(\partial A_0)=\mu(\partial B_0)=0$ and $I,J\subset\mathbb R$ two bounded intervals.
We start by proving the lemma for $w_t\in\mathbb Z^d$ and $K=\{0\}$. 
We follow a decomposition somewhat similar to~\cite[Proof of Lemma 4.3]{AN17}, see also~\cite[Proof of Theorem 3.1]{DN20} and~\cite[Proof of Theorem 1]{AT20}, with the obvious difference that one needs to figure out how to exploit
the Joint LLT~\ref{lem:jointllt} and the Joint LLD~\ref{lem:lld}.

Writing $\mathbf{x}=\phi_u(x)$ with $(x,u)\in\widehat{\mathcal M}$, we use the product structure of the measure $\widehat \nu$ and partition the set considering the different values taken by $N_t$:
\begin{align}\label{sumQn}
  \nu\left( A\cap\{\phi_t \in B,\, 
  \kappa_{N_t}=w_t\}
\right)
 = \frac 1{\mu(\tau)}\sum_{n\ge 0} \int_I Q_n(t,u)\, du\, ,
 \end{align}
where
\begin{align*}
Q_n(t,u)&:=\mu\left(A_0\cap T^{-n} B_0\cap\{
\kappa_n
=w_t,\,
\widetilde\tau_n\in u+t-n{\mu(\tau)}-J\}\right)\\
&=\mu\left(A_0\cap T^{-n}(B_0)\cap\left\{
\widehat\Psi_n  \in (w
_t,
t-n{\mu(\tau)})+
\{0\}\times J_u\right\}\right)\, ,
\end{align*}
with $J_u=u-J$, recalling that $\widehat\Psi_n=(\kappa_n,\widetilde\tau_n)$. 
For $L$ large, we split the sum as
\[
  \nu\left(A\cap\{\phi_t \in B,\, \kappa_{N_t}=w_t\}
  \right)= S_1(t,L)+S_2(t,L)\, ,
\]
where
\[
S_1(t,L):=\frac 1{\mu(\tau)}\sum_{n\,:\,|n-t/{\mu(\tau)}|\le La_t}\int_I Q_n(t,u)\,du\, ,
\]
\[
S_2(t,L):=\frac 1{\mu(\tau)}\sum_{n\,:\,|n-t/{\mu(\tau)}|> La_t}\int_I Q_n(t,u)\,du\, .
\]

The main ingredient needed for the Proof of Proposition~\ref{LLTkappaNt} is
\begin{lemma}
	\label{lem:S1S2}
	\begin{itemize}
		\item[(a)] $\lim_{L\to\infty}\lim_{t\to\infty} a_t^d S_1(t,L)=
		\widetilde g_d(w)\nu(A)\nu(B)$,
		\item[(b)] $\lim_{L\to\infty}\limsup_{t\to\infty} a_t^d S_2(t,L)=0$.
	\end{itemize}
\end{lemma}
The proof of Lemma~\ref{lem:S1S2} is provided in the paragraph~\ref{sub:llllllll}
below.
Equipped with the statement of Lemma~\ref{lem:S1S2} we can complete

\begin{pfof}{Proposition~\ref{LLTkappaNt}}
Note that~\eqref{LLTFLOWkappa} for $w_t\in\mathbb Z^d$ and $K=\{0\}$ follows directly from Lemma~\ref{lem:S1S2} due to~\eqref{sumQn}. It remains to go from this special case to the general case.
Let $w_t\in\mathbb R^d$
and  consider a bounded subset $K$ of $\mathbb R^d$. Then $(w_t+K)\cap\mathbb Z^d$ contains at most $(\diam(K)+1)^2$ integers, we can label them $w_{t,i}$
for $i=1,...,(\diam(K)+1)^2$
 (ordering them e.g. by their first coordinate, and then by their second, and completing if necessary by the successors of the last one for this order). Then
\begin{align*}
a_t^d\ \nu\left( A\cap\{\phi_t \in B,\,
\kappa_{N_t}\in w
_t+K\}\right)&=a_t^d\!\!\! \sum_{i=1}^{\#((w_t+K)\cap\mathbb Z^d)}\nu\left( A\cap\{ \phi_t \in B,\,
\kappa_{N_t}=w_{t,i}\}\right)\\
&\sim \#((w_t+K)\cap\mathbb Z^d) \widetilde g_d(w)\nu (A)\nu(B)\, ,
\end{align*}
applying~\eqref{LLTFLOWkappa} with $K=\{0\}$ for each sequence $(w_{t,i})_t$.
Indeed, since $K$ is a bounded set,
$\lim_{t\rightarrow +\infty}w_t/a_t=w$ and $\lim_{t\rightarrow +\infty}a_t=+\infty$, we obtain that $\lim_{t\rightarrow +\infty}w_{t,i}/a_t=w$
for all $i$.
\end{pfof}

\subsection{Proof of Lemma~\ref{lem:S1S2}}
\label{sub:llllllll}
We will use Lemmas~\ref{lem:jointllt} and~\ref{lem:lld}.\\
\begin{pfof}{Lemma~\ref{lem:S1S2}(a)}
	This will follow from Lemma~\ref{lem:jointllt}.
	We consider the range $|n-t/{\mu(\tau)}|\le La_t$. 
	Since $\frac{w_t}{a_n}=\frac{w_t}{a_t}\frac{a_t}{a_n}\sim  \sqrt{\mu(\tau)}w$ and since $\frac{t-n\mu(\tau)}{a_n}$
	is bounded, it follows from
	Lemma~\ref{lem:jointllt} that
	\begin{align*}
	a_n^{d+1}\frac{Q_n(t,u)}{\mu(\tau)}&\sim \frac 1{\mu(\tau)} 
	g_{d+1}\left(\sqrt{\mu(\tau)}w
    ,\frac{t-n{\mu(\tau)}}{a_n}\right)\mu(A_0)\mu(B_0)
\Lambda_{d+1}(\{0\}\times J)\\
&\sim \frac{\mu(\tau)}{|I|
}	g_{d+1}\left(\sqrt{\mu(\tau)}w
,\frac{t-n{\mu(\tau)}}{a_n}\right)\nu(A)\nu(B)
\, ,
	\end{align*}
	uniformly in $n$ such that  $|n-t/{\mu(\tau)}|\le La_t$.
	Hence
\[
	\begin{split}
	a_t^dS_1(t,L) &\sim \sum_{n\,:\,|n-t/{\mu(\tau)}|\le La_t}\frac{(\mu(\tau))^{1+\frac d2}}{a_n}
	g_{d+1}\left(w\sqrt{\mu(\tau)},\frac{
	       t-n\mu(\tau)}{a_n}\right)\nu(A)\nu(B)
       \, .
	\end{split}
	\]
	Approximating Riemann sums by Riemann integrals, the right hand side converges, as $t\rightarrow +\infty$, to
\[
(\mu(\tau))^{\frac d2}\int_{-L{\mu(\tau)}^{3/2}}^{L{\mu(\tau)}^{3/2}} g_{d+1}\left(w\sqrt{\mu(\tau)},z\right)\, dz\, 
\nu(A)\nu(B)
\]
which itself converges to 
$\widetilde g_d(w)
\nu(A)\nu(B)
$ as $L\rightarrow +\infty$, as announced,
since $\widetilde g_d(w)=(\mu(\tau))^{\frac d2}\int_{\mathbb R}g_{d+1}(w\sqrt{\mu(\tau)},z)\, dz$.
\end{pfof}

For the proof of Lemma~\ref{lem:S1S2}(b),
note that $Q_n(t,u)\le \tQ_n(t)$ where
\[
\tQ_n(t)=\sup_{u\in I} \mu\left(\widehat\Psi_n\in (w_t,t-n{\mu(\tau)})+\{0\}\times J_u\right).
\]
Lemma~\ref{lem:S1S2}(b) is an immediate consequence of the next two sublemmas.

\begin{sublemma} \label{lem:1a}
For any $c\in(0,1/\mu(\tau))$,
	\(\displaystyle
	\lim_{L\to\infty}\limsup_{t\to\infty}a_t^d\sum_{n>ct:La_t< |n-t/\mu(\tau)|}\tQ_n(t)=0\, .
	\)
\end{sublemma}
\begin{proof}
	In this range, $n\gg t$, so $\frac n{a_n^{d+1}}=n^{-\frac{d-1}2}(\log n)^{-\frac{d+1}2}\ll \frac t{a_t^{d+1}}$.
	Thus, for any $L$ large enough, using Lemma~\ref{lem:lld} with $|z|_\infty= |t-n\mu(\tau)|$, we obtain that
	\[
	\sum_{n>ct:La_t<  |n-t/\mu(\tau)|} \tQ_n(t) 
	\ll \frac{t}{a_t^{d+1}}
	\sum_{n>ct: La_t<|n-t/\mu(\tau)|}
	\frac{|\log|t-n\mu(\tau)||}{1+|t-n\mu(\tau)|^2} 
	\ll \frac{t}{a_t^{d+1}} \frac{\log(a_t)}{La_t}\ll \frac 1{La_t^d}\, ,
	\]
	since $\log u/u^2$ has primitive $-(1+\log u)/u$ and since $a_t^2=t\log t\sim 2t\log a_t$.
\end{proof}

\begin{sublemma} \label{lem:2}
For any $c\in(0,1/\mu(\tau))$,
\(
\lim_{L\to\infty}\limsup_{t\to\infty}a_t^d\sum_{n<ct}\tQ_n(t)=0
\).
\end{sublemma}
\begin{proof}
	In this range, $t-n\mu(\tau)\approx t\gg n$.
Thus it follows from Lemma~\ref{lem:lld} that
\begin{align*}
\sum_{n<ct}\tQ_n(t)
&\ll \sum_{n<ct}\frac{n}{a_n^{d+1}}\frac{1+\log t}{t^2}\\
&\ll \frac{1+\log t}{t^2}\sum_{n< ct}\frac{1}{n^{\frac {d-1}2}(\log n)^{^\frac {d+1}2}}\\
&\ll \frac{1+\log t}{t^2}\frac{t^{\frac{3-d}2}}{(\log t)^{\frac {d+1}2}}\ll t^{-\frac{d+1}2}{(\log t)^{-\frac {d-1}2}}=o (a_t^{-d})\, .
\end{align*}
\end{proof}

\section{Proof of mixing for the Lorentz gas (Theorem~\ref{cor:mixingrate})}\label{sec:proofmixing}

In this section we assume $d=1$ or $d=2$. 
Corollary~\ref{cor:easymix} states mixing for 
functions with compact support in the suspension $\widehat{\mathcal{M}}\times\mathbb Z^d$.
To deal with the natural class of functions (with compact support in the manifold) in Theorem~\ref{cor:mixingrate}
we crucially rely on the following tightness-type result,
which is the most delicate part of this work.
\begin{thm}\label{lem:tight}
Let $K_0>0$ be fixed and let $B_0$ be the set of $\mathbf x\in\widetilde{\mathcal M}$ with position at a distance at most $K_0$ of the origin. For any positive integer $R_0$, let $B_{R_0}$ be the set of configurations $\mathbf x\in\widetilde{\mathcal M}$ belonging to $B_0$
and with no collision during the time interval $[0,R_0)$.
Then
\[
\lim_{R_0\rightarrow +\infty}\limsup_{t\rightarrow +\infty}a_t^d\widetilde\nu\left(B_{R_0}\cap\Phi_{-t}(B_0)\right)=0\, .
\]
\end{thm}
Before proceeding to the proof of
Theorem~\ref{lem:tight}, let us see how Theorem~\ref{cor:mixingrate} follows from
 Corollary~\ref{cor:easymix} and Theorem~\ref{lem:tight}.\\
\begin{pfof}{~Theorem~\ref{cor:mixingrate}}
Assume that $f$ and $g$ are nonnegative with support in $B_0$. 
We will use Corollary~\ref{cor:easymix} and the sets $E_{\pm n}$ defined therein. 
Observe that $B_0\setminus E_{-n}\subset B_n$. 
Let $(f_n)_n$ (resp.  $(g_n)_n$) be an increasing  sequence 
of continuous functions supported  in
$E_{-2n}$ (resp. $E_{2n}$)  coinciding with $f$ (resp. $g$)
on $E_{-n}$ (resp. $E_n$) and converging pointwise
to $f$ (resp. $g$).
Thus it will follow from Theorem~\ref{lem:tight} and  time-reversibility of $\Phi$ that
\begin{align}\nonumber
\lim_{n\rightarrow +\infty}\limsup_{t\rightarrow +\infty} &\left|a_t^d\int_{\widetilde{\mathcal M}}[f.g\circ \Phi_t-f_n.g_n\circ\Phi_t]\, d\widetilde\nu\right|\\
&\le \lim_{n\rightarrow +\infty}\limsup_{t\rightarrow +\infty}2\Vert f\Vert_\infty\Vert g\Vert_\infty a_t^d\widetilde\nu\left(B_{n}\cap\Phi_{-t}(B_0)\right)=0\, .\label{diffint}
\end{align}
Thus, for every $n$,
\begin{align*}
\limsup_{t\rightarrow +\infty} &\left|a_t^d\int_{\widetilde{\mathcal M}}f.g\circ \Phi_t-\widetilde g_d(0)\int_{\widetilde{\mathcal M}}f\, d\widetilde\nu\int_{\widetilde{\mathcal M}}g\, d\widetilde\nu\right|
\le\limsup_{t\rightarrow +\infty}\left|a_t^d\int_{\widetilde{\mathcal M}}[f.g\circ \Phi_t-f_n.g_n\circ\Phi_t]\, d\widetilde\nu\right|\\
&+
\limsup_{t\rightarrow +\infty}   \left|a_t^d\int_{\widetilde{\mathcal M}}f_n.g_n\circ \Phi_t\, d\widetilde{\nu}-\widetilde g_d(0)\int_{\widetilde{\mathcal M}}f_n\, d\widetilde\nu\int_{\widetilde{\mathcal M}}g_n\, d\widetilde\nu\right|\\
&+\widetilde g_d(0)\left|\int_{\widetilde{\mathcal M}}f_n\, d\widetilde\nu \int_{\widetilde{\mathcal M}}g_n\, d\widetilde\nu -\int_{\widetilde{\mathcal M}}f\, d\widetilde\nu\int_{\widetilde{\mathcal M}}g\, d\widetilde\nu\right|\\
&\le \limsup_{t\rightarrow +\infty}\left|a_t^d\int_{\widetilde{\mathcal M}}[f.g\circ \Phi_t-f_n.g_n\circ\Phi_t]\, d\widetilde\nu\right|\\
&+\widetilde g_d(0)\left|\int_{\widetilde{\mathcal M}}f_n\, d\widetilde\nu \int_{\widetilde{\mathcal M}}g_n\, d\widetilde\nu -\int_{\widetilde{\mathcal M}}f\, d\widetilde\nu\int_{\widetilde{\mathcal M}}g\, d\widetilde\nu\right|\, ,
\end{align*}
where we used Corollary~\ref{cor:easymix} applied to $f_n,g_n$ in the last inequality.
Since this holds true for any $n$, we conclude by taking the limit as $n\rightarrow +\infty$ thanks to~\eqref{diffint} and to the dominated convergence theorem. 
\end{pfof}

The rest of this section is devoted to the proof of Theorem~\ref{lem:tight}. Recall that $K_0$ is fixed and that we have to estimate $\widetilde\nu(B_{R_0}\cap \Phi_{-t}(B_0))$. The strategy of our proof is divided in two steps.
In a first step (corresponding to Subsection~\ref{Step1}), 
we explain how we can neglect "bad" configurations with first or last long free flights (Lemma~\ref{lem:B0}) or 
having a small number of collisions within the time interval $[0,t]$ (Lemmas~\ref{lem:B0'} and~\ref{B1second}). 
In a second step (corresponding to Subsection~\ref{Step2}),  we will estimate the probability of the set of "good" configurations belonging to $B_{R_0}\cap\Phi_{-t}(B_0)$ by writing (as in the proof of Proposition~\ref{LLTkappaNt} but with additional sums and complications) this set as a union of sets, the measure of which corresponds to the measure of a set that can be expressed in terms of the Sinai billiard map $T$. 

In the process of proving Theorem~\ref{lem:tight}, we obtain the following large deviation result (proved in Section~\ref{Step1}), which is interesting in its own right.

\begin{prop}\label{prop:L} For $c_1$ small enough,
\[\mu(\tau_{\lfloor c_1 t\rfloor}> t)\le\nu(N_t\le c_1t)/\min\tau=\mathcal O(\log t/t)\, .
\]
 \end{prop}
The bound $\mathcal O(\log t/t)$ is optimal because 
$\tau$ is in the domain of non standard CLT with normalization 
$\sqrt{t\log t}$. This bound  is in accord with the optimal result  in \emph{the i.i.d.
scenario} with $\sqrt{t\log t}$ normalization, see~\cite{BBY,Roz90}.
\subsection{Notations and recalls for the proof of Theorem~\ref{lem:tight}}\label{notationtightness}
Before entering deeper in the proof, let us introduce some needed 
notations. 
Recall that $\kappa$ stands for the cell change (with values in $\mathbb Z^d$). It will be useful to consider $\widetilde\kappa:M\rightarrow\mathbb Z^2$ for the cell-change for the $\mathbb Z^2$-periodic Lorentz gas; so that $\kappa=\pi_d(\widetilde\kappa)$ where $\pi_2=Id$ and $\pi_1:\mathbb R^2\rightarrow \mathbb R$ is the canonical projection on the first coordinate. Note that, when $d=2$, $\widetilde\kappa=\kappa$. 
 We extend the definition of $\widetilde \kappa$ to $\widetilde{\mathcal M}$ by setting 
$\widetilde \kappa(\Phi_u(q,\vec v))=\widetilde \kappa(p_d(q),\vec v)$ for every $x=(q,\vec v)\in \widetilde M$
and every $u\in[0,\tau(x))$.
Let
us write $\widetilde N_t(\mathbf x)$ for the number of collisions in the time interval $(0,t]$ for a trajectory starting from $\mathbf x\in\widetilde{\mathcal M}$. 
Recall that $H_0$ is the bounded coboundary defined in~\eqref{coboundPsi}. Throughout the rest of this section we fix $c_1$ so that 
\begin{equation}\label{eq:defc1}
c_1\in(0,1/(1000\mu(\tau)))\text{ and } 2c_1\Vert H_0\Vert_\infty<1/100.
\end{equation}
We consider the constant $a_0$ appearing in Lemma~\ref{lem:jointllt0}. 
Up to decreasing if necessary its value,
it follows from e.g. \cite[Theorem 7.37, Remark 7.38]{ChernovMarkarian}
that there exists $C'_0>0$ such that
\begin{align}\label{Deco2}
\forall n'\in\mathbb N_0,\quad Cov\left(f((\widetilde\kappa\circ T^m)_{m\ge n'}),g((\widetilde\kappa\circ T^{m'})_{m'\le 0})\right)\le C'_0\Vert f\Vert_\infty\Vert g\Vert_\infty e^{-a_0n'}\, ,
\end{align}
for any bounded measurable functions $f,g$.
by noticing that $f((\widetilde\kappa\circ T^m)_{m\ge 0})$ is constant on stable curves and that
$g((\widetilde\kappa\circ T^{m'})_{m'\le -1})$ is constant on unstable curves and as such, these functions are bounded by their infinite norms in the respective spaces $\mathcal H^-$
and $\mathcal H^+$ considered in \cite{ChernovMarkarian}.\\
We fix $K>0$ so that 
\begin{align}\label{Deco1} C_0'e^{-a_0 \lfloor K\log t\rfloor}\le t^{-100} 
\end{align}
We recall that it is proved in~\cite{SV07} that
\begin{equation}\label{tailkappa}
\mu(\widetilde\kappa=z)=\mathcal O(|z|^{-3})\, ,
\end{equation}
and that the set $\mathfrak C$ of unit vectors of $\mathbb R^2$ corresponding to the corridor directions in $\mathcal D_2$ (i.e. the direction of a line in $\mathbb R^2$ touching no obstacle) is finite. 
Finally, recall  that by~\cite[Propositions 11--12, Lemma 16]{SV07},
\begin{equation}\label{controltimelog}
\forall V>0,\quad \mu\left(\widetilde\kappa=z, \exists |j|\le V\log (|z|+2),\ j\ne 0,\  |\widetilde\kappa|\circ T^j>|z|^{4/5}\right)=\mathcal O\left(|z|^{-3-\frac 2{45}}\right)\,. 
\end{equation}

\subsection{Control of "bad" configurations}\label{Step1}
The first next lemma allows us to neglect trajectories
with no collision before time $t$. 
\begin{lemma}\label{lem:B0'}
	The following estimate holds true as $t\rightarrow +\infty$,
\[
\widetilde\nu\left(B_0\cap\Phi_{-t}(B_0)\cap\{\widetilde N_t=0\}\right)=o(a_t^{-d})\, .
\]
\end{lemma}
\begin{proof}
When $d=2$, the lemma is immediate: as soon as $t>4K_0$, $\widetilde N_t\ge 1$, otherwise, at time $t$, the trajectory cannot be  $2K_0$-close  of its initial position, at time $0$.\\
When $d=1$, we can have $\widetilde N_t=0$ because of possible long free-flights
in the vertical direction that remain at a bounded distance. However, using the representation of $(\widetilde{\mathcal M},(\Phi_t)_t,\widetilde\nu)$ as a suspension flow over a $\mathbb Z$-extension, 
\begin{align*}
\widetilde\nu\left(B_0\cap\{ \widetilde N_t=0\}\right)
&\le ( 2K_0+1) \nu(\widetilde N_t=0)\\
&\le \frac{2K_0+1}{\mu(\tau)}\int_{0}^{+\infty}\mu(\tau >s+t)\, ds\\
&\ll \int_{0}^{+\infty}(1+s+t)^{-2}\, ds\ll t^{-1}=o(a_t^{-1})\, ,
\end{align*}
where we used~\eqref{tailV} and~\eqref{linktauV}.
\end{proof}

The next lemma ensures that we can neglect trajectories
with long first or long last free flight.
\begin{lemma}\label{lem:B0}
There exists a constant $C'>0$ such that, for all $R_0$ and $t$ large enough,  
	\[
	\widetilde\nu\left(B_0\cap \{|\widetilde\kappa|>a_t^d\log R_0\}\right)\le C'a_t^{-d}/\log R_0
	\, .
	\]
\end{lemma}
\begin{rmk}\label{rmk:B0}
Note that, since $\Phi_t$
preserves the measure $\widetilde\nu$, we also have
$\widetilde\nu\left(\Phi_{-t}\left(B_0\cap \{|\widetilde\kappa|>a_t^d\log R_0\}\right)\right)\le C'a_t^{-d}/\log R_0$.
\end{rmk}
\begin{proof}[Proof of Lemma~\ref{lem:B0}]
Using again the representation of $(\widetilde{\mathcal M},(\Phi_t)_t,\widetilde\nu)$ as a suspension flow over a $\mathbb Z^d$-extension, we note that
	\begin{align*}
	\nonumber\widetilde\nu&\left(B_0\cap\{|\widetilde\kappa|>
	a_t^d\log R_0\}
	\right)\le(2K_0+1)^2\nu\left(|\widetilde\kappa|>
	a_t^d\log R_0\right)\\
	\nonumber&\le \frac {(2K_0+1)^2}{\mu(\tau)}\mathbb E_\mu\left[\tau 1_{\{|\widetilde\kappa|>a_t^d\log R_0\}}\right]\\
	\nonumber&\ll\int_{a_t^d\log R_0}^{+\infty}\mu(\tau>s)\, ds+a_t^d\log R_0\mu(|\widetilde \kappa|>a_t^d\log R_0))\\
	&\ll \int_{a_t^d\log R_0}^{+\infty}s^{-2}\, ds+(a_t^d\log R_0)^{-1}\ll a_t^{-d}/\log R_0\, ,
	\end{align*}	
using~\eqref{tailV}.
\end{proof}

The lemma below deals with the remaining range, namely $n=N_t\le c_1 t$
with $c_1$ as in~\eqref{eq:defc1}.
\begin{lemma}\label{B1second}
For all $R_0>0$,
	\begin{equation*}
\widetilde\nu\left(B_{0}\cap\Phi_{-t} (B_0)\cap\{\widetilde N_t\le c_1 t\}\right)	
=o(a_t^{-d})\, ,
	\end{equation*}
	as $t\rightarrow +\infty$.
\end{lemma}
To prove this lemma, we will deal separately with the cases $d=1$ and $d=2$.
The main ingredients of the proofs below in these two cases come down to a very delicate decomposition of the involved sum  along with fine estimates
via the use of~\eqref{controltimelog} and of Lemma~\ref{lem:lld} (joint LLD).
The proof of Lemma~\ref{B1second} for $d=1$ will use the following intermediate results.

\begin{sublemma}\label{Nt<logt}
Recall that $K$ satisfies~\eqref{Deco1}. Then
\[
\widetilde\nu\left(B_0\cap\{\widetilde N_{t/100}\le K\log t\}\right)\le (2K_0+1)^2
\nu( N_{t/100}\le K\log t)\ll (\log t)/t\, .\]
\end{sublemma}
\begin{pfof}{the sublemma}
	Recall that $B_0$ has diameter $2K_0$.
The first inequality comes from the fact that $B_0$ contains at most $(2K_0+1)^2$
copies of $\mathcal M$. Let us prove the second inequality. 
Observe that
	\begin{align}
	\nonumber
	&\nu (N_{t/100}\le K\log t)=\frac 1{\mu(\tau)}\mathbb E_{\mu}\left[\tau.\mathbf 1_{\{\tau_{\lfloor K\log t\rfloor+1}\ge t/100\}}\right]\\
	&\ll 
	\mathbb E_\mu\left[\tau. \mathbf 1_{\bigcup_{k=0}^{\lfloor K\log t\rfloor}\{\tau\circ T^k>t/(100(1+K\log t)\}}\right]\\
	\label{Nt100}&\ll 
	\mathbb E[\tau. \mathbf 1_{\{\tau\ge t/(100(1+K\log t))\}
	}
	]+
	\mathbb E\left[\tau. \mathbf 1_{\bigcup_{k=1}^{\lfloor K\log t\rfloor}\{\tau< t/(100(1+K\log t)),\ \tau\circ T^k\ge t/(100(1+K\log t))\}}\right]\, .
	\end{align}
Now, proceeding as in the proof of Lemma~\ref{lem:B0}, it follows from~\eqref{tailV} that for all $t>2$,\footnote{We use here again the classical formula $\mathbb E_\mu[X]=\int_0^{+\infty}\mu(X>z)\, dz$
valid for any positive measurable $X:M\rightarrow[0,+\infty)$. }
\begin{align}
\nonumber &\mathbb E_\mu\left[\tau. \mathbf 1_{\left\{\tau>\frac t{100(K\log t+1)}\right\}}\right]
=\int_0^{+\infty}\mu\left(\tau. \mathbf 1_{\left\{\tau>\frac t{100(K\log t+1)}\right\}}>z\right)\, dz \\
\nonumber	&\quad\quad = \int_{0}^{\frac t{100(K\log t+1)}} \mu\left(\tau>\frac t{100(K\log t+1)}\right)\, dz+\int_{\frac t{100(K\log t+1)}}^{+\infty} \mu\left(\tau>z\right)\, dz\\
	&\quad\quad\ll (t/\log t)^{-1}
	\, ,\label{Nt100a}
\end{align}
providing a control of the first term of the right hand side of~\eqref{Nt100}. 
For the second term of the right hand side of~\eqref{Nt100}, we 
distinguish the case of small (resp.big) values of $\tau$. 
Set $m:=(1+\frac 1{45})^{-1}$. On the one hand,
\begin{align}
	\mathbb E_\mu\left[\tau 1_{\{\tau\le t^{m
			}\}}. \mathbf 1_{\bigcup_{k=1}^{\lfloor K\log t\rfloor}\{\tau\circ T^k\ge t/(100(1+K\log t))\}}\right]\ll t^{
		m
	}\log t(t/\log t)^{-2}=t^{m-2}(\log t)^3\, ,\label{Nt100b}
\end{align}
where we used $\tau\le t^m
$  
and $\mu\left(\bigcup_{k=1}^{\lfloor K\log t\rfloor}\{\tau\circ T^k\ge z\}\right)\le K\log t\, \mu(\tau>z)\ll z^{-2}\log t$. On the other hand, since
$\tau-|\widetilde\kappa|$ is uniformly bounded by some constant $L_0$, 
\begin{align}
\nonumber&\mathbb E_\mu\left[\tau 1_{\{t^m\le \tau\le t/(100(1+K\log t))\}}. \mathbf 1_{\bigcup_{k=1}^{\lfloor K\log t\rfloor}\{\tau\circ T^k\ge t/(100(1+K\log t))\}}\right]\\
\nonumber&\le\sum_{z:t^m-L_0\le |z| \le t/(100(1+K\log t))+L_0}\hspace{-20mm}
(|z|+L_0)\mu\left(\widetilde\kappa=z,\quad\exists
k=1,...,K\log t,\,  |\widetilde\kappa|\circ T^k>\frac t{100 K\log t}-L_0\right)\\
&\quad\ll\sum_{z\in supp(\widetilde\kappa)\, :\, |z|\ge t^m-L_0}(|z|+L_0)|z|^{-3-\frac 2{45}}\ll t^{-m(1+\frac 2{45})}=t^{m-2}\, ,\label{Nt100c}
\end{align}
where we apply~\eqref{controltimelog} with $V=\frac{2K}m$ (indeed, for $t$ large enough, $K\log t\le V\log(t^m-L_0+2)$). Thus, the last bound of the sublemma follows from~\eqref{Nt100},~\eqref{Nt100a},~\eqref{Nt100b} and~\eqref{Nt100c} since $m-2=-\frac{47}{46}<-1$.
\end{pfof}

Lemma~\ref{B1second} in the case $d=1$ follows from the following result.
\begin{sublemma}\label{LDNt}
\[
	\widetilde\nu\left(B_0\cap\{\widetilde N_t\le c_1t\}\right)\le  (1+2K_0)^2\nu(N_t\le c_1t)\ll (\log t)/t\, .
\]
\end{sublemma}
\begin{pfof}{the sublemma}
Again the first inequality follows from the fact that $B_0$ contains at most
$(2K_0+1)^2$ copies of $\mathcal M$. The main issue is to establish
the last upper bound.
Since the flow $\phi$ preserves $\nu$, $\nu(N_t\le c_1t)=\nu(N_t\circ \phi_{-t/2}\le c_1t)$. Furthermore, the fact that $N_t\circ \phi_{-t/2}\le c_1t$ means that there are at most $c_1t$ collisions in the time interval $[-t/2;t/2]$, so at most $c_1t$ collisions in both time intervals $[-t/2;0]$
and $[0;t/2]$, which implies that both $\tau_{\lfloor c_1t\rfloor}$
and $\tau-\tau_{-\lfloor c_1 t\rfloor}$ are larger than $t/2$,  writing as usual $\tau_{-k}:=-\sum_{m=-k}^{-1}\tau\circ T^m$ and $\tau_m(\phi_u(x)):=\tau_m(x)$ for all $(x,u)\in\widehat {\mathcal M}$. 
Therefore
	\begin{align}
    \nu(N_t\le c_1t)=\nu(N_t\circ \phi_{-t/2}\le c_1t)\le \nu(\tau_{\lfloor c_1t\rfloor}>t/2,\tau-\tau_{-\lfloor c_1t\rfloor}>t/2    )\, ,\label{intermediatedim1b2}
	\end{align}
	In what follows we show that this quantity is $o(a_t^{-1})$.
By  Lemma~\ref{lem:B0}, $\nu(\tau> a_t^2)\ll a_t^{-2}$ and by Sublemma~\ref{Nt<logt}, $\nu(\tau_{\lfloor K\log t\rfloor}>t/100)\ll (\log t)/t $ and 
$\nu(\tau-\tau_{-\lfloor K\log t\rfloor}>t/100)\ll (\log t)/t $ (up to time reversibility of $\Phi$). This combined with~\eqref{intermediatedim1b2} ensures that
	\begin{align}\label{Nt<c1t}
	\nu(N_{t}\le c_1t)
	&\le p_t+\mathcal O((\log t)/t)\, ,
	\end{align}
	with
	\begin{align*}
	p_t&:=\nu\left(\tau<a_t^2,\tau_{\lfloor c_1 t\rfloor}>t/2,\tau-\tau_{-\lfloor c_1 t\rfloor}>t/2, \tau_{\lfloor K\log t\rfloor}<t/100
	,\tau-\tau_{-\lfloor K\log t\rfloor}<t/100
	\right)\\
	&=\sum_{a'=0}^{a_t^2}\mu\left(\tau>a',\, 
	\tau_{\lfloor c_1 t\rfloor}>t/2,\tau-\tau_{-\lfloor c_1 t\rfloor}>t/2, \tau_{\lfloor K\log t\rfloor}<t/100
	,\tau-\tau_{-\lfloor K\log t\rfloor}<t/100
	\right)
	\\
	&\le \sum_{a'=0}^{a_t^2}\mu\Big(|\widetilde\kappa|>a'-2\Vert H_0\Vert_\infty,\, 
	\tau_{\lfloor c_1t\rfloor}>t/2,\\
	&
	\quad\quad\quad\quad\quad\quad\quad\quad\quad\quad,\tau+\tau_{-\lfloor c_1t\rfloor}>t/2, \tau_{\lfloor K\log t\rfloor}<t/100,\,  \tau-\tau_{-\lfloor K\log t\rfloor}<t/100
	\Big)
	\, ,
	\end{align*}
	with the function $H_0$ appearing in~\eqref{coboundPsi} (since $|V|=\tau$ when $d=2$). 
	Observe that
	\[\tau_{\lfloor c_1t\rfloor}\circ T^{\lfloor K\log t\rfloor}>\tau_{\lfloor c_1t\rfloor-\lfloor K\log t\rfloor}\circ T^{\lfloor K\log t\rfloor}
	=\tau_{\lfloor c_1t\rfloor}-\tau_{\lfloor K\log t\rfloor}
	\]
	and that
	\[-\tau_{-\lfloor c_1t\rfloor}
	\circ T^{-\lfloor K\log t\rfloor}
	>-\tau_{-\lfloor c_1t\rfloor	+\lfloor K\log t\rfloor}\circ T^{-\lfloor K\log t\rfloor}
	= -\tau_{-\lfloor c_1t\rfloor
	}+\tau-\tau+\tau_{-\lfloor K\log t\rfloor}
	\, .
	\]
	It follows that 
	\begin{align*}
	p_t&\le\sum_{a'=0}^{a_t^2\log R_0}\mu\left(|\widetilde\kappa|>a'-2\Vert H_0\Vert_\infty,\, 
	\tau_{\lfloor c_1t\rfloor}\circ T^{\lfloor K\log t\rfloor}>49t/100,|\tau_{-\lfloor c_1t\rfloor}|\circ T^{-\lfloor K\log t\rfloor}>49t/100\right).
	\end{align*}
Recall that, for all $m\in\mathbb Z$, $|\widetilde\kappa_{m}|\ge \tau_{m}-2m\Vert H_0\Vert_\infty$
and that, due to~\eqref{eq:defc1}, $2c_1\Vert H_0\Vert_\infty<1/100$. Thus,
		\begin{align*}
	p_t
	&\le\sum_{a'=0}^{a_t^2\log R_0}\mu\left(|\widetilde\kappa|>a'-2\Vert H_0\Vert_\infty,\, 
	|\widetilde\kappa|_{\lfloor c_1t\rfloor}\circ T^{\lfloor K\log t\rfloor}>48t/100,-|\widetilde\kappa|_{-\lfloor c_1t\rfloor}\circ T^{-\lfloor K\log t\rfloor}>48t/100\right)\, .
	\end{align*}
	Thus, using~\eqref{Deco2} (twice) combined with~\eqref{Deco1},
	\begin{align}\label{intermediatedim1b}
 p_t&\le \sum_{a'=0}^{a_t^2\log R}
	\left(\mu\left(|\widetilde\kappa|>a'-2\Vert H_0\Vert_\infty\right)\left(\mu\left( 
	|\widetilde\kappa|_{\lfloor c_1t\rfloor}>48t/100\right)\right)^2+\mathcal O(t^{-100})\right)\, .
	\end{align}
	By Lemma~\ref{lem:lld} (with $d=0$) and since $\tau$ is $\mu$-integrable,
	\begin{align*}
	\mu\left( |\widetilde\kappa|_{\lfloor c_1t\rfloor}>48t/100\right)
	&\le\mu\left( \tau_{\lfloor c_1t\rfloor}>47t/100\right)\\
	&\le  \mu(\tau_{\lfloor c_1t\rfloor}>t^2)+\mu\left( 47t/100<\tau_{\lfloor c_1t\rfloor}<t^2\right)\\
	&\ll \frac{\mathbb E_{\mu}[\tau_{\lfloor c_1 t\rfloor}]}{t^2}+\sum_{k= 2t/5}^{t^2}\frac{t}{\sqrt{t\log t}}\frac{\log (k-\lfloor c_1t\rfloor\mu(\tau))}{1+(k-\lfloor c_1t\rfloor\mu(\tau))^2}\\
	&\ll t^{-1}+\sum_{k\ge t(2/5-c_1\mu(\tau))}\frac{t}{\sqrt{t\log t}}\frac{\log t}{1+k^2}\\
	&\ll \frac{t}{\sqrt{t\log t}}\frac{\log t}t=\sqrt{\frac{\log t}t}\, . 
	\end{align*}
Combining this with~\eqref{intermediatedim1b} and~\eqref{Nt<c1t}, we infer
	\begin{align*}
	\nu(N_t<c_1t)&\le \sum_{a'=0}^{a_t^2\log R_0}
	\left((a'+1)^{-2}\frac{\log t}t+\mathcal O(t^{-100})\right)+\mathcal O((\log t)/t)=\mathcal O(\log t/t)\, .
	\end{align*}
\end{pfof}

We take the line below to quickly complete

\begin{pfof}{~Proposition~\ref{prop:L}}
We observe that
\begin{align*}
\mu(\tau_{\lfloor c_1 t\rfloor}> t)&\le\frac{\widehat\nu(\{(x,u)\in\widehat{\mathcal M}\, :\, \tau_{\lfloor c_1 t\rfloor}(x)> t,\ u\le\min\tau\})}{\min\tau}\\
&\le \frac{\widehat\nu(\{(x,u)\in\widehat{\mathcal M}\, :\, N_t(x)< \lfloor c_1t\rfloor,\ u\le\min\tau\})}{\min\tau}\\
&\le \frac{\nu(N_t\le  \lfloor c_1t\rfloor)}{\min\tau} \le \frac{\nu(N_t\le  c_1t)}{\min\tau}\, ,
\end{align*}
and conclude due to Sublemma~\ref{LDNt}.
\end{pfof}

We continue with

\begin{pfof}{Lemma~\ref{B1second}}
When $d=1$, the result follows from Sublemma~\ref{LDNt} since
\[
\widetilde\nu\left(B_0\cap\Phi_{-t}(B_0)\cap\{\widetilde N_t\le c_1t\}\right)
\le \widetilde\nu\left(B_0\cap\{\widetilde N_t\le c_1t\}\right)\ll (\log t)/t=o(a_t^{-1})\, .
\]
Unfortunately this estimate is not enough when $d=2$. We assume from now on throughout this proof that $d=2$. Recall that we have to prove that
\[
 \widetilde\nu\left(B_0\cap\Phi_{-t}(B_0)\cap\{\widetilde N_t\le c_1t\}\right)=o(a_t^{-2})=o((t\log t)^{-1})\quad\mbox{(assuming }d=2\mbox{)}\, .
\]
	We start with some preliminary calculation that will allow us
	to argue that we can neglect the configurations with more than one long free flight (of length larger than $t/(100(K\log t)^2)$) among the $K\log t$ future and past collision times.
Let us write 
\[D_t:=\{\exists k,\ell\, :\,  k\ne \ell; |k|,|\ell|\le K\log t,\min(\tau\circ T^k,\tau\circ T^\ell)>t/(100(K\log t)^2)\}\subset \mathcal M\]
and $\widetilde D_t$ for the corresponding event in $\widetilde{\mathcal M}$.
\begin{sublemma}\label{twolongff}
	For all $\varepsilon\in(0,\frac 1{45})$,
\begin{equation*}\label{ESTIM0}
\widetilde\nu\left(B_0\cap\Phi_{-t/2}(\widetilde D_t)\right)\le (2K_0+1)^2\nu(
D_t)= o(t^{-1-\frac 1{45}+\varepsilon})=o(a_t^{-2})\, ,\quad\mbox{as }t\rightarrow +\infty\, .
\end{equation*}
\end{sublemma}
\begin{pfof}{the sublemma.}
Again, as in the proofs of Sublemmas~\ref{Nt<logt} and~\ref{LDNt}, the first inequality follows from the fact that $\widetilde{\mathcal M}$ is made of at most $(2K_0+1)^2$ copies of $\mathcal M$ and that $\phi$ preserves the measure $\nu$. It remains to prove the last estimate. 
Using the suspension flow representation and the H\"older inequality applied for any $p<2<q$ such that $\frac 1p+\frac 1q=1$ and close enough to 2, we observe that
	\begin{align*}
	\nu(D_t)&\le 2 
	\sum_{-K\log t\le k
		<\ell\le K\log t}
	\mathbb E_\mu[\tau. \mathbf 1_{\tau\circ T^k>t/(100(K\log t)^2), 
		\, \tau\circ T^\ell>t/(100(K\log t)^2)}]\\
	\nonumber&\le 2
	\sum_{-K\log t\le k
		<\ell\le K\log t}\Vert\tau\Vert_{L^{p}}
	\left(\mu(\tau\circ T^k>t/(100(K\log t)^2),\tau\circ T^\ell>t/(100(K\log t)^2))\right)^{\frac 1q}
	\\
	\nonumber&\le 2
	\sum_{-K\log t\le k,\ell\le K\log t\, :\, k\ne \ell}\Vert\tau\Vert_{L^{p}}
	\left(\sum_{i\in Supp(\widetilde\kappa):|i|>t/(100(K\log t)^2)} \mu(\widetilde \kappa\circ T^k=i,|\widetilde\kappa|\circ T^\ell>|i|^{\frac 45}\right)^{\frac 1q}\, .
	\end{align*}
Finally applying~\eqref{controltimelog} with $V=4K$ and $t$ large enough so that $2K\log t <V\log(2+\frac{t}{100(K\log t)^2})$, we conclude that, for $t$ large enough,
	\begin{align}	\nonumber
\nu(D_t)&\ll 
	(\log t)^2
	\left(\sum_{i>t/(100(K\log t)^2)} i^{-3-\frac 2{45}}\right)^{\frac 1q}\, ,\label{FFF0dim2}
	\end{align}
and this bound holds true for an arbitrary real number $q>2$.
\end{pfof}

We are back to the proof of Lemma~\ref{B1second} assuming that $d=2$.
We will decompose the quantity we have to estimate in $p_{t,1}+p_{t,2}$,
distinguishing the case where the free flight at time $t/2$ is larger or smaller than $t/4$.

	{\bf{Estimate when the free flight at time $t/2$ 
			is larger than $t/4$.}} In this part, we study
\begin{equation}
p_{t,1}:=\widetilde\nu\left(B_0\cap\Phi_{-t}(B_0)\cap\{\widetilde N_t\le c_1t,\tau\circ \Phi_{\frac t2}>t/4\}\right)\, .
\end{equation}
We will use the fact that we can neglect the trajectories such that $\tau\circ \Phi_{t/2}>a_t^3$ since, due to Lemma~\ref{lem:B0},
\begin{equation}\label{B0t/2}
\widetilde\nu(B_0\cap\{\tau\circ \Phi_{t/2}>a_t^3\})\le (2K_0+1)^2\nu\left(\tau\circ \phi_{t/2}>a_t^3\right)\ll a_t^{-3}\, .
\end{equation}
It follows from~\eqref{B0t/2} and from  Sublemma~\ref{twolongff} that
\begin{equation}
p_{t,1}=\widetilde\nu\left(B_0\cap\Phi_{-t}(B_0)\cap\left\{\widetilde N_t\le c_1t\right\}\cap\Phi_{-\frac t2}\left(\left\{\tau\in\left[\frac t4,a_t^3\right]\right\}\setminus\widetilde D_t\right)\right)+o(a_t^{-2})\, .
\end{equation}
Let us study the event appearing in the above formula.\\
The fact that $\tau\circ \Phi_{\frac t2}>t/4$ and that $\Phi_{t/2}\not\in D_t$ implies that the $K\log t$ free flights just before and just after the one occurring at time $t/2$ have  all length smaller than $t/(100(K\log t)^2)$ and thus that
	\begin{equation}\label{tauKlogt}\tau_{\lfloor K\log t\rfloor}\circ \widetilde T\circ\Phi_{t/2}<t/(100(\lfloor K\log t\rfloor))\quad\mbox{and}\quad |\tau_{-\lfloor K\log t\rfloor}|\circ\Phi_{t/2}<t/(100(\lfloor K\log t\rfloor))\, .
	\end{equation}
Recall that the configuration is in $B_0\cap\Phi_{-t}(B_0)$ and satisfies
$\tau\circ\Phi_{t/2}>t/4$. Since $\widetilde N_t\le c_1t$ and  since we are in dimension 2, the free flight of length $t/4$ made at time $t/2$ has to be canceled by the sum of the other (at most $(\lfloor c_1t\rfloor-1)$) free flights made during the time interval $[0;t]$. So,
\begin{equation}\label{tauc1t}
|\tau_{-\lfloor c_1 t\rfloor}|\circ\Phi_{t/2}>t/8\quad\mbox{or}\quad 
(\tau_{\lfloor c_1 t\rfloor}-\tau)\circ\Phi_{t/2}>t/8\, .
\end{equation}
The combination of Conditions~\eqref{tauKlogt} and~\eqref{tauc1t} implies
that at least one of the two next conditions should holds true
	\[
	|\tau_{-\lfloor c_1 t\rfloor}|\circ \widetilde T^{-\lfloor K\log t\rfloor}\circ\Phi_{t/2}>
	|\tau_{-\lfloor c_1 t\rfloor}|\circ\Phi_{t/2}-|\tau_{-\lfloor K\log t\rfloor}|\circ\Phi_{t/2}>11t/100
	\]
	or
	\[
	\tau_{\lfloor c_1 t\rfloor}\circ \widetilde T^{\lfloor K\log t\rfloor}\circ\Phi_{t/2}>
	(\tau_{\lfloor c_1 t\rfloor}-\tau)\circ\Phi_{t/2}-\tau_{\lfloor K\log t\rfloor}\circ \widetilde T\circ\Phi_{t/2}>11t/100\, .
	\]
Note that the second  condition above corresponds to the first one above one up to composing by $\Phi_t$ and up to using time reversal.
Therefore, 

\begin{align*}
p_{t,1}&\le 2\widetilde\nu\left(B_0\cap\left\{\tau\circ \Phi_{\frac t2}\in\left[\frac t4;a_t^3\right],\ \tau_{\lfloor c_1 t\rfloor}\circ \widetilde T^{\lfloor K\log t\rfloor}\circ\Phi_{t/2}>\frac{11t}{100} \right\}\right)   +o(a_t^{-2})\\
&\le 2(2K_0+1)^2\nu\left(\tau\circ \phi_{\frac t2}\in\left[\frac t4;a_t^3\right],\ \tau_{\lfloor c_1 t\rfloor}\circ T^{\lfloor K\log t\rfloor}\circ\phi_{t/2}>\frac{11t}{100} \right)   +o(a_t^{-2})
\end{align*}
using again the fact that $B_0$ is made of at most $(2K_0+1)^2$ copies
of $\mathcal M$. Now using the $\phi_{t/2}$-invariance of $\nu$, we obtain
\begin{align}
p_{t,1}&\le 2(2K_0+1)^2\nu\left(\tau\in\left[\frac t4;a_t^3\right],\ \tau_{\lfloor c_1 t\rfloor}\circ T^{\lfloor K\log t\rfloor}>\frac{11t}{100} \right)   +o(a_t^{-2})\\
	\nonumber&\ll o(a_t^{-2})+
	\sum_{a'=t/4}^{a_t^3}\mu\left(\tau>a',
	\tau_{\lfloor c_1 t\rfloor}\circ T^{\lfloor K\log t\rfloor}>11t/100\right) \\
	\nonumber&\ll o(a_t^{-2})
	+\sum_{a'=t/4}^{a_t^3}\mu\left(|\kappa|>a'-2\Vert H_0\Vert_\infty,|\kappa|_{\lfloor c_1 t\rfloor}\circ T^{\lfloor K\log t\rfloor}>10t/100 \right)\\
	\nonumber&\ll o(a_t^{-2})+
	\sum_{a'=t/4}^{a_t^3}\left(\mu\left(|\kappa|>a'-2\Vert H_0\Vert_\infty\right)\mu\left(|\kappa|_{\lfloor c_1 t\rfloor}>10t/100\right)+\mathcal O(t^{-100})\right)\\
	\label{AAA}&\ll o(a_t^{-2})+\sum_{a'=t/4}^{a_t^3}(a')^{-2}
	\mu\left(|\kappa|_{\lfloor c_1 t\rfloor}>t/10\right) .
	\end{align}
	But, using Lemma~\ref{lem:lld} (with $d=0$) and the
	$\mu$-integrability of $\tau$,
	\begin{align}
	\nonumber\mu\left( |\widetilde\kappa|_{\lfloor c_1 t\rfloor}>t/10\right)
	&\le\mu\left( \tau_{\lfloor c_1 t\rfloor}>9t/100\right)\\
	\nonumber&\le\mu\left( \tau_{\lfloor c_1 t\rfloor}>t^3\right)+\mu\left( 9t/100<\tau_{\lfloor c_1 t\rfloor}\le t^3\right)\\
	\nonumber &\ll \frac{\mathbb E_{\mu}[\tau_{c_1 t}]}{t^3}+\sum_{k= 9t/100}^{t^3}\frac{t}{\sqrt{t\log t}}\frac{\log (k-\lfloor c_1 t\rfloor\mu(\tau))}{1+(k-\lfloor c_1 t\rfloor\mu(\tau))^2}\\
	\nonumber &\ll t^{-2}+\sum_{k\ge t(9/100-c_1\mu(\tau))}\frac{t}{\sqrt{t\log t}}\frac{\log t}{1+k^2}\\
	&\ll t^{-2}+ \frac{t}{\sqrt{t\log t}}\frac{\log t}t=\sqrt{\frac{\log t}t}\, . \label{LD}
	\end{align}
	This together with~\eqref{AAA} implies that
	\begin{align}\label{BBB}
	p_{t,1}
	=o(a_t^{-2})\, .
	\end{align}
	
	{\bf{Estimate when the free flight at time $t/2$ is smaller that $t/4$.}}
	Let
	\begin{equation*}
	p_{t,2}:=\widetilde\nu\left(B_0\cap\Phi_{-t}(B_0)\cap\{\widetilde N_t\le c_1t,\tau\circ \Phi_{\frac t2}\le t/4\}\right)\, .
	\end{equation*}
Using again the fact that $B_0$ is made of at most $(2K_0+1)^2$ copies of $\mathcal M$ and that $\phi_{t/2}$ preserves $\nu$,
	\begin{align}
	p_{t,2}&\le (2K_0+1)^2\nu(N_t\le c_1t,\tau\circ\phi_{t/2}\le t/4)\label{controlpt2}\\
	 &\le (2K_0+1)^2\nu\left(
	N_{t/2}\le c_1t,N_{-t/2}\le c_1t, \tau\le t/4\right)\, .\nonumber
	\end{align}
	This together with Sublemma~\ref{twolongff} and time reversibility gives that
	\begin{align*}
	&\nu(N_t<c_1t,\tau\circ\phi_{t/2}\le t/4)\\
	&\le o(a_t^{-2})+ \nu\left(
	N_{t/2}\le c_1t,N_{-t/2}\le c_1t, \tau\le t/4
	,\min(\tau_{\lfloor K\log t\rfloor}-\tau,|\tau_{-\lfloor K\log t\rfloor})|\le t/100\right)\\
	&\le o(a_t^{-2})+2 \nu\left(
	N_{t/2}\le c_1t,\, N_{-t/2}\le c_1t, \, \tau\le t/4, \, 
	|\tau_{-\lfloor K\log t\rfloor}|<t/100\right).
	\end{align*}
	Since we also know that
	\[
	|\tau_{-\lfloor c_1 t\rfloor}|\circ T^{-\lfloor K\log t\rfloor}=|\tau_{-\lfloor K\log t\rfloor-\lfloor c_1 t\rfloor}-\tau_{-\lfloor K\log t\rfloor}|
	>|\tau_{-\lfloor c_1 t\rfloor}-\tau|-|\tau_{-\lfloor K\log t\rfloor}-\tau|>\frac t2 -\frac t4-\frac t{100}\, .
	\] 
	we obtain
	\begin{align*}
	&\nu(N_t\le c_1t,\tau\circ\phi_{t/2}\le t/4)\le o(a_t^{-2})+2 \nu(
	N_{t/2}\le c_1t,\, 
	|\tau_{-\lfloor c_1 t\rfloor}|\circ T^{-\lfloor K\log t\rfloor}>24t/100,\tau\le t/4).
	\end{align*}
		Let
	\[
	A_{a',t'}=
	\{\tau\ge a'-4\Vert H_0\Vert_\infty,
	\tau_{\lfloor c_1 t\rfloor}-a'
	> 48t/100\}\, .
	\]
	Using again the representation by a suspension flow and the correlation estimate~\eqref{Deco2} combined with~\eqref{Deco1},
	\begin{align*}
	&\nu(N_t\le c_1t,\tau\circ\phi_{t/2}\le t/4)+o(a_t^{-2})\\
	&\le \sum_{a'=0}^{\lfloor t/4\rfloor}\mu\left(\tau\ge  a',\tau_{\lfloor c_1 t\rfloor}-a'\ge t/2,|\tau_{-\lfloor c_1 t\rfloor}|\circ T^{-\lfloor K\log t\rfloor}> 24t/100\right)\\
	&\le \sum_{a'=0}^{\lfloor t/4\rfloor}\mu\left(|\kappa|\ge a'-2\Vert H_0\Vert_\infty,|\kappa|_{\lfloor c_1 t\rfloor}-a'\ge 49t/100,\left||\kappa|_{-\lfloor c_1 t\rfloor}\right|\circ T^{-\lfloor K\log t\rfloor}> 23t/100\right)\\
	&\le \sum_{a'=0}^{\lfloor t/4\rfloor}\left(\mu\left(A_{a',t}\right)\mu\left(|\tau_{-\lfloor c_1 t\rfloor}|> 22t/100\right)+\mathcal O(t^{-100})\right)\\
	&\le \sum_{a'=0}^{\lfloor t/4\rfloor}\mu\left(A_{a',t}\right)\sqrt{\frac{\log t}t}\, ,
	\end{align*}
	where in the last line we have used~\eqref{LD}.
	
	The previous displayed estimate
	gives that
	\begin{align*}
	\nu(N_t\le c_1t,\tau\circ\phi_{t/2}\le t/4)&\le o(a_t^{-2})+\nu(N_{48t/100+4\Vert H_0\Vert_\infty}\le c_1 t)\sqrt{\frac{\log t}t}\\
	&= o(a_t^{-2})\, ,
	\end{align*}
	which together with~\eqref{controlpt2} ensures that $p_{t,2}=o(a_t^{-2})$, which combined with~\eqref{BBB} ends the proof of
	Lemma~\ref{B1second} in the case $d=2$.~\end{pfof}

\subsection{Control on "good" configurations}\label{Step2}
Due to Lemmas
~\ref{B1second} and~\ref{lem:B0} and to Remark~\ref{rmk:B0}, it remains to study the 
$\widetilde\nu$-measure of the set of configurations $\mathbf x\in B_{R_0}\cap \Phi_{-t}(B_0)$, that have at least $c_1t$ collisions in the time interval $[0,t)$ and with first and last free flight both smaller than $a_t^d\log R$.
The next lemma provides the domination of
this measure by a sum. 
Let us write $\mathfrak C$ for the set of \emph{unit direction of 
	corridors} in $\mathcal D_2$. We recall that this set is finite (see e.g.~\cite{SV07}).

\begin{lemma}\label{lem:decomp}
There exists a positive integer $L_0$, a positive real number $C_0$ and a compact set $K'\subset\mathbb Z^d\times\mathbb R$ such that, for all integer $R_0>L_0$ and all $t>0$,
	\begin{align*}
	&\widetilde\mu\left(B_{R_0}\cap \Phi_{-t}(B_{0})\cap\{\widetilde N_t> c_1t,\ |\widetilde \kappa|\le a_t^d\log R,\ |\widetilde \kappa\circ\Phi_t|\le a_t^d\log R\}\right)\\
	&\quad\le C_0 \sum_{\vec w_1,\vec w_2\in\mathfrak C}\sum_{n=\lfloor c_1 t\rfloor }^{\lfloor t/\min\tau\rfloor}\sum_{a= R_0-L_0}^{\lfloor a_t^d\log R_0\rfloor}\sum_{b=  0}^{\lfloor a_t^d\log R_0\rfloor}
	\mu(A'_{a,b,n,t}(\vec w_1,\vec w_2,K'))\, ,
	\end{align*}
where we set
\begin{align}\nonumber
A'_{a,b,n,t}(\vec w_1,\vec w_2,K')=&\left\{\widehat\Psi_n \in \left(-\pi_d(a\vec w_1+b\vec w_2), t-n\mu(\tau)-b-a\right)+K', \right.\\
&\quad\left.
|\widetilde\kappa\circ T^{-1}|\ge a,|\widetilde\kappa\circ T^n|\ge b\right\}\, .\label{eq:defApr}
\end{align}

\end{lemma}
\begin{proof}
Let $\mathbf x\in B_{R_0}\cap \Phi_{-t}(B_{0})\cap\{\widetilde N_t> c_1t,\ |\widetilde \kappa|\le a_t^d\log R,\ |\widetilde \kappa\circ\Phi_t|\le a_t^d\log R\}$. 
We will parametrise $\mathbf x$ by $(x,-u,\ell)\in M\times (-\infty,-R_0]\times\mathbb Z^d$, with $u\in[0,\tau(T^{-1}(x)))$. We write  
$\mathbf  x$ under the form $\Phi_{-u}(\widetilde x)$ with $\widetilde x\in \widetilde M$ corresponding to the configuration of the particle at the next (future) collision time. This configuration $\widetilde x$ belongs to some cell $\mathcal C_\ell$ with $\ell\in\mathbb Z^d$ and thus $\widetilde x$ can be rewritten under the form $(p_{d,0}^{-1}(q)+\ell,\vec v)$  for some $x=(q,\vec v)\in M$ (as explained in Section~\ref{sec:obstacles}). By construction
$u\in[R_0,\tau(T^{-1}(x)))$. We parametrize $\Phi_t(\mathbf x)$ by $((T^{n}(x),s),\ell')\in\widehat{\mathcal M}\times \mathbb Z^d$, as follows. 
Recalling  that $N_t$ is the lap number introduced above~\eqref{eq:lapn}, we write $\Phi_t(\mathbf x)$ under the form
$\Phi_s(\widetilde T^{n}(\widetilde x))$ with $n=N_t(\phi_{-u}(x))-1$
and $s\in[0, \tau(T^{n}(x))$.
Due to our assumptions on $\mathbf x$, we know that $N_t(\phi_{-u}(x))\ge \lfloor c_1 t\rfloor+1 $ so that $n\ge \lfloor c_1 t\rfloor$. Moreover, $N_t\le 1+ t/\min\tau$. Thus
\begin{equation}\label{boundn}
 \lfloor c_1 t\rfloor\le n\le t/\min\tau\, .
\end{equation}
It follows from~\eqref{linktauV} and~\eqref{coboundPsi}
that $\tau-|\widetilde\kappa|$ is uniformly bounded. Recall that $u\le\tau(T^{-1}(x))$
and $s\le\tau(T^n(x))$.
We discretise $u,s$ by setting
\begin{align}\label{defia}
a&:=\max(0,\lfloor u\rfloor-\Vert \tau-|\widetilde\kappa|\Vert_\infty)\le |\widetilde\kappa(T^{-1}(x))|\le a_t^d\log R_0\, ,\\
b&:=\max(0,\lfloor s\rfloor-\Vert \tau-|\widetilde\kappa|\Vert_\infty)\le |\widetilde\kappa(T^n(x))|\le a_t^d\log R_0\, .\label{defib}
\end{align}
With the previous notations,  
\begin{equation}\label{taun}
\widetilde\tau_n(x)=t-n\mu(\tau)-u-s=t-n\mu(\tau)-a-b+\mathcal O(1)\, ,
\end{equation}
where $\mathcal O(1)$ is uniformly bounded
(independently of $\mathbf x$). 
Furthermore, there exist two unit directions of corridors $\vec w_1,\vec w_2\in\mathfrak C$ that are  "close" to be collinear 
to, respectively, the first and last free flight, meaning that
\[
\widetilde\kappa(T^{-1}(x))=|\widetilde\kappa(T^{-1}(x))|\vec w_1+\mathcal O(1)\quad\mbox{and}\quad
\widetilde\kappa(T^n(x))=|\widetilde\kappa(T^n(x))|\vec w_2+\mathcal O(1)\, ,
\]
with $\mathcal O(1)$ uniformly bounded (independently of $\mathbf x$).
This implies that  
$\ell=a\pi_d(\vec w_1)+\mathcal O(1)$ and
$\ell'=-b\pi_d(\vec w_2)+\mathcal O(1)$, where again
$\mathcal O(1)$ is uniformly bounded (independently of $\mathbf x$). Thus, for a given $(a,b,\vec w_1,\vec w_2)$,
only a uniformly bounded number of values of $(\ell,\ell')$ are possible.
Second, this implies also that
\begin{equation}\label{kappan}
\kappa_n+\pi_d(a\vec w_1+b\vec w_2)
=\ell'-\ell+a\pi_d(\vec w_1)+b\pi_d(\vec w_2)=\mathcal O(1) .
\end{equation}
Recalling that $\widehat\Psi_n=\left(\kappa_n,\widetilde\tau_n\right)$, it follows from~\eqref{defia},~\eqref{defib},~\eqref{taun} and~\eqref{kappan} that we can find a compact set $K'$ independent of $\mathbf x$ such that, with previous notations,
\[
x\in A'_{a,b,n,t}(\vec w_1,\vec w_2,K')\, .
\]
The bounds on $n,a,b$ comes from respectively
~\eqref{boundn},~\eqref{defia} (and $u\ge R_0$) and~\eqref{defib}.
The multiplicative constant $C_0$ comes from the bounded number of possible values of $(\ell,\ell',\lfloor u\rfloor,\lfloor s\rfloor)$ once $a$ and $b$ are fixed.
This ends the proof of the lemma.
\end{proof}

Since $\mathfrak C$ is finite, it is enough to fix $\vec w_1,\vec w_2$ and to prove that
\begin{align}\label{eq:sh0}
 \lim_{R_0\rightarrow +\infty}\limsup_{t\rightarrow +\infty}a_t^d\sum_{n=\lfloor c_1t\rfloor }^{\lfloor t/\min\tau\rfloor}\sum_{a= R_0}^{\lfloor a_t^d\log R_0\rfloor}\sum_{b=  0}^{\lfloor a_t^d\log R_0\rfloor}
\mu(A'_{a,b,n,t}(\vec w_1,\vec w_2,K'))=0\, .
\end{align}
The aimed result~\eqref{eq:sh0} will be proved via the next technical lemma, the proof of which uses a splitting of the summation over $n,a, b$ in smaller ranges.

\begin{lemma}\label{B1'}
There exists $C'>0$
 such that for all $R_0$
 \begin{equation*}
\sum_{n=\lceil c_1t\rceil}^{\lfloor t/\min\tau\rfloor}\sum_{a,b=0,...,\lfloor a_t^d\log(R_0)\rfloor\, :\, \max(a,b)\ge R_0}^{}
\mu(A'_{a,b,n,t}(\vec w_1,\vec w_2,K'))\le C'a_t^{-d}R_0^{-\frac 2{45}}+o(a_t^{-d})\, ,
\end{equation*}
as $t\rightarrow +\infty$.
\end{lemma}
\begin{proof}
Note that, by measure preserving and time reversal,
$(\widetilde\kappa\circ T^{-1},\widehat\Psi_n,\widetilde\kappa\circ T^n)$ has the same distribution as $(-\kappa\circ T^n,(-\widetilde\kappa_n,\widetilde\tau_n),-\widetilde\kappa\circ T^{-1})$ and so
\[
\mu(A'_{a,b,n,t}(\vec w_1,\vec w_2,K'))=\mu(A'_{b,a,n,t}(-\vec w_2,-\vec w_1,K''))
\]
with $K''=\{(\ell,r)\in\mathbb Z^d\times\mathbb R\, :\, (-\ell,r)\in K'\}$, with the notation~\eqref{eq:defApr}.
Thus, up to replacing $(a,b,\vec w_1,\vec w_2,K')$
by  $(b,a,-\vec w_2,-\vec w_1,K'')$, it is enough to prove that there exists $C''_0$ such that, for all
$R_0$,
 \begin{equation}\label{B1'0}
\sum_{n=\lceil c_1t\rceil}^{\lfloor t/\min\tau\rfloor}\sum_{a= R_0}^{\lfloor a_t^d\log(R_0)\rfloor}\sum_{b=0}^a
\mu(A'_{a,b,n,t}(\vec w_1,\vec w_2,K'))\le C_0''a_t^{-d}R_0^{-\frac 2{45}}+o(a_t^{-d}),\quad\mbox{as }t\rightarrow +\infty\, .
\end{equation}
The main ingredients of the proof of this are Lemma~\ref{lem:jointllt0} together with its Corollary~\ref{coro:llt0cor} and an argument similar to the one used in the proof of Lemma~\ref{lem:S1S2}.
To exploit Lemma~\ref{lem:jointllt0}, recall~\eqref{eq:defApr} and note  that
\begin{align}\label{eq:exp0}
\mu(A'_{a,b,n,t}(\vec w_1,\vec w_2,K'))
\le \mathbb E_{\mu}\left[\mathbf 1_{|\widetilde\kappa|\circ T^n\ge b}
\mathbf 1_{|\widetilde\kappa|\circ T^{-1}\ge a}
.h\left(
\widehat\Psi_n-(-\pi_d(a\vec w_1+b\vec w_2),
t-b-a-n\mu(\tau))\right)\right]\, ,
\end{align}
where $h:\Z^{d}\times\mathbb R\to (0,\infty)$ is the integrable function with compactly supported
Fourier transform given by $h(y_1,...,y_{d+1}):=K''\mathbf 1_{|y_1|,|y_d|\le K''}\frac{(1-\cos(y_{d+1}/K''))}{y_{d+1}^2}$ for some suitable $K''>0$.
Recall~\eqref{Deco2} and~\eqref{Deco1} and assume that $n>4 \lfloor K\log t\rfloor$. Let $k_n=k_t=\lfloor K\log t\rfloor$.  
It follows from~\eqref{eq:exp0} and Corollary~\ref{coro:llt0cor},
that, for any $\varepsilon_0\in\left(0,\frac 12-\frac 1{45}\right)$ and any $n=\lfloor c_1t\rfloor,...,t/\min\tau$,
\begin{align}\label{muA'}
&\mu(A'_{a,b,n,t}(\vec w_1,\vec w_2,K'))\ll \mu(|\widetilde\kappa|\ge a)\mu(|\widetilde\kappa|\ge b)
\widetilde Q^{(0)}_{n,a,b}(t)\\
\label{err1}&\quad+\mathcal O\left(
t^{-100}
+\log t\, a_t^{-d-2}
(\mu(|\widetilde\kappa|\ge a)\mu(|\widetilde\kappa|\ge b)^{\frac 12-\varepsilon_0}\right.
\\
\label{err2}&\quad
\quad \left.+a_t^{-d-2}\mu(|\widetilde\kappa|\ge b)
 \left\Vert \widehat\Psi_{2k_n} 1_{\{|\widetilde\kappa|
 	\circ T^{-1}
 	\ge a\}}\right\Vert_{L^1}\right)\, ,
\\ &
\mbox{with}\quad
\widetilde Q^{(0)}_{n,a,b}(t)
:=\mathbb E_{\mu}\left[h\left(\widehat\Psi_n-(-\pi_d(a\vec w_1+b\vec w_2),t-b-a-n\mu(\tau)
)\right)\right]\, .
\end{align}

{\bf{Estimating the term in the right hand side of~\eqref{muA'}.}}

We claim that, adapting carefully the argument of Lemma~\ref{lem:S1S2},
\begin{equation}\label{CCC0}
\sup_{a,b}\sum_{n=\lceil c_1 t\rceil}^{\lfloor t/\min\tau\rfloor}\widetilde Q^{(0)}_{n,a,b}(t)=\mathcal O(a_t^{-d})
\, ,
\end{equation}
which implies that
\begin{equation}\label{B1}
\sum_{n=\lceil c_1 t\rceil}^{\lfloor t/\min\tau\rfloor}\sum_{b\ge 0,a\ge R_0}\mu(|\widetilde\kappa|\ge a)\mu(|\widetilde\kappa|\ge b)
\widetilde Q^{(0)}_{n,a,b}(t)=
\mathcal O(a_t^{-d}R_0^{-1})\, .
\end{equation}
We prove the claim~\eqref{CCC0}.
Fix $L>0$.\\
First, it follows from Lemma~\ref{lem:jointllt0} that
\begin{align*}
\widetilde Q_{n,a,b}^{(0)}(t)&\ll 
a_t^{-d-1}\left(g_{d+1}\left(\frac{-\pi_d(a\vec w_1+b\vec w_2),t-b-a-n\mu(\tau)}{a_t}\right)+(\log t)^{-1}\right)
\end{align*}
and so that
\begin{align}
\nonumber\sum_{n\, :\, |t-b-a-n\mu(\tau)|<La_t\mu(\tau)}\widetilde Q_{n,a,b}^{(0)}(t)&\ll a_t^{-d}\int_{\mathbb R}g_{d+1}\left(\frac{-\pi_d(a\vec w_1+b\vec w_2)}{a_t},y\right)\, dy+ o(a_t^{-d}) \\
&\ll a_t^{-d}\Vert g_d\Vert_\infty+ o(a_t^{-d})\ll a_t^{-d}\, ,\label{CCC1}
\end{align}
with $g_d:=\int_{\mathbb R}g_{d+1}(\cdot,y)\, dy$.\\
Second, if $n>c_1t$ and $La_t< | n-(t-b-a)/\mu(\tau)|$, then, 
 as soon as $t$ large enough, it follows from Lemma~\ref{lem:lld}
 that, setting $v_{a,b,n}:=(-\pi_d(a\vec w_1+b\vec w_2),t-b-a-n\mu(\tau)
  )$, 
\begin{align}\nonumber
 \widetilde Q_{n,a,b}^{(0)}(t)&\ll \sum_{k_3\in\mathbb Z}\frac{1}{1+k_3^2} \mu\left(\widehat \Psi_n-v_{a,b,n}=(0,0,k_3)+\mathcal O(1)\right)\\
 &\ll \frac{t}{a_t^{d+1}}\sum_{k_3\in\mathbb Z}\frac{1}{1+k_3^2} \frac{\log(2+|v_{a,b,n}+k_3|)}{(2+|v_{a,b,n}+k_3|)^2}\, .\label{EEE1}
\end{align}
But, on the one hand,
\begin{align}\nonumber
\sum_{k_3:|v_{a,b,n}+k_3|\ge |v_{a,b,n}|/2}\frac{1}{1+k_3^2} \frac{\log(2+|v_{a,b,n}+k_3|)}{(2+|v_{a,b,n}+k_3|)^2}
&\ll \frac{\log(4+|v_{a,b,n}|)}{(1+|v_{a,b,n}|)^2}  \sum_{k_3}\frac 1{1+k_3^2}\\
&\ll \frac{\log(4+|v_{a,b,n}|)}{(1+|v_{a,b,n}|)^2}\, ,\label{EEE2}
\end{align}
and, on the other hand,
\begin{align}\nonumber
\sum_{k_3:|v_{a,b,n}+k_3|<|v_{a,b,n}|/2}\frac{1}{1+k_3^2} \frac{\log(2+|v_{a,b,n}+k_3|)}{(2+|v_{a,b,n}+k_3|)^2}
&\ll \sum_{u:|u|<|v_{a,b,n}|/2}\frac{1}{1+|u-v_{a,b,n}|^2} \frac{\log(2+|u|)}{(2+|u|)^2}\\
&\ll \frac{1}{|v_{a,b,n}|^2}\sum_{u:|u|<|v_{a,b,n}|/2} \frac{\log(2+|u|)}{(2+|u|)^2}\ll \frac{1}{|v_{a,b,n}|^2}\, .\label{EEE3}
\end{align}
It follows from~\eqref{EEE1},~\eqref{EEE2} and~\eqref{EEE3} that
\[
 \widetilde Q_{n,a,b}^{(0)}(t)\ll \frac{\log(2+|v_{a,b,n}|)}{(1+|v_{a,b,n}|)^2}\, ,
\]
as $\widetilde Q_n(t)$ in sublemma~\eqref{lem:1a}. Therefore
	\begin{align}
	\nonumber&\sum_{n>c_1t:La_t<  |n-(t-b-a)/\mu(\tau)|} \tQ^{(0)}_{n,a,b}(t) \\
	\nonumber&\quad\ll \frac{t}{a_t^{d+1}}
	\sum_{n: La_t<|n-(t-a-b)/\mu(\tau)|}
	\frac{|\log|n\mu(\tau)-(t-a-b)||}{1+|n\mu(\tau)-(t-a-b)|^2}\\ 
	&\quad\ll \frac{t}{a_t^{d+1}} \frac{\log(a_t)}{La_t}\ll \frac 1{La_t^d}\ll a_t^{-d}\, ,\label{CCC2}
	\end{align}
	since $\log u/u^2$ has primitive $-(1+\log u)/u$ and since $a_t^2=t\log t\sim 2t\log a_t$.
The claim~\eqref{CCC0} follows from~\eqref{CCC1} and~\eqref{CCC2}.

{\bf{Estimating the terms in~\eqref{err1}.}}
The first term leads to
\begin{equation}\label{eq:s2pre}
\sum_{n=\lceil c_1t\rceil}^{\lfloor t/\min\tau\rfloor}\sum_{a= R_0}^{\lfloor a_t^d\log R_0\rfloor}\sum_{b=0}^at^{-100}\ll t^{-99}a_t^{2d}(\log R_0)^2=o(a_t^{-d})\, .
 \end{equation}
It remains to estimate the contribution of the second part of~\eqref{err1}. 
Note that
\begin{align*}
\sum_{a\ge  R_0} \mu(|\widetilde\kappa|\ge a)\left(\sum_{b=0}^a\mu(|\widetilde\kappa|\ge b)^{\frac 12-\varepsilon_0}  \right)
&\ll \sum_{a\ge  R_0} a^{-2}\left(\sum_{b=0}^a(b+1)^{-1+2\varepsilon_0} \right)\\ 
&\ll  \sum_{a\ge  R_0} a^{-2+2\varepsilon_0}=R_0^{-1+2\varepsilon_0}\, ,
\end{align*}
since $\varepsilon_0<\frac 12$.
Since we also know that $t\log t=a_t^2$, we obtain that
\begin{align}\label{eq:es2}
 \sum_{n=\lceil c_1 t\rceil}^{\lfloor t/\min\tau\rfloor}\sum_{a\ge  R_0}\sum_{b=0}^a\log t\,a_t^{-d-2}\mu(\widetilde\kappa\ge a)\mu(\widetilde\kappa\ge b)^{\frac 12-\varepsilon_0}
 \ll a_t^{-d}R_0^{-(1-2\varepsilon_0)}\, ,
\end{align}
with $1-2\varepsilon_0>0$.

{\bf{Estimating the term in~\eqref{err2}}}
 We claim that
 \begin{equation}\label{CCC4}
\sum_{a\ge R_0}
\left\Vert \widehat\Psi_{2k_n} 1_{\{|\widetilde\kappa|\circ T^{-1}\ge a\}}\right\Vert_{L^1}\ll  R_0^{-\frac 2{45}}\log t\, .
\end{equation}
 Since we also know that $t\log t=a_t^2$ and  that
 $\sum_{b\ge 0}\mu(|\widetilde\kappa|\ge b)=\mathbb E[|\widetilde\kappa|]<\infty$,
 we obtain that
 \begin{align}\label{eq:es3}
 \sum_{n=\lceil c_1 t\rceil}^{\lfloor t/\min\tau\rfloor}\sum_{a\ge R_0 }
 \sum_{b\ge 0}a_t^{-d-2}\mu(|\widetilde\kappa|\ge b)
 \left\Vert \widehat\Psi_{2k_n} 1_{\{|\widetilde\kappa|
 	\circ T^{-1}
 	\ge a\}}\right\Vert_{L^1}
 =\mathcal O\left(
a_t^{-d}R_0^{-\frac 2{45}}\right).
\end{align}
 
We now prove the claim~\eqref{CCC4}.
First, compute that
\begin{align*}
\sum_{a\ge  R_0 }
&
 \left\Vert \widehat\Psi_{2k_n} 1_{\{|\widetilde\kappa|\circ T^{-1}\ge a\}}\right\Vert_{L^1}\le
\sum_{\ell=0}^{2k_n-1} \sum_{a\ge R_0}
\sum_{a'\ge a}
\sum_{b'\ge 1}
 b'\mu(|\widetilde\kappa|\circ T^{-1}=a',|\widetilde\kappa|\circ T^\ell=b')
  \\
& \le
 \sum_{\ell=1}^{2k_n} \sum_{a'\ge R_0}\sum_{a=R_0}^{a'}
 \sum_{b'\ge 1}
 b'\mu(|\widetilde\kappa|=a',|\widetilde\kappa|\circ T^\ell=b')\\
& \le
\sum_{\ell=1}^{2k_n} \sum_{a'\ge R_0}\sum_{b'\ge 1}
a'b'\mu(|\widetilde\kappa|=a',|\widetilde\kappa|\circ T^\ell=b')\, ,
\end{align*}
where the sum over $a',b'$ is taken over the positive real numbers (non necessarily integer) such that the summand is non null.

We claim that, uniformly in $n$ and in  $\ell=1,...,2k_n$,  
\begin{align}\label{esss}
\sum_{a'\ge R_0}\sum_{b'\ge 1}
a'b'\mu(|\widetilde\kappa|=a',|\widetilde\kappa|\circ T^\ell=b')\ll R_0^{-\frac 2{45}}
\end{align}
The previous two displayed equations give the claim~\eqref{CCC4}.

It remains to prove the claim~\eqref{esss}.
We proceed via considering all relevant cases of $a',b'$.

{\bf{Case 1: Contribution of the $a',b'$ such that $b'^{\frac 45}\le a'\le b'$.}}
\begin{align*}
&\sum_{b'\ge 1}
\sum_{a'\in[\max(R_0,(b')^{4/5});b']}  a'b'\mu(|\widetilde\kappa|=a',|\widetilde\kappa|\circ T^\ell=b')\\
&\le \sum_{b'\ge R_0}
b'\mathbb E_\mu\left[\sum_{a'\in[ (b')^{4/5} ;b']}a' \mathbf 1_{|\widetilde\kappa|=a'}\mathbf 1_{|\widetilde\kappa|\circ T^\ell=b'}\right]\\
&\le \sum_{b'\ge R_0}
|b'|^2
\mu(|\widetilde\kappa|\ge (b')^{\frac 45},|\widetilde\kappa|\circ T^\ell=b')\, ,
\end{align*}
where in the last equation we used that
\begin{equation}\label{MAJO}
\sum_{a'\in[a_-;a_+]}a' 1_{\{|\widetilde\kappa|=a'\}}\le a_+ 1_{\{|\widetilde\kappa|\ge a_-\}}\, .
\end{equation}
Thus applying \eqref{controltimelog} with $V:=100/a_0$, we obtain
\[
\mu(|\widetilde\kappa|\ge (b')^{\frac 45},|\widetilde\kappa|\circ T^\ell=b')\ll |b'|^{-3-\frac 2{45}}\, ,
\]
uniformly in $\ell\le\frac {100\log(2+|b'|)}{a_0}$, and it follows from~\eqref{Deco2} combined with~\eqref{tailkappa} that
\[
\mu(|\widetilde\kappa|\ge (b')^{\frac 45},|\widetilde\kappa|\circ T^\ell=b')
\le \mu(|\widetilde\kappa|\ge (b')^{\frac 45})\mu(|\widetilde\kappa|\circ T^\ell=b')+C'e^{-a_0\ell}\ll  |b'|^{-3-\frac 2{45}}
\]
uniformly in $\ell\ge\frac {100\log(2+|b'|)}{a_0}$.
Therefore
\begin{align}
\sum_{a'\in[ \max(R_0,(b')^{4/5});b']}  a'b'&\mu(|\widetilde\kappa|=a',|\widetilde\kappa|\circ T^\ell=b')
\ll \sum_{b'\ge  R_0 }|b'|^2\, |b'|^{-3-\frac 2{45}}\\
&\ll \sum_{b'\ge  R_0}  |b'|^{-1-\frac 2{45}}
\ll R_0^{-\frac 2{45}}\, .\label{CCC4a}
\end{align}

{\bf{Case 2: Contribution of the $a',b'$ such that $(a')^{\frac 45}<b'<a'$.}}
\begin{align}
\nonumber&\sum_{b'\ge 1}\sum_{a'\in[\max(R_0,b')
; (b')^{\frac 54}]}a'b'\mu(|\widetilde\kappa|=a',|\widetilde\kappa|\circ T^\ell=b')\\
\nonumber&\le \sum_{a'\ge R_0}
\sum_{b' \in[(a')^{\frac 45};a']}
a'b'\mu(|\widetilde\kappa|=a',|\widetilde\kappa|\circ T^\ell=b')\\
\nonumber&\le \sum_{a'\ge R_0}
|a'|^2 \mu(|\widetilde\kappa|=a',|\widetilde\kappa|\circ T^\ell>(a')^{\frac 45})\\
&\ll \sum_{a'\ge R_0}
  |a'|^2 |a'|^{-3-\frac 2{45}}\ll \sum_{a'\ge R_0}
  |a'|^{-1-\frac 2{45}}\ll R_0^{-\frac 2{45}}\, ,\label{CCC4b}
\end{align}
using again~\eqref{MAJO} and~\eqref{controltimelog} again with $V=100/a_0$
when $\ell\le V\log(2+|a')$ and~\eqref{Deco2} otherwise.

 {\bf{Case 3: Contribution of the $a',b'$ such that $a'<(b')^{\gamma}<b'$ for some $\gamma\in(0,1)$ (e.g. $\gamma=\frac 45$).}}
 
\begin{align}
\nonumber&\sum_{b'\ge 1}
\sum_{a'\in[R_0;(b')^{\gamma}]}a' b'\mu(|\widetilde\kappa|=a',|\widetilde\kappa|\circ T^\ell=b')\\
\nonumber&\le \sum_{b'\ge R_0^{\frac 1\gamma}}
\sum_{a'\in[R_0;(b')^{\gamma}]}a' b'\mu(|\widetilde\kappa|=a',|\widetilde\kappa|\circ T^\ell=b')\\
\nonumber&\le \sum_{b'\ge R_0^{\frac 1\gamma}}
 |b'|^\gamma b'\mu(|\widetilde\kappa|\ge R_0,|\widetilde\kappa|\circ T^\ell=b')\\
&\ll \sum_{b'\in Supp(|\widetilde\kappa|): b'\ge R_0^{\frac 1\gamma}}
(b')^{\gamma} |b'|\, |b'|^{-3}
\ll R_0^{\frac 1\gamma(-1+\gamma)}\, ,\label{CCC4c}
\end{align}
using~\eqref{tailkappa}.

{\bf{Case 4: Contribution of the $a',b'$ such that $b'<(a')^{\gamma}<a'$ with $\gamma\in(0,1)$ (e.g. $\gamma=\frac 45$).}}

\begin{align}
\nonumber&\sum_{a'\ge R_0}\sum_{b'\in[1;(a')^{\gamma}]}a' b'\mu(|\widetilde\kappa|=a',|\widetilde\kappa|\circ T^\ell=b')
\nonumber\le \sum_{a'\ge R_0}(a')^{\gamma} |a'| \mu(|\widetilde\kappa|=a')\\
\quad\quad&\ll \sum_{a'\in Supp(|\widetilde\kappa|):a'\ge R_0}(a')^{\gamma} |a'| |a'|^{-3}
\ll R_0^{-1+\gamma}\, .\label{CCC4d}
\end{align}

The claim~\eqref{esss} follows from~\eqref{CCC4a},~\eqref{CCC4b},~\eqref{CCC4c} and~\eqref{CCC4d}, ending the proof of~\eqref{CCC4} and so of~\eqref{eq:es3}. 
Estimate~\eqref{B1'0} and so the lemma then follows from~\eqref{muA'},~\eqref{B1},~\eqref{eq:s2pre},~\eqref{eq:es2} and~\eqref{eq:es3}
~\end{proof}

 \subsubsection{Concluding the proof of Theorem~\ref{lem:tight}}
 
 The conclusion follows from Lemmas~\ref{lem:B0} (and the comment thereafter),~\ref{B1second},~\ref{lem:decomp} and~\ref{B1'}.

 \section{Proof of joint CLT (Lemma~\ref{lem:clt})}
 \label{sec:proofjCLT}
Let $d\in\{0,1,2\}$. 
In this section we show that arguments established in~\cite{BalintGouezel06} and~\cite{PeneTerhesiu21} can be adapted to the study of $\widehat\Psi$ instead of $\kappa$. The main idea comes down to a basic observation, namely that
$\widehat\Psi$ can be written as the sum of a vector in 
$\Z^{d+1}$ that 'behaves like' $\kappa$ and of a bounded function. The mentioned vector
in $\Z^{d+1}$ is precisely
$\left(\kappa,|\widetilde\kappa|-\mathbb E_{\mu}[|\widetilde\kappa|]\right)$ which, as $\kappa$, is constant on good sets and has a similar tail probability.
In particular, the distribution of $\widehat\Psi$ is in the domain of a nonstandard CLT with normalization
$\sqrt{n\log n}$. The details are provided around equation~\eqref{eq:rewr} below.

 We will prove the convergence in distribution of $(\widehat \Psi_n/\sqrt{n\log n})_n$ by establishing the pointwise convergence of its characteristic function, with the use of Fourier perturbed operator on the quotient tower
constructed by Young in~\cite{Young98}(see \cite{Chernov99}) as Sz\'asz and Varj\'u did in~\cite{SV07} to establish the CLT and LLT for $\kappa$. 
It follows from~\eqref{coboundPsi} that
$\tau=|V|=|\widetilde \kappa|+\mathcal O(1)$, where $\widetilde\kappa:M\rightarrow \mathbb Z^2$ is the cell change in the $\mathbb Z^2$-periodic Lorentz gas (see 
Subsection~\ref{notationtightness}). \\
We have already recalled several properties of $\widetilde\kappa$. Let us recall, in particular, the precise tail of
$\widetilde\kappa$ (this is partially recalled in~\eqref{tailkappa}).
 By \cite{SV07} completed by \cite{PeneTerhesiu21}, there exist $L_0>0$ and a finite set $\mathcal E$ made of $(L,w)\in(\mathbb Z^d)^2$
with $w$ prime such that 
\begin{align}\label{eq:t1}
|\widetilde\kappa|>L_0\quad\Rightarrow \quad \exists (L,w)\in\mathcal E,\ \exists N\in\mathbb N^*\ \widetilde\kappa=L+Nw
\end{align}
and
\begin{align}\label{eq:tail}
\mu(\widetilde\kappa=L+Nw)=c_{L,w}N^{-3}+\mathcal O(N^{-4})\, ,\quad\mbox{as }N\rightarrow +\infty\, ,
\end{align}
with $c_{L,w}>0$. 
This set $\mathcal E$ parametrizes the set of corridors mentioned in Section~\ref{sec:proofmixing} (the set $\mathfrak C$ therein corresponds to the set of unit vectors proportional to some $w$ such that there exists $L\in\mathbb Z^2$ such that $(L,w)\in\mathcal E$). 
Then, when $d=2$, the variance matrix $\Sigma_0$ (for the Sinai billiard map) appearing in the Central Limit Theorem for the displacement given by~\eqref{eq:cltV} corresponds to the following quadratic form
\begin{equation}\label{defSigma0dim2}
\mbox{If }d=2,\quad
\forall t\in\mathbb R^2,\quad
\langle\Sigma_0 t,t\rangle:=\frac 12\sum_{(L,w)\in\mathcal E}c_{L,w}\langle t,w\rangle^2\, .
\end{equation}
It is not degenerate since, when $d=2$, we assume the existence of at least two non parallel corridors, and so of two non parallel $w,w'$ such that there exists $L,L'\in\mathbb Z^2$ such that
$(L,w),(L',w')\in\mathcal E$.
 
When $d=1$, setting $\pi_1(w_1,w_2)=w_1$,  $\Sigma_0$ is given by the formula
\begin{equation}\label{defSigma0dim1}
\mbox{If }d=1,\quad \Sigma_0:=\frac 12 \sum_{(L,w)\in\mathcal E}c_{L,w}(\pi_1(w))^2
\end{equation}
 which is non null since we assumed the existence of at least an unbounded line touching no obstacle.

We recall that the variance matrix $\Sigma$ for the flow appearing in~\eqref{eq:Sigma} is given by $\Sigma=\Sigma_0/\sqrt{\mu(\tau)}$.\\

The variance matrix $\Sigma_{d+1}$ of the limit of $a_n^{-1}\widehat\Psi_n$
will appear to be given by the following pretty similar formula:
\begin{equation}\label{Sigmad+1}
\forall t\in\mathbb R^{d+1},\quad
\langle\Sigma_{d+1} t,t\rangle:=\frac 12\sum_{(L,w)\in\mathcal E}c_{L,w}\langle t,(\pi_d(w),|w|)\rangle^2\, ,
\end{equation}
with, as in Section~\ref{notationtightness}, 
\[
\forall w\in\mathbb Z^2,\quad \pi_2(w)=w\quad \mbox{and}\quad \pi_1(w_1,w_2)=w_1\, ,\] 
and with the convention \[
\forall (w,z)\in\mathbb Z^2\times\mathbb R,\quad(\pi_0(w),|w|)=w\quad\mbox{and more generally}\quad (\pi_0(w),z)=z .
\]

Throughout this section, we fix some (arbitrary)  $q\in [1,2)$, and some $b_q>2$ so that 
\begin{equation}\label{defbq}
\frac{1}{b_q}+\frac 1q<1\, .
\end{equation}
This choice will determine the choice of the Banach space on the Young tower.

\subsection{Expression of the characteristic functions via Fourier Perturbed operator}
We observe that
\[|\hPsi(x)-\hPsi(y)|\le d(x,y)+d(T(x),T(y))\, ,\]
for any $x,y$ in the same connected component of $M\setminus (\mathcal S_0\cup T^{-1}(\mathcal S_0))$, where $\mathcal S_0$
is the set of post-collisional vectors tangent to $\partial\Omega$. 
We recall that the diameter of the connected components of 
$M\setminus \bigcup_{k=-n}^nT^{-k}(\mathcal S_0)$ is $\mathcal O(\beta_1^n)$ for some $\beta_1\in(0,1)$.

As in~\cite{SV07}, we consider the towers constructed by Young
in \cite{Young98} (see also \cite{Chernov99}). We recall some facts on Young towers
and introduce some notations that we shall use in the remainder of this paper.
We let $(\Delta,f_\Delta,\mu_{\Delta})$ be the hyperbolic tower, which is an extension of $(M,T,\mu)$ by $\pi:\Delta\mapsto M$
(with $\pi(x,\ell)=T^\ell(x)$) and write
 $\left(\overline\Delta,f_{\overline\Delta},\mu_{\overline\Delta}\right)$
for the quotient tower (obtained from $\Delta$ by quotienting 
out the stable manifolds).
The quotient tower is identified with $\overline\Delta:=\{(x,\ell)\in\Delta\, :\, x\in\overline Y\}$, where 
$\overline Y$ is an unstable curve of a well chosen set $Y\subset M$, and write  $\overline\pi:\Delta\rightarrow\overline\Delta$ for the projection
corresponding to the holonomy along the stable curves of $\Delta$.
The dynamical system $(\Delta,f_\Delta,\mu_{\Delta})$ is given by
\begin{itemize}
	\item The space $\Delta$ is the set of couples $(x,\ell)\in Y\times\mathbb N_0$ such that $\ell<R(x)$, where $R$ is a return time to $Y$.
	\item The map $f_\Delta$ is given by $f_\Delta(x,\ell)=(x,\ell+1)$
	if $\ell<R(x)-1$ and $f_\Delta(x,R(x)-1)=(T^{R(x)}(x),0)$.
	\item The probability measure $\mu_{\Delta}$ is given by $\mu_{\Delta}(A\times\{\ell\})=\mu_\Delta(A\cap \{R>\ell\})/\mathbb E_\mu[R.1_Y]$, for any measurable set $A\subset Y$.
\end{itemize}
We assume that the greatest common divisor (g.c.d.) of $R$ is $1$, which can be done because of total ergodicity\footnote{The idea of using the total ergodicity of $T$ for constructing a new tower with $g.c.d.(R)=1$ was suggested in~\cite[Section 4]{Young98} 
and used in~\cite{SV04} for ensuring aperiodicity of the version of $\kappa$
on $\overline\Delta$. The details of such a tower construction are contained in \cite[Appendix B]{FP09AIHP}.}
of $T$ ; this assumption is not essential, since one can also deal directly with  $g.c.d.(R)\ne 1$, but it helps simplifying the proofs and notation throughout the remainder of this paper.

The partition $\mathcal P$ on $\Delta$ consists of a union of partitions of the different levels which become finer and finer  as one goes up in the tower. The partition
$\mathcal P$ is used to define a separation time $s(\cdot,\cdot)$
on $\Delta$:
\[
s(x,y):=\inf\{n\ge -1\, :\, \mathcal P(f_\Delta^{n+1}(x))\ne \mathcal P(f_\Delta^{n+1}(y))\}\, ,
\]
The separation time $s(x,y)$ satisfies the following property:
$\pi(x)$ and $\pi(y)$ are in the same connected component of
$M\setminus (\mathcal S_0\cup T^{-1}(\mathcal S_0))$ if $s(x,y)\ge 0$.
In particular if $s(x,y)>2n$, then $\pi(f_\Delta^n(x))$ and $\pi(f_\Delta^n(y))$ are in the same connected component of $M\setminus\bigcup_{k= -n}^nT^{-k}(\mathcal S_0)$. 
Since the atoms of the partition $\mathcal P$ are unions of stable curves, this separation time has a direct correspondent $\overline s(\cdot,\cdot)$ on the quotient tower $\overline\Delta$. 
Let $P$ be the transfer operator of $ \left(\overline\Delta,f_{\overline\Delta},\mu_{\overline\Delta}\right)$, i.e. $P$ is defined on $L^1(\mu_{\overline\Delta})$ by
\[
\int_{\overline\Delta}H.P(G)\, d\mu_{\overline\Delta}=\int_{\overline\Delta}H\circ f_{\overline\Delta}.
G\, d\mu_{\overline\Delta}\, .
\]
Let $\beta\in(\beta_1^{\frac 14},1)$ and close enough to 1. 
It follows from~\cite{Young98} and~\cite{Chernov99} that there exists
$\varepsilon'>0$ such that, for all $\varepsilon\in]0,\varepsilon'[$,  $P$ is quasicompact on the Banach space $\mathcal B=\mathcal B_{\varepsilon}$ corresponding to the set of
functions of the form $e^{\varepsilon\omega}H$, with $H\in\mathcal B_0$,
where $\omega(x,\ell)=\ell$ and where $\mathcal B_0$ is the Banach space of bounded functions $H:\overline\Delta\rightarrow \mathbb C$ that are Lipschitz continuous with respect to the ultrametric $\beta^{\overline s(\cdot,\cdot)}$ (the space $\mathcal B_0$ corresponds to the space $\mathcal B_\varepsilon$ when $\varepsilon=0$). The space $\mathcal B$
is then endowed with the norm $\Vert\cdot\Vert_{\mathcal B}$ given by
\begin{equation}\label{eq:Bo}
 \Vert H\Vert_{\mathcal B}=\Vert e^{-\varepsilon\omega}H\Vert_{\mathcal B_0}\, .
\end{equation}
Recall $b_q$ satisfies~\eqref{defbq}.
Choose $\varepsilon$ small enough so that $e^{\varepsilon \omega}\in \mathbb L^{b_q}(\mu_{\overline\Delta})$ which implies that $\mathcal B$ 
is continuously embedded in ${\mathbb L}^{b_q}(\mu_{\overline\Delta})$ since
\begin{equation}\label{fixb0}
 \Vert H\Vert_{L^{b_q}}\le \Vert e^{\varepsilon\omega}\Vert_{L^{b_q}}\Vert e^{-\varepsilon\omega}H\Vert_{\infty}\le\Vert e^{\varepsilon\omega}\Vert_{L^{b_q}}\Vert H\Vert_{\mathcal B} \, .
\end{equation}
(This particular choice of $b_q$ will be used in the proof of Sublemma~\ref{sunl:pi} below.)
Since we assume that $g.c.d.(R)=1$, 1 is the only (dominating) eigenvalue of modulus 1 of $P$, and it is simple and isolated in the spectrum of $P$. In particular,
there exists $\theta\in(0,1)$ (depending on $(\beta,\varepsilon)$) such that
\begin{equation}\label{quasicompactness}
\Vert P^n-\mathbb E_{\mu_{\overline\Delta}}[\cdot]1_{\overline\Delta}\Vert_{\mathcal L(\mathcal B)}=\mathcal O(\theta^n)\, ,\quad\mbox{as }n\rightarrow +\infty\, .
\end{equation}
Let $t\in\mathbb R^{d+1}$. 
Recall that via~\cite[eq. (6.2) verified in Corollary 9.4]{BBM19}, 
\begin{align}\label{coboundPsibis}
\hPsi\circ\pi=\overline\Psi\circ\overline\pi+\chi\circ f_\Delta-\chi\, ,
\end{align}
with $\overline\Psi=(\overline\kappa,\overline\tau):\overline\Delta\rightarrow \mathbb Z^d\times\mathbb R$, 
where $\overline\tau$ is the version of $\widetilde\tau=\tau-\mu(\tau)$ on $\overline\Delta$ given
by
\[ 
\overline\tau:=\widetilde\tau\circ\pi +\sum_{n\ge 1}\left(\tau\circ\pi\circ f_{\Delta}^n-\tau\circ\pi\circ f_{\Delta}^{n-1}\circ f_{\overline\Delta}\right)
\]
and where $\chi=(\mathbf{0},\chi_0)$ with $\mathbf 0$ the null element of $\mathbb Z^d$ and with
\[
\chi_0:=\sum_{n\ge 0}\left(\tau\circ\pi\circ f_\Delta^n-\tau\circ\pi\circ f_\Delta^{n}\circ \overline\pi\right)\, .
\]
By~\cite[Proof of Lemma 8.3]{BBM19} (see also Section~\ref{sec:app}) $\overline\tau$ 
is locally Lipschitz continuous (on each atom of Young's partition) with respect to the ultrametric $\beta^{s(\cdot,\cdot)}$, and $\chi:\Delta\rightarrow \{0\}^d\times\mathbb R$ is bounded and 
Lipschitz in the following sense: 
\begin{equation}\label{regbartau}
\sup_{k\ge 1}
\sup_{x,y:s(x,y)>2k}\frac{|\chi(f_\Delta^k(x))-\chi(f_\Delta^k(y))|}{\beta^k}<
\infty\, ,
\end{equation}
It follows from the coboundary equation~\eqref{coboundPsibis} that
\[
\mathbb E_\mu[e^{i\langle t,\frac{\widehat \Psi_n}{a_n
}\rangle}]
=   \mathbb E_{\mu_{\Delta}}\left[e^{-i\langle \frac{t}{a_n
	},\chi\rangle} e^{i\langle \frac{t}{a_n
},\overline\Psi_n\rangle\circ\overline\pi}e^{i\langle \frac{t}{a_n},\chi\circ f_\Delta^n\rangle}\right]\, .
\]
Let $\overline K:\overline\Delta\rightarrow\mathbb Z^2$
be the version of $\widetilde\kappa$  on $\overline\Delta$, i.e. the function such that
$\overline K\circ\overline\pi=\widetilde\kappa\circ\pi$. 
It follows from~\eqref{coboundPsi} and~\eqref{coboundPsibis} that
the function $\Theta:\overline\Delta\to\mathbb R^{d+1}$ defined by
\begin{align}\label{eq:rewr}
\Theta:=\overline\Psi-\Upsilon
,\quad\mbox{with }
\Upsilon:=\left(\pi_d(\overline K),|\overline K|-\mathbb E_{\mu_{\overline \Delta}}[|\overline K|]\right)\, ,
\end{align}
is bounded and Lipschitz, and $\Upsilon$ is constant on partition elements (as  $\overline\kappa=\pi_d(\overline{K})$ corresponding to the cell change). Since both $\overline\Psi$ and $\Upsilon$ have mean zero,
$\mathbb E_{\mu_{\overline \Delta}}[\Theta]=0$.

We define the Fourier-perturbed operators $P_t,\widetilde P_t\in\mathcal L(\mathcal B)$ by  $P_tv=P(e^{i\langle t,\overline\Psi\rangle} v),\widetilde P_tv=P(e^{i\langle t,\Upsilon\rangle} v)$ for $t\in\R^{d+1}$.
By~\cite{SV07} (which exploits~\cite{Young98,Chernov99}), up to enlarging the value of $\theta\in(0,1)$ appearing in~\eqref{quasicompactness}, there exist $\beta_0\in(0,\pi]$, 
a continuous function $t\mapsto \lambda_t\in\mathbb C$  and two families of operators $(\Pi_t)_t$ and $(U_t)_t$
acting on $\mathcal B$ such that $t\mapsto \Pi_t\in\mathcal L(\mathcal B,L^1(\overline\Delta))$ is continuous and such that, for every
$t\in[-\beta_0,\beta_0]^{d+1}$ and every positive integer $n$,
\begin{equation}
\label{spgap-Sz}
P_t^n=\lambda_t^n\Pi_t+U_t^n\, ,\quad \widetilde P_t^n=\widetilde \lambda_t^n\widetilde\Pi_t+\widetilde U_t^n\, ,
\end{equation}
\begin{equation}
\label{spgap-Sz-bis}
\mbox{with}\quad\quad\quad\quad\sup_{t\in[-\beta_0,\beta_0]^d}\left(\Vert U_t^n\Vert_{\mathcal B}+\Vert \widetilde U_t^n\Vert_{\mathcal B}\right)
=\mathcal O\left(\theta^n\right)\, ,\quad\mbox{as }n\rightarrow +\infty\, .
\end{equation}
Set $k=k_n:=(\log n)^2$ and $F_{k,u}(x):=e^{i\langle u,\overline\Psi_k(\overline\pi(x))+\chi\circ f_\Delta^k(x)\rangle}$. 
It follows from~\eqref{coboundPsibis} that
\begin{align*}
\mathbb E_\mu[e^{i\langle t,\frac{\widehat \Psi_n}{a_n}\rangle}]
&=\mathbb E_{\mu_{\Delta}}\left[e^{-i\langle \frac{t}{a_n},\chi\circ f_\Delta^k\rangle} e^{i\langle \frac{t}{a_n},\overline\Psi_n\rangle\circ\overline\pi\circ f_\Delta^k}e^{i\langle \frac{t}{a_n},\chi\circ f_\Delta^{n+k}\rangle}\right]\\
&=\mathbb E_{\mu_{\Delta}}\left[F_{k,-\frac{t}{a_n}} e^{i\langle \frac{t}{a_n},\overline\Psi_n\rangle\circ\overline\pi}
F_ {k,\frac{t}{a_n}}\circ f_\Delta^n\right]\, .
\end{align*}
We approximate $F_{k,u}(x)$
by its conditional expectation $\widehat F_{k,u}(x)=\overline F_{k,u}(\overline\pi(x))$ on the set $\{y\in\Delta\, :\, s(x,y)>2k\}$, where $s$ is the separation time on $\Delta$ as recalled earlier in this section.
 Since $e^{i\langle u,\overline\Psi\rangle}$ is bounded and Lipschitz on $\overline\Delta$ and since  $e^{i\langle t,\chi\rangle}$
is bounded and Lipschitz on $\Delta$ (in the sense of \eqref{regbartau}), it follows that
\begin{align*}
\mathbb E_\mu[e^{i\langle t,\frac{\widehat \Psi_n}{a_n}\rangle}]
&=\mathbb E_{\mu_{\overline\Delta}}\left[\overline F_{k,-\frac{t}{a_n}} e^{i\langle \frac{t}{a_n},\overline\Psi_n\rangle}
\overline F_ {k,\frac{t}{a_n}}\circ f_{\overline\Delta}^n\right]+\mathcal O(\beta^k)\\
&=\mathbb E_{\mu_{\overline\Delta}}\left[\overline F_{k,\frac{t}{a_n}} P_{\frac{t}{a_n}}^n(
\overline F_ {k,-\frac{t}{a_n}})\right]+\mathcal O(\beta^k)\\
&=\lambda_{t/a_n}^{n-2k}\mathbb E_{\mu_{\overline\Delta}}\left[\overline F_{k,t} \Pi_{\frac{t}{a_n}}(
P_{\frac{t}{a_n}}^{2k}(\overline F_ {k,-\frac{t}{a_n}}))\right]+\mathcal O(\beta^k+\theta^{n-2k})\, ,
\end{align*}
as $n\rightarrow +\infty$.
Furthermore $\Vert \overline F_ {k,u}\Vert_{\infty}\le 1 $ and $ P_{\frac{t}{a_n}}^{2k}(\overline F_ {k,-\frac{t}{a_n}})$ are uniformly (in $k,n$) Lipschitz
with respect to Young's ultrametric $\beta^{s(\cdot,\cdot)}$.
Thus by continuity of  $t\mapsto \Pi_t\in\mathcal L(\mathcal B,L^1(\overline\Delta))$, 
\begin{align*}
\mathbb E_\mu[e^{i\langle t,\frac{\widehat \Psi_n}{a_n}\rangle}]
=\lambda_{\frac{t}{a_n}}^{n-2k}\left(\mathbb E_{\mu_{\overline\Delta}}\left[\overline F_{k,\frac{t}{a_n}}\right] \mathbb E_{\mu_{\overline\Delta}}
\left[P_{\frac{t}{a_n}}^{2k}(\overline F_ {k,-\frac{t}{a_n}}))\right]+o(1)\right)\, ,
\end{align*}
as $n\rightarrow +\infty$. 
Observe that, for every $t\in\mathbb R^{d+1}$,
\begin{align*}
\mathbb E_{\mu_{\overline\Delta}}\left[\overline F_{k,\frac{t}{a_n}}\right]
&=\mathbb E_{\mu_\Delta}\left[e^{i\langle \frac{t}{a_n},\overline\Psi_k\circ\overline\pi+\chi\circ f_\Delta^k\rangle}\right]=1+ o(1)\quad \mbox{as }n\rightarrow +\infty\, ,
\end{align*}
due to the dominated convergence theorem (using the fact that $\lim_{k\rightarrow +\infty}\overline\Psi_k/k=0$  $\mu_{\overline\Delta}$-almost-surely).
Analogously
\[
\forall t\in\mathbb R^{d+1},\quad \mathbb E_{\mu_{\overline\Delta}}
\left[P_{\frac{t}{a_n}}^{2k}(\overline F_ {k,-\frac{t}{a_n}}))\right]
= \mathbb E_{\mu_\Delta}\left[e^{i\langle \frac{t}{a_n},\overline\Psi_k\circ\overline\pi-\chi\rangle\circ f_\Delta^k}\right]=1+ o(1)\, ,\quad\mbox{as }n\rightarrow +\infty\, .
\]
Thus 
\begin{align}\label{eq:Four}
\mathbb E_\mu[e^{i\langle t,\frac{\widehat \Psi_n}{a_n}\rangle}]
=\lambda_{\frac{t}{a_n}}^{n-2k}+o(1)\, .
\end{align}
An important observation that will allow us to adapt the results of~\cite{PeneTerhesiu21} to the present context is that
\[
P_t-\widetilde P_t=P\left(\left(e^{i\langle t,\overline\Psi\rangle}-e^{i\langle t,\Upsilon\rangle}\right)\cdot\right)
=\widetilde P_t\left(\left(e^{i\langle t,\Theta\rangle}-1\right)\cdot\right)\, .
\]
\subsection{Regularity of the dominating eigenvalues and its spectral projector}
In this part, we prove that for $q\in[1,2)$
chosen before~\eqref{defbq}, 
\[\Vert \Pi_t-\Pi_0\Vert_{\mathcal B\rightarrow L^q(\mu_{\overline\Delta})}=\mathcal O(t)\quad \mbox{and}\quad
\lambda_t=1-\log(1/|t|)\langle t,\Sigma_{d+1} t\rangle+\mathcal O(t^2)\, ,
\]
as $t\rightarrow 0$.
We do so, via the following several steps: we first establish in Sublemma~\ref{sub:asl} an equivalent of $\lambda_t-1$, and we use it to establish in Sublemma~\ref{sunl:pi} the announced estimate of $\Vert \Pi_t-\Pi_0\Vert_{\mathcal B\rightarrow L^q(\mu_{\overline\Delta})}$ that we finally use to establish in Sublemma~\ref{sunl:lamb} the announced expansion of $\lambda_t$.

To obtain such estimates, we will control the error between
$\lambda_t$ and $\widetilde\lambda_t$, and between
$\Pi_t$ and $\widetilde\Pi_t$. Here we crucially exploit that
$\Upsilon$ and  $\kappa$ satisfy similar properties. This allows us to adapt some results obtained for $\kappa$ in~\cite{SV07,PeneTerhesiu21} with the use of~\cite{BalintGouezel06}.

We start by studying $P_t-\widetilde P_t$. 
Observe that $(e^{i\langle t,\overline\Theta\rangle}-1)\cdot\in\mathcal L(\cB)$ is dominated by  $\left\Vert e^{i\langle t,\overline\Theta\rangle}-1\right\Vert_{\mathcal B_0}$. This implies that $\left\Vert\widetilde P_t-P_t\right\Vert_{\mathcal L(\cB)}=\mathcal O(|t|)$, and thus that
\begin{equation}\label{Pit-tildePit}
\left\Vert \Pi_t-\widetilde\Pi_t\right\Vert_{\mathcal B}=\mathcal O(t)\, ,
\end{equation}
using the usual Cauchy integral expression for $\Pi_t$ and $\widetilde \Pi_t$. 
In particular, the announced estimate on $\Vert \Pi_t-\Pi_0\Vert_{\mathcal B\rightarrow L^q(\mu_{\overline\Delta})}$ will follow from the same estimate for $\Vert \widetilde\Pi_t-\widetilde\Pi_0\Vert_{\mathcal B\rightarrow L^q(\mu_{\overline\Delta})}$. 

Lemma~\ref{lem:clt} follows immediately  from~\eqref{eq:Four}, combined with the continuity of $t\mapsto\Pi_t\in\mathcal L(\mathcal B\rightarrow L^1(\mu_{\overline\Delta}))$
and from the first sublemma below.

 \begin{sublemma}\label{sub:asl}  
 As $t\rightarrow 0$, 
 $1-\lambda_t\sim \log(1/|t|)\langle t,\Sigma_{d+1} t\rangle$,
 where $\Sigma_{d+1}$ is given by~\eqref{Sigmad+1}.
\end{sublemma}

\begin{pfof}{Sublemma~\ref{sub:asl}}
To study the expansion of $t\mapsto \lambda_t$, in both Sublemmas~\ref{sub:asl} and~\ref{sunl:lamb}, 
we consider the normalized eigenvectors $v_t,\widetilde v_t$  of respectively $P_t,\widetilde P_t$ associated with $\lambda_t,\widetilde\lambda_t$ given by
$v_t= \frac{\Pi_t(1_\Delta)}{\mathbb E_{\mu_\Delta}[\Pi_t(1_\Delta)]}$ and $\widetilde v_t= \frac{\widetilde\Pi_t(1_\Delta)}{\mathbb E_{\mu_\Delta}[\widetilde\Pi_t(1_\Delta)]}$, and
we will use the following expressions
\[
\lambda_t=\mathbb E_{\mu_{\Delta}}[P_t(v_t)]=\mathbb E_{\mu_{\Delta}}[e^{i\langle t,\overline\Psi\rangle}v_t]\quad\mbox{and}\quad
\widetilde \lambda_t=\mathbb E_{\mu_{\Delta}}[\widetilde P_t(\widetilde v_t)]=
\mathbb E_{\mu_{\Delta}}[e^{i\langle t,\overline\Upsilon\rangle}\widetilde v_t]
\, ,
\]
Therefore
\begin{align}
\label{eq:lam}
\lambda_t-\widetilde\lambda_t &= I_1(t)+I_2(t)\, ,
\end{align}
with
\begin{align*}I_1(t):=\int_{\overline\Delta}e^{i\langle t,\overline\Psi\rangle}( v_t-\widetilde v_t)\, d\mu_{\overline\Delta}=\int_{\overline\Delta}(1-e^{i\langle t,\overline\Psi\rangle})( \widetilde v_t- v_t)\, d\mu_{\overline\Delta}\, ,
\end{align*}
and
\begin{align*}
\quad I_2(t):=\int_{\overline\Delta} (e^{i\langle t,\overline\Psi\rangle}-e^{i\langle t,\Upsilon\rangle})\widetilde v_t\, d\mu_{\overline\Delta}\, .
\end{align*}
As argued below, $I_1(t)$ and $I_2(t)$ are $\mathcal O(|t|^2)$.
Regarding $I_2$, we first note that $(e^{i\langle t,\overline\Psi\rangle}-e^{i\langle t,\Upsilon\rangle})\cdot\in\mathcal L(\cB)$ is dominated by the Lipschitz norm of $(e^{i\langle t,\overline\Psi\rangle}-e^{i\langle t,\Upsilon\rangle})=e^{i\langle t,\Upsilon\rangle}(e^{i\langle t,\Theta\rangle}-1)$.
It follows from the definitions of $v_t,\widetilde v_t$ and\eqref{Pit-tildePit} that
\begin{align*}
 \|v_t-\widetilde v_t\|_{\mathcal L(\cB)}=\mathcal O(t),\quad \mbox{as }t\rightarrow 0\, .
\end{align*}
Recall that $\cB\subset L^{b_q}$ with $b_q>2$ fixed satisfying~\eqref{defbq}.
Further, note that by~\eqref{eq:tail},$\Upsilon$ and thus $\overline\Psi$
are in $L^q$ for any $q<2$. By the choice of $b_q$, $b_q>q/(q-1)$.
Hence,
\[
 |I_1(t)|\ll |t|\,\int_{\overline\Delta}|\overline\Psi|\,|\widetilde v_t-v_t|\,d\mu_{\overline\Delta} 
 \ll|t|\,\|\overline\Psi\|_{L^q}\|\widetilde v_t-v_t\|_{L^{\frac{q}{q-1}}}\ll |t|
 \,|\widetilde v_t-v_t\|_{\cB}\ll t^2.
\]
Next, recalling the definition of $\Theta$ in~\eqref{eq:rewr},
\begin{align*}
 I_2(t)=\int_{\overline\Delta} e^{i\langle t,\Upsilon\rangle}(e^{i\langle t,\Theta\rangle}-1)\widetilde v_t\, d\mu_{\overline\Delta}=&I_2^1(t)+I_2^2(t)
\end{align*}
with
\[
I_2^1(t):= it\int_{\overline\Delta} e^{i\langle t,\Upsilon\rangle}\,\Theta\,\widetilde v_t\, d\mu_{\overline\Delta},\quad\mbox{and}\quad
I_2^2(t):=\int_{\overline\Delta} e^{i\langle t,\Upsilon\rangle}(e^{i\langle t,\Theta\rangle}-1-it\Theta)\widetilde v_t\, d\mu_{\overline\Delta}\, .
\]
Now, since $\Theta$ is bounded, 
\[
 |I_2^2(t)|\ll |t|^2\int_{\overline\Delta} |\Theta|^2\,|\widetilde v_t|\, d\mu_{\overline\Delta}\ll  (|t|\, \Vert\Theta\Vert_\infty)^2\int_{\overline\Delta}|\widetilde v_t|\, d\mu_{\overline\Delta}
 \ll |t|^2.
\]
For $I_2^1$, write
\begin{align*}
 I_2^1(t)&=it\int_{\overline\Delta} e^{i\langle t,\Upsilon \rangle}\,\Theta\,\widetilde v_0\, d\mu_{\overline\Delta} +it\int_{\overline\Delta} e^{i\langle t,\Upsilon \rangle}\,\Theta\,(\widetilde v_t-\widetilde v_0)\, d\mu_{\overline\Delta}\\
 &=it\int_{\overline\Delta}\Theta\,d\mu_{\overline\Delta}+it\int_{\overline\Delta} (e^{i\langle t,\Upsilon \rangle}-1)\,\Theta\,\widetilde v_0\, d\mu_{\overline\Delta}+\mathcal O(|t|^2)\\
 &=0+\mathcal O(|t|^2),\quad\mbox{as }t\rightarrow 0\, ,
 \end{align*}
where we have used that $\int_{\overline\Delta}\Theta\,d\mu_{\overline\Delta}=0$.
Putting the above together, $|I_2(t)|\ll |t|^2$.

The bounds for $I_1$ and $I_2$ together with~\eqref{eq:lam} give that
\begin{equation}\label{eq:l2}
\lambda_t= \widetilde\lambda_t +\mathcal O(|t|^2),\quad \mbox{as }t\rightarrow 0\, .
\end{equation}
It remains to show that $1-\widetilde\lambda_t$ has the desired asymptotic, using
the form of $\Upsilon $ in~\eqref{eq:rewr}.
To this end we observe that
\begin{align}
\label{eq:tlam}
1-\widetilde\lambda_t= \mathbb E_{\mu_{\overline\Delta}}[1-\widetilde P_t(\widetilde v_t)]=\mathbb E_{\mu_{\overline\Delta}}[(1-e^{i\langle t,\Upsilon\rangle})\widetilde v_t]=I'_1(t)+I'_2(t)\, ,
\end{align}
with
\begin{align*}I'_1(t):=\int_{\overline\Delta} (1-e^{i \langle t,\Upsilon\rangle})\widetilde v_0\, d\mu_{\overline\Delta}\quad\mbox{ and }\quad I'_2(t):=\int_{\overline\Delta} (1-e^{i\langle t,\Upsilon\rangle})(\widetilde v_t-\widetilde v_0)\, d\mu_{\overline\Delta}\, .
\end{align*}
We estimate $I'_1(t)$.
Recall equations~\eqref{eq:t1},~\eqref{eq:tail} and that $\int_{\overline\Delta}\Upsilon\, d\mu_{\overline\Delta}=0$.
For any $L,w,N$,  write $W_N(L,w):=(\pi_d(L+Nw),|L+Nw|-\mathbb E_{\mu}[|\widetilde\kappa|])$
and compute that
\begin{align*}
I'_1(t)&=\int_{\overline\Delta} (1-e^{i \langle t,\Upsilon\rangle}-i \langle t,\Upsilon\rangle)\, d\mu_{\overline\Delta}\\
&=\sum_{(L,w)\in\mathcal E}\sum_{N\ge 1}
\left(e^{i\langle t, W_N(L,w)\rangle}-1-i\langle t,W_N(L,w)\rangle\right) \mu
(\{ \kappa=L+Nw
\})+\mathcal O(|t|^2)\\
&=\sum_{(L,w)\in\mathcal E}\sum_{N=1}^{1/|t|}  \left(e^{i\langle t,W_N(L,w)\rangle}-1-i\langle t,W_N(L,w)\rangle\right) 
\left(c_{L,w}N^{-3}+\mathcal O(N^{-4})\right)
\\
&\quad +O\left(|t|^2+|t|\sum_{(L,w)\in\mathcal E}|w|\sum_{N>1/|t|} N. N^{-3}
\right )\\
&=\sum_{(L,w)\in\mathcal E} \frac{c_{L,w}}{2}\sum_{N=1}^{1/|t|}\, \langle t,W_N(L,w)\rangle^2 N^{-3} 
+\mathcal O(|t|^2)\, .
\end{align*}
Since
\(
\langle t, W_N(L,w)\rangle^2 N^{-3} = N^{-1} \langle t,(\pi_d(w),|w|)\rangle^2 +\mathcal O(|t|^2N^{-2})\, ,
\)
\begin{align}
\nonumber I'_1(t)&=
\sum_{(L,w)\in\mathcal E}
\frac{c_{L,w}}{2}\, \log (1/|t|)\langle t, (\pi_d(w),|w|)  \rangle^2+\mathcal O(|t|^2)\\
&=\log(1/|t|) \langle \Sigma_{d+1} t,t\rangle+\mathcal O(|t|^2)\, .\label{eq:i1}
\end{align}
For $I'_2(t)$, we just need to explain that the argument of~\cite{SV07} via \cite{BalintGouezel06},
which provides the asymptotic of the eigenvalue associated to perturbation by $\kappa$
instead of $\Upsilon $ as here, goes through.
The required ingredients in the argument of~\cite{SV07, BalintGouezel06} are: 1) tail of $\kappa$  and 2) the 'double probability' estimate~\eqref{controltimelog}. Regarding 1), we already know the tail of $\Upsilon $: see equations~\eqref{eq:t1} and~\eqref{eq:tail}.
Regarding 2), an analogue of~\eqref{controltimelog}, we recall that we already know that this
holds for $\widetilde\kappa$. 
Because of the expression of $\Upsilon$ (via $\widetilde\kappa$),
the arguments used in~\cite[Proofs of Propositions 11–12]{SV07} ensure that
\begin{align}
\label{eq:svd}
\mu_{\overline\Delta}\left(A_{N,V}\right)=\mathcal O(N^{-3-\varepsilon}),\quad\mbox{as }N\rightarrow +\infty\, ,
\end{align}
where
\[
A_{N,V}:=\left\{\Upsilon=W_N(L,w),\ \exists |j|\le V\log (n+2),\ |\Upsilon\circ f^j|>|W_N(L,w)|^{4/5}\right\}.
\]
Equations~\eqref{eq:tail} and~\eqref{eq:svd} together with~\cite[Proof of Theorem 3.4]{BalintGouezel06} ensure that\footnote{More precisely, see \emph{Estimate of $\lambda_t$ (there) in the proof of}~\cite[Proof of Theorem 3.4]{BalintGouezel06}. In particular, see~\cite[Lemma 3.16 and 3.19]{BalintGouezel06}}

\begin{align}
\label{eq:lntla}
I'_2(t)=  o\left(|t|^2\log(1/|t|)\right),\quad\mbox{as }t\rightarrow 0\, ,
\end{align} 
The conclusion from
this together with~\eqref{eq:l2},~\eqref{eq:tlam} and~\eqref{eq:i1}.~\end{pfof}

In the remainder of this section, we prove a stronger version
of Sublemma~\ref{sub:asl}, along with a strong continuity estimate on $\Pi_t$,  that will be
essential in the proofs of Lemmas~\ref{lem:jointllt0}
and~\ref{lem:jointllt} carried out in Section~\ref{sec:proofjMLLT}.

Recall that $q\in[1,2)$ has been fixed at the beginning of the present section.
\begin{sublemma}\label{sunl:pi} 
\begin{align*}
	\left\Vert\widetilde\Pi_t-\widetilde\Pi_0\right\Vert_{\mathcal B
		\rightarrow L^{
			q}(\mu_{\overline\Delta})}+\left\Vert\Pi_t-\Pi_0\right\Vert_{\mathcal B
		\rightarrow L^{
			q}(\mu_{\overline\Delta})}=\mathcal O(t)\, \quad\mbox{as }t\rightarrow 0\, .
	\end{align*}
\end{sublemma}
\begin{proof}
Due to~\eqref{Pit-tildePit}, it is enough to control
$\left\Vert\widetilde\Pi_t-\widetilde\Pi_0\right\Vert_{\mathcal B
	\rightarrow L^{
		q}(\mu_{\overline\Delta})}$. 
 We claim that with the choice of $b_q$ (see~\eqref{defbq}) and $\varepsilon$ in the text before~\eqref{fixb0}, \cite[Proposition 5.4]{PeneTerhesiu21} applies to ensure that
\begin{equation}\label{prop5-4}
\Vert 1_Y(\widetilde\Pi_t-\widetilde\Pi_0)\Vert_{\mathcal B\rightarrow \mathcal B_0}=\mathcal O(|t|),\quad\mbox{as }t\rightarrow 0\, .
\end{equation}
Using~\eqref{prop5-4}, we modify the proof of  \cite[Proposition 5.3]{PeneTerhesiu21}
to conclude the proof of the sublemma. Set $\pi_0(x,\ell):=x$, recall that $\omega(x,\ell)=\ell$ and, using~\cite[Formula (44)]{PeneTerhesiu21}, we write $(\widetilde\Pi_t-\widetilde\Pi_0)(w)(x)=I_{1,t}(x)+I_{2,t}(x)+I_{3,t}(x)$
with
\begin{align*}
I_{1,t}(x)&:=\left[(e^{i\langle t,\Upsilon_{\omega(x)}(\pi_0(x))\rangle}-1)\widetilde\Pi_0(w)\right]\\
I_{2,t}(x)&:=\left[e^{i\langle t,\Upsilon_{\omega(x)}(\pi_0(x))\rangle}
(\widetilde\Pi_t-\widetilde\Pi_0) (w)(\pi_0(x))\right]
\\
I_{3,t}(x)&:=
(\widetilde \lambda_t^{-\omega(x)}-1)
[e^{i\langle t,\Upsilon_{\omega(x)}(\pi_0(x))\rangle}\widetilde\Pi_t (w)(\pi_0(x))]\, .
\end{align*}
First,
\begin{align}
\nonumber\Vert I_{1,t}\Vert_{L^q(\mu_{\overline\Delta})}^q&=\mathbb E_{\mu_{\overline\Delta}} \left[\left((e^{i\langle t,\Upsilon_{\omega(\cdot)}(\pi_0(\cdot))\rangle}-1)\right)^q\right]\Vert w\Vert_{L^1(\mu_{\overline\Delta})}^q\\
\nonumber&\le |t|^q\mathbb E_{\mu_{\overline\Delta}}\left[ \left|\Upsilon_{\omega(\cdot)}(\pi_0(\cdot))\right|^q\right]\Vert w\Vert_{L^1(\mu_{\overline\Delta})}^q\\
\nonumber&\ll |t|^q\sum_{n\ge 0}\mathbb E_{\mu}\left[1_{Y\cap \{R> n\}}   |\Upsilon_n|^q \right]\Vert w\Vert_{L^1(\mu_{\overline\Delta})}^{q}\\
&\ll |t|^q\sum_{n\ge 0}\mu_Y(R> n)^{1/q'} \Vert |\Upsilon_n|^{q}\Vert_{L^{p'}(\mu_{\overline Y})}\Vert w\Vert_{L^1(\mu_{\overline\Delta})} ^{q}\ll |t|^q\Vert w\Vert_{\mathcal B} ^{q}\, ,\label{ControlI1}
\end{align}
taking $p'>1$ and $q'>1$ such that $\frac 1{p'}+\frac 1{q'}=1$ and $qp'<2$, and
using the fact that $\mu_{\overline Y}(R\ge n)$ decreases exponentially fast.
Second, it follows from~\eqref{prop5-4} that
\begin{align}
\nonumber
\left\Vert I_{2,t}\right\Vert_{L^q(\mu_{\overline\Delta})}^{q}&\le\left\Vert
(\widetilde\Pi_t-\widetilde\Pi_0) (w)(\pi_0(\cdot))\right\Vert_{L^q(\mu_{\overline\Delta})}^q\\
&\ll \Vert 1_Y(\widetilde\Pi_t-\widetilde\Pi_0)(w)\Vert^q_\infty \ll |t|^q\Vert w\Vert_{	\mathcal B}^q\, .\label{ControlI2}
\end{align}
Third, by Sublemma~\ref{sub:asl}, there exists some $a'>0$ such that, for $t$ small enough,  
$1>|\widetilde \lambda_t|>e^{-a'\frac {|t|^2\log(1/|t|)} 2}$, and so
\begin{align}
\nonumber\left\Vert I_{3,t}\right\Vert_{L^q(\mu_{\overline\Delta})}^q &\le 
\left\Vert 1_Y\widetilde\Pi_t (w)\right\Vert_\infty\, 
\left|\lambda_t-1\right|^q\mathbb E_{\mu_{\overline\Delta}}\left[ \left|\omega(\cdot)
e^{a'\frac {|t|^2\log(1/|t|)} 2(\omega(\cdot)-1)}\right|^{q}\right]\\
\nonumber&\ll \left(|t|^2\log(1/|t|)\left\Vert \widetilde\Pi_t (w)\right\Vert_{\mathcal B}\right)^{q} \sum_{n\ge 1}\mu_Y(R>n)n^{q} e^{a'q\frac {n|t|^2\log(1/|t|)} 2}\\
&\ll  \left(|t|^2\log(1/|t|) \Vert w\Vert_{\mathcal B}\right)^{q}\, ,\label{ControlI3}
\end{align}
provided $|t|$ is small enough, using again that
 $\mu_Y(R>n)$ decays exponentially fast in $n$.
The conclusion follows from~\eqref{ControlI1},~\eqref{ControlI2} and~\eqref{ControlI3}.

It remains to complete

{\bf{Proof of the claim~\eqref{prop5-4}}.}
Recall $b_q$ satisfies~\eqref{defbq} and that $\varepsilon$ has been fixed in the text before~\eqref{fixb0}; in particular, $\frac{1}{b_q}+\frac 1q<1$. 
There exists $p\in (2,b_q)$
so that $\frac{1}{p}+\frac 1q<1$. In particular, $1<q\frac{p-1}{p}$
and so $2\frac{p-1}{p}>\frac 2q>1$. Let $\gamma\in\left(1,\frac{p-1}p\right)$.
Let $h\in \mathcal B$. 
Note that $h=vw$, with  $w:=e^{-\varepsilon\omega}h\in\mathcal B_0$ and with $v:=e^{\varepsilon\omega}\in L^{b_q}$ constant on partition elements.

With these choices, $h,v,p,b=b_q,\gamma$
satisfy the assumptions of \cite[Proposition 5.4, Lemma C.2]{PeneTerhesiu21} which still holds true with the same proof in when replacing $\Pi_t$ therein by the present $\widetilde \Pi_t$ (this is equivalent to replacing $\overline\kappa$ in~\cite{PeneTerhesiu21} by the present $\overline\Psi$).
Again, this adaptation is possible because $\overline\Psi$ is constant on partition elements and $\overline\Psi$ has a similar tail to that of
$\overline\kappa$ (again, see equations~\eqref{eq:t1} and~\eqref{eq:tail}).
(The similar tail ensures, in particular, that $\overline\Psi$.
is as integrable as $\overline\kappa$. ) As a consequence, \cite[Proof of Lemma C.2]{PeneTerhesiu21} goes through with $\overline\kappa$ replaced by $\overline\Psi$. Moreover, because of the same properties of $\overline\Psi$, the proof of \cite[Proposition 5.4]{PeneTerhesiu21}
can be easily modified to prove the claim~\eqref{prop5-4}.

The adaptation of \cite[Proposition 5.4]{PeneTerhesiu21} implies  that
\[
\Vert 1_Y(\widetilde\Pi_t-\widetilde\Pi_0)(h)\Vert_{\mathcal B_0}\ll |t| \Vert 1_Y\widetilde \Pi'_0(h)\Vert_{\mathcal B_0}+ |t|^\gamma\left(\Vert h \Vert_{\mathcal B}+\Vert e^{-\varepsilon\omega}h\Vert_{\mathcal B_0} \Vert e^{\varepsilon\omega}\Vert_{L^{b_q}}\right)\, .
\]
By~\cite[Lemma C.2]{PeneTerhesiu21} with $\overline\kappa$ replaced by $\overline\Psi$,

\[
\left( \Vert 1_Y\widetilde\Pi'_0(h)\Vert_{\mathcal B_0}\le \Vert h \Vert_{\mathcal B}+\Vert e^{-\varepsilon\omega}h\Vert_{\mathcal B_0} \Vert e^{\varepsilon\omega}\Vert_{L^{b_q}}\right)\, ,
\]
We conclude that
\[
\Vert 1_Y(\widetilde\Pi_t-\widetilde \Pi_0)(h)\Vert_{\mathcal B_0}\ll |t|\left( \Vert h \Vert_{\mathcal B}+\Vert e^{-\varepsilon\omega}h\Vert_{\mathcal B_0} \Vert e^{\varepsilon\omega}\Vert_{L^{b_q}}\right)
\ll |t|\Vert h\Vert_{\mathcal B}\, ,
\]
since $\Vert e^{-\varepsilon\omega}h\Vert_{B_0}=\Vert h\Vert_{\mathcal B}$.~\end{proof}

\begin{sublemma}\label{sunl:lamb} 
As $t\to 0$,
	$1-\lambda_t= \log(1/|t|)\langle t,\Sigma_{d+1} t\rangle  +O(|t|^{2})$.\\
\end{sublemma}

\begin{proof}
We keep the notations of the proof of Sublemma~\ref{sub:asl}. It follows from
~\eqref{eq:l2},~\eqref{eq:tlam} and~\eqref{eq:i1} that
$1-\lambda_t=\log(1/|t|)\langle t,\Sigma_{d+1} t\rangle +I'_2(t) +O(|t|^{2})$. 
The proof that $|I'_2(t)|=O(|t|^2)$ 
follows exactly as in
the proof of~\cite[Lemma 6.1]{PeneTerhesiu21}
(replacing everywhere $\overline\kappa,P_t,\Pi_t,v_t,\lambda_t$ therein  by $\Upsilon,\widetilde P_t,\widetilde\Pi_t,\widetilde v_t,\widetilde \lambda_t$ and following the proof line by line). This is due to the fact that $\Upsilon$ satisfies the following properties (that are also satisfied by $\overline\kappa$):  $\Upsilon$ is constant on partition elements,  equations~\eqref{eq:tail}, ~\eqref{eq:lntla} and~\eqref{controltimelog} hold, and the estimate on $\widetilde\Pi_t$ stated in  Sublemma~\ref{sunl:pi} holds.
\end{proof}

\section{Proofs of joint MLLT (Lemmas~\ref{lem:jointllt0} and~\ref{lem:jointllt})}\label{sec:proofjMLLT}
Let $d\in\{0,1,2\}$.
 Compared to CLT, a specific property required to prove the MLLT is the non-arithmeticity (or minimality), which is treated in the next lemma.
\begin{lemma}\label{nonarithmeticity}
	For every proper closed subgroup  $\Gamma$ of 	$\mathbb Z^d\times\mathbb R$ and  for every $a\in \mathbb Z^d\times\mathbb R$,
	\[
	\mu\left(\widehat\Psi+g-g\circ T\not\in a+\Gamma\right)>0\, .
	\]
\end{lemma}
\begin{proof}
We adapt the proof of \cite[Lemma A.3]{DN16} to $d+1$-dimensional observable $\widehat\Psi$, with a slightly different presentation. 
Assume there exists a proper subgroup $\Gamma$ of 
$\mathbb Z^d\times\mathbb R$, a measurable function
$g=(g_1,...,g_{d+1}):M\rightarrow \mathbb Z^d\times\mathbb R$ and $a=(a_1,...,a_{d+1})\in \mathbb Z^d\times\mathbb R$ such that $\widehat\Psi+g-g\circ T\in a+\Gamma$
$\mu$-almost surely. 
\begin{itemize}
	\item Let us prove that we can find a family $(v_1,...,v_{d+1})$ of generators of $\Gamma$ of the form
	\[
	v_i:=(\mathbf{e_i},\alpha_i)\ \mbox{for }i=1,...,d,\quad 
	v_{d+1}:=(\mathbf 0,\alpha_{d+1})\, ,
	\]
	where $\mathbf{e_i}$ is the $i$-th vector of the canonical basis of $\mathbb R^d$ and where $\mathbf 0$ is the null element of $\mathbb Z^d$.
	Let $\alpha_1,...,\alpha_d\in\mathbb R$ be such that
	$v_i:=(\mathbf{e_i},\alpha_i)\in \Gamma$ for $i=1,...,d$. Such numbers exist since the projection of $\Gamma$ on the first $d$ coordinates generates $\mathbb Z^d$ (since it has been proved in \cite{SV07} that $\kappa:M\rightarrow\mathbb Z^d$ is non-arithmetic).
	
	Observe that $\Gamma\cap(\{\mathbf 0\}\times\mathbb R)
	$
	is a discrete subgroup of $\{
	\mathbf 0\}\times\mathbb R$. Indeed it is a closed subgroup, and it cannot be $
	\{\mathbf 0\}\times\mathbb R$ otherwise $\Gamma$ would be $\mathbb Z^d\times\mathbb R$,  since then any element $(a'_1,...,a'_{d+1})$ of $\mathbb Z^d\times\mathbb R$ could be rewritten $\left(\sum_{i=1}^da'_i.v_i\right)+\left(a'_{d+1}-\sum_{i=1}^da'_i\alpha_i\right)(\mathbf 0,1)$.
	Hence, 
	 $\Gamma\cap(\{\mathbf 0\}\times\mathbb R)$ is discrete and has the form
	$\{
	\mathbf 0\}\times (\alpha_{d+1}\mathbb Z)$ for a non-negative real number $\alpha_{d+1}$. Set $v_{d+1}:=(\mathbf 0
	,\alpha_{d+1})\in \Gamma$.

	Let us prove that $(v_1,...,v_{d+1})$ generates the group $\Gamma$.
	Let $a'=(a'_1,...,a'_{d+1})\in \Gamma\subset \mathbb Z^d\times\mathbb R$. Set $w=a'-\sum_{i=1}^da'_iv_i=(\mathbf 0,\beta)$.
	By definition of $\alpha_{d+1}$, there exists $m\in\mathbb Z$ such that $\beta =m\alpha_{d+1}$. Thus $a'\in \sum_{i=1}^{d+1}\mathbb Zv_i$.
	\item Let us prove that there exist $r,\alpha\in\mathbb R$ and two measurable functions $c':M\rightarrow \mathbb Z$ and  $G:M\rightarrow\mathbb R$ such that $\mu(0<G<\min\tau)>0$ and
	\[
	\tau-\mu(\tau)=r+  \alpha.c'+G-G\circ T\, .
	\]
	It follows from the previous item that there exists a measurable function $c=(c_1,...,c_{d+1}):M\rightarrow \mathbb Z^{d+1}$ such that
	\[ 
	\widehat\Psi +g-g\circ T= a+ \sum_{i=1}^{d+1}c_i. v_i\, \quad \mu-a.s.\, ,
	\]
	by taking $c_i=\kappa_i+g_i-g_i\circ T-a_i $ for $i\in\{1,...,d\}$ 
	and 
	\begin{equation}
	\label{formulec3}
	c_{d+1}=\left(\tau-\mu(\tau)+g_{d+1}-g_{d+1}\circ T-a_{d+1}-\sum_{i=1}^d\alpha_i.c_i\right)/\alpha_{d+1}\mathbf 1_{\alpha_{d+1}\ne 0}\, .
	\end{equation}
Thus 
	\[ 
	\tau-{\mu(\tau)}+g_{d+1}-g_{d+1}\circ T=a_{d+1}+\sum_{i=1}^{d+1} \alpha_i.c_i\, .
	\]
	Now we observe that $\kappa=-\kappa\circ\xi\circ T$ and $\tau=\tau\circ\xi\circ T$ with $\xi$ the involution mapping $(q,\varphi)\in \partial\Omega\times[-\frac\pi 2;\frac\pi 2]$ to $(q,-\varphi)$, so the previous identity composed with $\xi\circ T$ becomes
	\[ 
	\tau -{\mu(\tau)} +g_{d+1}\circ\xi \circ T-g_{d+1}\circ T\circ\xi \circ T=a_{d+1}+\sum_{i=1}^{d+1} \alpha_i.c_i\circ \xi \circ T\, .
	\]
	Therefore, by taking the average of the two previous identities, we obtain
	\begin{align*}
	 \tau &-{\mu(\tau)} +\frac{g_{d+1}\circ\xi \circ T-g_{d+1}\circ T\circ\xi \circ T}{2}\\
	 &+\frac{g_{d+1}-g_{d+1} \circ T}{2}=a_{d+1}+ \sum_{i=1}^{d+1} \alpha_i\frac{c_i+c_i\circ \xi \circ T}2\, .
	\end{align*}

	But for $i=1,...,d$,
	\[c_i+c_i\circ\xi \circ T=g_i-g_i\circ T-g_i\circ\xi \circ T+g_i\circ T\circ\xi \circ T\, 
	\]
	is a coboundary, and so, due to~\eqref{formulec3}, $\alpha_{d+1}\frac{c_{d+1}+c_{d+1}\circ\xi \circ T}2-(\tau -{\mu(\tau)} -a_{d+1})$ is also a coboundary.
	Thus we have proved the existence of two measurable functions
	$c':M\rightarrow \mathbb Z$ and $G:M\rightarrow\mathbb R$ such that
	\[
	\tau -{\mu(\tau)} =a_{d+1}+ \frac{\alpha_{d+1}}2.c'+G-G\circ T\, , i.e.\ \tau=r+ \frac{\alpha_{d+1}}2.c'+G-G\circ T\, ,
	\]
	with $r:=\mu(\tau)+a_{d+1}$. 
	The condition $\mu(0<G<\min\tau)>0$ is obtained up to adding a constant to $G$. The above identity would contradict the non-aritmeticity of $\tau$.
	\item We follow exactly the second part of the proof of \cite[Lemma A.3]{DN16}.\\
	Set $\alpha:=\frac{\alpha_{d+1}}2$.
	For $\delta>0$, we consider the set $\mathcal C''_\delta$ of points $y=\phi_t(x)$ with $x\in M$, $0<t<\tau(x)$ and $|t-G(x)|<\delta$.\\
	It follows from the previous item that 
	the first return time $\zeta$ to $\mathcal C''_0$ takes its values in $r\mathbb N + \alpha \mathbb Z$. Indeed if $y=\phi_t(x)\in\mathcal C''_0$ with $t=G(x)\in(0;\tau(x))$ and $\phi_s(y)=\phi_u(T^n(x))\in \mathcal C''_0$, with $u=t+s-\tau_n(x)=G(T^n(x))\in(0,\tau(T^n(x)))$, then $s=G(T^n(x))-t+\tau_n(x)=G(T^n(x))-G(x)+\tau_n(x)=nr+\alpha c'_n$.\\
	We choose $\varepsilon>0$ so that $r+\varepsilon\in \alpha\mathbb Q$. Let $b$ be the smallest positive element of
	$\alpha\mathbb Z+(r+\varepsilon)\mathbb Z$, so that $\alpha\mathbb Z+(r+\varepsilon)\mathbb Z=b\mathbb Z$. Indeed, if $r+\varepsilon=\alpha\frac p q$ with $p,q\in\mathbb Z$, $q\ne 0$, then 
	$\alpha\mathbb Z+(r+\varepsilon)\mathbb Z=\frac{\alpha}q\left(q\mathbb Z+ p\mathbb Z \right)=b\mathbb Z$, with $b:=\frac{|\alpha|\, gcd(p,q)}{|q|}$.
	\\ The first return time to $\mathcal C'_0:=\{(x,G(x)),x\in M, G(x)\le \tau(x)+\varepsilon\}$ for the suspension flow $\phi^{(\varepsilon)}$ over $(M,T)$ with roof function $\tau +\varepsilon$ is in $\alpha\mathbb Z+(r+\varepsilon)\mathbb N\subset b\mathbb Z$. Indeed, 
	if $y=\phi_t(x)\in\mathcal C'_0$ with $t=G(x)\in(0;\tau(x)+\varepsilon)$ and $\phi^{(\varepsilon)}_s(y)=\phi^{(\varepsilon)}_u(T^n(x))\in \mathcal C'_0$, with $u=t+s-\tau_n(x)-n\varepsilon=G(T^n(x))\in(0,\tau(T^n(x))+\varepsilon)$, then $s=G(T^n(x))-t+\tau_n(x)+n\varepsilon=G(T^n(x))-G(x)+\tau_n(x)+n\varepsilon=nr+\alpha c'_n+n\varepsilon$.\\
	Thus, if $\delta\in(0,b/2)$, the return time of $\phi^{(\varepsilon)}$ to $\mathcal C'_\delta$ (defined as $\mathcal C''_\delta$ but for $\phi^{(\varepsilon)}$) occurs only at time $t$ at distance at most $2\delta$ of
	$b\mathbb Z$, which contradicts the mixing of $\phi^{(\varepsilon)}$. 
		As explained in \cite[Lemma A.2]{DN16}, the proof of the mixing of the suspension flow $\phi^{(\varepsilon)}$ follows the same line as the mixing of the billiard flow established in~\cite[Sections 6.10-6.11]{ChernovMarkarian} thanks to the temporal distance (which remains unchanged if we replace $\tau$ by $\tau+\varepsilon$).
	
\end{itemize}

\end{proof}

\begin{pfof}{Lemma~\ref{lem:jointllt0}}
	Let $p>2$, $G,H,h$ as in the assumptions of Lemma~\ref{lem:jointllt0}. 
	To prove the joint MLLT, we will use estimates established in
	Section~\ref{sec:proofjCLT}.
	We keep the notations of this section with the couple $(q,p_q)$ being chosen so that $q\in [1,2)$ such that $\frac 1p+\frac 1q=1$ and $b_q>p$ (this will imply that $\frac{1}{b_q}+\frac 1q<1$).
	
On $\Delta$ and $\overline\Delta$, we keep the convention $u_n:=\sum_{k=0}^{n-1}u\circ f_\Delta^k$ and
$\overline u_n:=\sum_{k=0}^{n-1}\overline u\circ f_{\overline\Delta}^k$
for any $u:\Delta\rightarrow\mathbb R^{d'}$ and 
$\overline u:\overline\Delta\rightarrow\mathbb R^{d'}$.\\

For simplicity, we keep the notation $G,H,\widehat\Psi$
for the functions defined on $\Delta$ (instead of $M$) corresponding to $G\circ\pi,H\circ\pi,\widehat\Psi\circ\pi$ respectively. 
Since the functions $G$ and $H$ are bounded and dynamically H\"older continuous on $M$, the functions $G$ and $H$ are also bounded and dynamically H\"older on $\Delta$ in the following sense: up to increasing the value of 
$\beta\in(0,1)$ in the Young Banach space $\mathcal B$ introduced in Section~\ref{sec:proofjCLT}, $G$ and $H$ satisfy the following property
\begin{equation}\label{regGH}
s(x,y)>2k\quad\Rightarrow\quad |G(f_\Delta^k(x))-G(f_\Delta^k(y))|<L_G\beta^k,\ 
|H(f_\Delta^k(x))-H(f_\Delta^k(y))|<L_H\beta^k\, .
\end{equation}
Recall that $\varepsilon$ in the definition of Young's Banach space $\mathcal B$ (see text before~\eqref{fixb0}) is so that $\mathcal B$ is continuously embedded in $\mathbb L^{b_q}(\mu_{\overline \Delta}) $. Also, by Sublemma~\ref{sunl:pi}, 
\[
\Vert \Pi_t-\Pi_0\Vert_{\mathcal B\rightarrow L^q(\mu_{\overline\Delta})}=\mathcal O(t)\, .
\]
With the above notations, we are led to the study of the following quantity:
\[
\mathbb E_{\mu_\Delta}[G .h(\widehat\Psi_n-L)H\circ f_\Delta^n]\, .
\]
It follows from the Fourier inversion theorem that
\begin{align}\label{FourierTransform}
\mathbb E_{\mu_\Delta}[G .h(\widehat\Psi_n-L)H\circ f_\Delta^n]&=
\frac 1{(2\pi)^{d+1}}\int_{\mathbb T^d\times\mathbb R}e^{-i\langle t,L\rangle}\widehat h(t)\mathbb E_{\mu_\Delta}\left[Ge^{i\langle t,\widehat\Psi_n\rangle}H\circ f_\Delta^n\right]\, dt\, ,
\end{align}
with $\widehat h(t):=\sum_{\ell\in\mathbb Z^d}\int_{\mathbb R}h(\ell,x)e^{i\langle t,(\ell,x)\rangle}\, dx$.

\begin{itemize}
	\item \underline{Step 1}: Transition to the quotient\\
\begin{itemize}
\item Approximation.\\ 
	Recall, from~\eqref{coboundPsibis} that  
	$\hPsi\circ\pi=\overline\Psi\circ\overline\pi+\chi\circ f_\Delta-\chi$, with
	$\overline \Psi=\left(\overline\kappa,\overline \tau \right)$ with values in $\mathbb Z^d\times\mathbb R$ uniformly locally H\"older on each partition element,  with $\chi$ bounded and dynamically H\"older in the sense of~\eqref{regbartau}.
	Thus
	\begin{align*}
	\mathbb E_{\mu_\Delta}&[Ge^{i\langle t,\widehat\Psi_n\rangle}H\circ f_\Delta^n]=
	\mathbb E_{\mu_\Delta}[G\circ f_\Delta^{k_n}e^{i\langle t,\overline\Psi_n\circ \overline f_{\overline\Delta}^{k_n}\circ\overline\pi+\chi\circ f_\Delta^{k_n+n}-\chi\circ f_\Delta^{k_n}\rangle}H\circ f_\Delta^{k_n}\circ f_\Delta^n]\\
	&=
	\mathbb E_{\mu_\Delta}\left(e^{-i\langle t,\overline\Psi_{k_n}\circ \overline\pi\rangle}(Ge^{-i\langle t,\chi\rangle})\circ f_\Delta^{k_n}.e^{i\langle t,\overline\Psi_n\circ\overline\pi\rangle}.\left((He^{i\langle t,\chi\rangle})\circ f_\Delta^{k_n}e^{i\langle t,\overline\Psi_{k_n}\circ \overline\pi\rangle}\right)\circ f_\Delta^n\right)\, .
	\end{align*}
	We approximate $\widehat G_{({k_n})}(t):=e^{-i\langle t,\overline\Psi_{k_n}\circ \overline\pi\rangle}(Ge^{-i\langle t,\chi\rangle})\circ f_\Delta^{k_n}$ by $G_{({k_n})}(t)\circ\overline\pi$ with
	\[
	G_{({k_n})}(t)\circ\overline\pi:=\mathbb E_{\mu_\Delta}[ e^{-i\langle t,\overline\Psi_{k_n}\circ\overline\pi\rangle}
	(Ge^{-i\langle t,\chi\rangle})\circ f_\Delta^{k_n} |s(.,.)>2{k_n}]\]
	and analogously $e^{i\langle t,\overline\Psi_{k_n}\circ \overline\pi\rangle}(He^{i\langle t,\chi\rangle})\circ f_\Delta^{k_n}$ by $H_{({k_n})}(-t)\circ\overline\pi$.
	\item Control of the error in this approximation.\\
		Since $G$ and $H$ and $e^{i\langle t,\chi\rangle}$ are uniformly bounded and dynamically H\"older in the sense of~\eqref{regGH} and~\eqref{regbartau}, 
	it follows that
	\begin{equation}\label{approxiG}
	\left\Vert \widehat G_{({k_n})}(t)-G_{({k_n})}(t)\circ \overline\pi\right\Vert_\infty\le C\left(\Vert G\Vert_\infty |t|+ L_G\right)\beta^{k_n}\, .
	\end{equation}
	Therefore
	\begin{align*}
	&\left|\mathbb E_{\mu_\Delta}[G .
	e^{i\langle t,\widehat\Psi_n\rangle}.H\circ f_\Delta^n]
	-\mathbb E_{\mu_{\overline \Delta}}\left[G_{({k_n})}(t).e^{i\langle t,\overline\Psi_n\rangle}.H_{({k_n})}(-t)\circ f_{\overline\Delta}^n\right]\right|\\
	\nonumber&\quad\quad\le  C'(1+|t|)\left(\Vert G\Vert_\infty\Vert H\Vert_\infty+L_H\Vert G\Vert_\infty+L_G\Vert H\Vert_{\infty}\right)
	\beta^{k_n}\, .
	\end{align*}
	Therefore
	\begin{align}
	\label{errorGH}\mathbb E_{\mu_\Delta}&[G .h(\widehat\Psi_n-L)H\circ f_\Delta^n]\\
	\nonumber=&
	\frac 1{(2\pi)^{d+1}}\int_{\mathbb T^d\times\mathbb R}e^{-i\langle t,L\rangle}\widehat h(t)\mathbb E_{\mu_{\overline \Delta}}\left[G_{({k_n})}(t).e^{i\langle t,\overline\Psi_n\rangle}.H_{({k_n})}(-t)\circ f_{\overline\Delta}^n\right]\, dt\\
	&+\mathcal O\left(\beta^{k_n} 
	\left(\Vert G\Vert_\infty\Vert H\Vert_\infty+L_H\Vert G\Vert_\infty+L_G\Vert H\Vert_{\infty}\right)
	\right)\, ,\label{errorGHbis}
	\end{align}
	since $
	\int_{\mathbb R^d}(1+|t|).|\widehat h(t)|\, dt<\infty$.
\end{itemize}
\item \underline{Step 2}: Use of the transfer operator of $f_{\overline\Delta}$.\\
Due to \eqref{errorGH}, we are led to the study of the integral in $t\in\mathbb R^d$ of $\widehat h(t)$ multiplied by:
\begin{align}\label{simplifGH}
\mathbb E_{\mu_{\overline \Delta}}&\left[G_{({k_n})}(t).e^{i\langle t,\overline\Psi_n\rangle}.H_{({k_n})}(-t)\circ f_{\overline\Delta}^n\right]
=\mathbb E_{\mu_{\overline \Delta}}\left[ H_{({k_n})}(-t).P_t^n(G_{({k_n})}(t))\right]\\
\nonumber&\quad\quad\quad=\mathbb E_{\mu_{\overline \Delta}}\left[ H_{({k_n})}(-t).P_t^{n-3{k_n}}(P_t^{3{k_n}}(G_{({k_n})}(t)))\right]\, .
\end{align}
We already know that $\Vert H_{({k_n})}(-t)\Vert_{L^p}
\le\Vert H\Vert_{L^p}$. 
Let $m\ge 2$.
Let us prove that
	there exists $C_0>0$ such that, for every $n\ge 3$, $t\in\mathbb R^{d+1}$ and
	$\bar x,\bar y\in\overline\Delta$ such that $\overline s(\bar x,\bar y)>0$, the following inequality holds true 
\begin{equation}\label{HolderP2kg}
\left| P_t^{m{k_n}}(G_{({k_n})}(t))(\bar x)-P_t^{m{k_n}}(G_{({k_n})}(t))(\bar y)\right|
\le C_0(1+|t|)\left\vert P^{m{k_n}}(|
G_{({k_n})}(t)|)(x)\right\vert\, \beta^{\overline s(\bar x,\bar y)}\, .
\end{equation}
Indeed, we observe that
	\begin{align*}
\nonumber&\left| P_t^{m{k_n}}(G_{({k_n})}(t))(\bar x)-P_t^{m{k_n}}(G_{({k_n})}(t))(\bar y)\right|\\
\quad\quad&\le \sum_{\phi_{(m{k_n})}}\left| e^{
g_{m{k_n}}(\phi_{(m{k_n})}(\bar x))}e^{i\langle t,\overline\Psi_{mk_n}(
\phi_{(m{k_n})}(\bar x))
\rangle}(
{G}_{({k_n})}(t))(\phi_{(m{k_n})}(\bar x))\right.\\
\nonumber\quad&\left.-e^{
g_{m{k_n}}(\phi_{m{k_n}}(\bar y))}e^{i\langle t,
\overline\Psi_{m{k_n}}
(\phi_{(m{k_n})}(\bar y))
\rangle}(
{G}_{({k_n})}(t))(\phi_{(m{k_n})}(\bar y))\right|\, ,
\end{align*}
where the sum is taken over the  inverse branches $\phi_{(m{k_n})}$ of $f_{\overline \Delta}^{m{k_n}}$ and where $g$ satisfies 
\begin{equation}\label{holderg}
|e^{g(\overline x)}-e^{g(\overline y)}|\le L_g e^{g(x)}\beta^{\overline s(\overline x,\overline y)},\mbox{if }\overline{s}(\overline x,\overline y)>1\, .
\end{equation}
 We conclude by noticing that
	\[
	G_{({k_n})}(t)(\phi_{(m{k_n})}(\bar x))=
	G_{({k_n})}(t)(\phi_{(m{k_n})}(\bar y))\, ,
	\]
since $m\ge 2$ and that
	\begin{align}
	\nonumber\left|e^{i\langle t,\overline\Psi_{m{k_n}}(\phi_{(m{k_n})}(\bar x))\rangle}
	-e^{i\langle t,\overline\Psi_{m{k_n}}(\phi_{(m{k_n})}(\bar y))\rangle}
	\right|&\le |t|L_{\overline\Psi}\sum_{j=0}^{m{k_n}-1}\beta^{\bar s(\bar x,\bar y)+m{k_n}-j}\\
	&\le |t|L_{\overline\Psi}\beta^{\overline s(\bar x,\bar y)+1}/(1-\beta)\, ,\label{holder1}
	\end{align}
	and
	\begin{align}
	\nonumber\left|e^{g_{m{k_n}}(\phi_{(m{k_n})}(\bar x))}-e^{g_{m{k_n}}(\phi_{(m{k_n})}(\bar y))}
	\right|&\le e^{g_{m{k_n}}(\phi_{(m{k_n})}(\bar x))}\left|e^{g_{m{k_n}}(\phi_{(m{k_n})}(\bar y))-g_{m{k_n}}(\phi_{(m{k_n})}(\bar x))}-1
	\right|\\
	\nonumber&\le  e^{g_{m{k_n}}(\phi_{(m{k_n})}(\bar x))}L_g\sum_{j=0}^{m{k_n}-1}\beta^{\overline s(\overline x,\overline y)+m{k_n}-j+1}\\
	&\le e^{g_{m{k_n}}(\phi_{(m{k_n})}(\bar x))}L_g\frac{\beta^{\overline s(\bar x,\bar y)+1}}{1-\beta}\, ,\label{holder2}
	\end{align}
due to~\eqref{holderg}. 
This ends the proof of~\eqref{HolderP2kg}.
Therefore
\begin{align*}
\Vert P^{2{k_n}}(|
G_{({k_n})}|)\Vert_{\mathcal B}\le \Vert P^{2{k_n}}(|
G_{({k_n})}|)\Vert_{\mathcal B_0}=\mathcal O\left(\Vert G\Vert_\infty\right)\, ,
\end{align*}	
and
\begin{align}
\nonumber\Vert P_t^{3{k_n}}(G_{({k_n})}(t))\Vert_{\mathcal B}&=\mathcal O\left((1+|t|)\Vert P^{3{k_n}}(|
 G_{({k_n})}|)\Vert_{\mathcal B}\right)\\
\nonumber&=\mathcal O\left((1+|t|)\Vert P^{{k_n}}(P^{2{k_n}}(|
G_{({k_n})}|))\Vert_{\mathcal B}\right)\\
&\le  \mathcal O((1+|t|)(\mathbb E_\mu[|G|]+\mathcal O(\theta^{k_n}\Vert  G\Vert_{\infty})))\, ,\label{P3kGk}
\end{align}	
where we used~\eqref{quasicompactness} at the last line.
\item \underline{Step 3}: Restriction to a  neighbourhood of 0.\\
Let $K>0$ be such that the support of $\widehat g$ is contained in 	$\mathbb T^d\times[-K,K]$.
Using Sublemma~\ref{sunl:lamb}, we consider $b_0\in(0;\min(1,\beta_0))$ (see~\eqref{spgap-Sz},~\eqref{spgap-Sz-bis}) small enough so that there exists $a'>0$ such that, for all $t\in[-b_0,b_0]^{d+1}$, the following holds true
\begin{equation}\label{quasicompact}
P_t^n=\lambda_t^n\Pi_t+\mathcal O(\theta^n)\, ,\quad
\lambda_t=e^{- \Sigma_{d+1} t \cdot t\, \log (1/|t|)+\mathcal O(t^2)}\, ,
\end{equation}
\begin{equation}\label{Pidiff}
\Pi_t-\mathbb E_{\mu_{\overline\Delta}}[\cdot]1_{\overline\Delta}
=\mathcal O(|t|)\mbox{ in }
\mathcal L((\mathcal B
,\Vert\cdot\Vert_{\mathcal B})\rightarrow L^{q}(\mu_{\overline\Delta}))\, ,
\end{equation}
	\begin{equation}\label{lambda0}
	\theta\le e^{-2 \Sigma_{d+1}t\cdot t\log(1/|t|)}\le \left|\lambda_t\right|\le e^{-a'|t|^2\log(1/|t|)}\, ,
	\end{equation}
	and that
	$\forall y>x>b_0 ^{-1}$, $\frac 12 (x/y)^\epsilon \le \log(x)/\log(y)\le 2 (y/x)^{\varepsilon}$
	(using for example Karamata's representation of slowly varying functions).
	This last condition will imply that, for every $n$ large enough (so that $\mathfrak a_n>b_0 ^{-1}$) and for every $u\in[-b_0  \mathfrak a_n,b_0  \mathfrak a_n]^{d+1}$, the following 
	inequalities hold true
	\begin{equation}\label{logx/logy}
	\frac 12\min(|u|^\varepsilon,|u|^{-\varepsilon})\le  \log(\mathfrak a_n/|u|)/\log \mathfrak a_n  \le    2\max(|u|^\varepsilon,|u|^{-\varepsilon})\, .
	\end{equation}
	
			 Lemma~\ref{nonarithmeticity} ensures that the spectral radius of $P_t$ is smaller than 1 for every $t\ne 0$. This implies that there exists $\theta_0\in(0,1)$ such that $\sup_{b_0 <|t|_\infty<K}\Vert P_t^n\Vert=\mathcal O(\theta_0^n)$ (by upper semi-continuity of the spectral radius). 
			Thus
	\begin{align}\label{Intinterm}
			&\left|\int_{b_0 <|t|_\infty<K}e^{-i\langle t,L\rangle}\widehat h(t)
			\mathbb E_{\mu_{\overline \Delta}}\left( H_{({k_n})}(-t).P_t^{n-3{k_n}}(P_t^{3{k_n}}(G_{({k_n})}(t)))\right)\, dt\right|\\
&\quad			\le \Vert h\Vert_{L^1} K \mathcal O\left(\theta_0^{n-3{k_n}}\Vert H\Vert_{L^p}\sup_{|t|<b_0}\Vert P_t^{3{k_n}}(G_{({k_n})}(t))\Vert_{\mathcal B}\right)\, ,
			\end{align}
since $L^p$ is continuously included in the dual of the Young space $\mathcal B$.
Thus, we can focus on $[-b_0 ,b_0 ]^d$.
It follows from~\eqref{quasicompact} and~\eqref{Pidiff} that, in $L^q(\mu_{\overline\Delta})$,
	\begin{align*}
	P_t^{n-3{k_n}}(P_t^{3{k_n}}(G_{({k_n})}(t)))&=\lambda_t^{n-3{k_n}}\left(\mathbb E_{\mu_{\overline\Delta}}[P_t^{3{k_n}}(G_{({k_n})}(t))]+\mathcal O(|t|\, \Vert P_t^{3{k_n}}(G_{({k_n})}(t))\Vert_{\mathcal B})
	\right)
	\\
	&+\mathcal O\left(\theta^{n-3{k_n}}(1+|t|)\Vert P_t^{3{k_n}}(G_{({k_n})}(t))\Vert_{\mathcal B}\right)
	\, ,
	\end{align*}
	and so
	\begin{align}
	\label{ERR1}
	&\int_{[-b_0 ,b_0 ]^{d+1}}e^{-i\langle t,L\rangle}
	\widehat h(t)\left(\mathbb E_{\mu_{\overline\Delta}}\left[H_{({k_n})}(-t).P_t^{n-3{k_n}}(P_t^{3{k_n}}(G_{({k_n})}(t)))\right]\right.\\
	&\quad\quad\left.-\mathbb E_{\mu_{\overline\Delta}}\left[H_{({k_n})}(-t)\right]\lambda_t^{n-3{k_n}}\mathbb E_{\mu_{\overline\Delta}}[P_t^{3{k_n}}(G_{({k_n})}(t))]\right)\, dt\\
	\nonumber &\quad\quad\quad\quad=\mathcal O\left(\Vert H\Vert_{L^{p}}\mathfrak a_n^{-d-2}\sup_{|t|<b_0}\Vert P_t^{3{k_n}}(G_{({k_n})}(t))\Vert_{\mathcal B}\right)\, ,
	\end{align}
	since
	\begin{align}
	\nonumber\int_{[-b_0 ,b_0 ]^{d+1}}&\left(|t|
	|\lambda_t|^{n-3{k_n}}+\theta^{n-3{k_n}}(1+|t|)\right)\, dt\le \int_{[-b_0 ,b_0 ]^{d+1}}|t|
	e^{-a' n|t|^2\log(1/|t|)}\, dt+ \mathcal O\left(\theta^{\frac n2}\right)\\
	\nonumber&\le \mathfrak a_n^{-d-2}\int_{[-b_0  \mathfrak a_n,b_0  \mathfrak a_n]^{d+1}}|u|
	e^{-a' n|u/\mathfrak a_n|^2
	\log(|\mathfrak a_n|/|u|)}\, du+ \mathcal O\left(\theta^{\frac n2}\right)\\
	\nonumber&\le \mathfrak a_n^{-d-2}\int_{[-b_0  \mathfrak a_n,b_0  \mathfrak a_n]^{d+1}}|u|
	e^{-\frac{a' |u|^2}{2\log(\mathfrak a_n)}\log(|\mathfrak a_n|/|u|)}\, du+ \mathcal O\left(\theta^{\frac n2}\right)\\
	\nonumber&\le \mathfrak a_n^{-d-2}\int_{[-b_0  \mathfrak a_n,b_0  \mathfrak a_n]^{d+1}}|u|
	e^{-\frac {a'}4 |u|^{2-\varepsilon}
	}\, du+  \mathcal O\left(\theta^{\frac n2}\right)\\
	\label{ERR1b}&=\mathcal O\left(\mathfrak a_n^{-d-2}\right)\, ,
	\end{align}
	where we used~\eqref{logx/logy}. 
	Furthermore, it follows from the definition of $H_{(k_n)}$ that
	\begin{align}
	\mathbb E_{\mu_{\overline\Delta}}\left[H_{({k_n})}(-t)\right] 
	&=\mathbb E_{\mu_{\Delta}}\left(e^{i\langle t,\overline\Psi_{{k_n}}\circ \overline\pi\rangle}(He^{i\langle t,\chi\rangle})\circ f_\Delta^{k_n}\right)\\ 
&=\mathbb E_{\mu_\Delta}[H ]+\mathcal O({k_n}t\Vert H\Vert_{L^{p}})
\, ,
	\end{align}
since $\chi$ is uniformly bounded and $\Vert \widehat\Psi_{k_n}\Vert_{L^q(\mu)}\ll k_n\Vert\widehat \Psi\Vert_{L^q(\mu)}$ since $q<2$.
Moreover, due to~\eqref{approxiG}, for all $t\in[-b_0;b_0]^{d+1}$,
	\begin{align*}
	\mathbb E_{\mu_{\overline\Delta}}[P_t^{3{k_n}}(G_{({k_n})}(t))]&=\mathbb E_{\mu_\Delta}[(e^{i\langle t,\overline\Psi_{2{k_n}}\circ\overline\pi-\chi\rangle}G)\circ f_\Delta^{k_n} ]+\mathcal O(\beta^{k_n})\\ 
	&
=\mathbb E_{\mu}\left[G\right]+\mathcal O\left(t\Vert G\Vert_{L^{1}}+t\Vert \widehat\Psi_{2{k_n}}.G\Vert_{L^1(\mu)}\right)\, .
\end{align*}
Combining this last two estimates with \eqref{ERR1} via~\eqref{ERR1b}, we infer that
	\begin{align*}
	&\int_{[-b_0 ,b_0 ]^{d+1}}e^{-i\langle t,L\rangle}\widehat h(t)\left(
	\mathbb E_{\mu_{\overline\Delta}}\left[H_{({k_n})}(-t).P_t^{n-3{k_n}}(P_t^{3{k_n}}(G_{({k_n})}(t)))\right]-\mathbb E_{\mu_{\Delta}}\left[H\right]\mathbb E_{\mu_{\Delta}}\left[G\right]\lambda_t^{n-3{k_n}}\right)\, dt\\
	\nonumber&
	\quad=\mathcal O\left(\mathfrak a_{n}^{-d-2}\left[\Vert H\Vert_{L^{p}}\sup_{|t|<b_0}\Vert P_t^{3{k_n}}(G_{({k_n})}(t))\Vert_{\mathcal B}+{k_n} \Vert G\Vert_{L^1}\Vert H\Vert_{L^{p}} +
	\Vert H\Vert_{L^1}
	\Vert \widehat\Psi_{2{k_n}}.G\Vert_{L^1(\mu)}\right]\right)
\, .	\end{align*}
	It follows from this last estimate combined with \eqref{errorGH} and \eqref{simplifGH} that
	\begin{align}\label{ERR3}
	\mathbb E_{\mu_\Delta}&[G .h(\widehat\Psi_n-L).H\circ f_\Delta^n]
	- \mathfrak a_{n}^{-d-1}  \mathbb E_{\mu_{\Delta}}\left[H\right]\mathbb E_{\mu_{\Delta}}[G]\int_{[-b_0  \mathfrak a_{n},b_0  \mathfrak a_{n}]^{d+1}}e^{-i\mathfrak a_{n}^{-1}\langle t,L\rangle}
	\widehat h(t/\mathfrak a_{n})
	\lambda_ {t/\mathfrak a_{n}}^{n-3{k_n}}\, dt\\
	& =\mathcal O\left(\beta^{k_n} \Vert G\Vert_{Holder}.\Vert H\Vert_{Holder}+\mathfrak a_{n}^{-d-2}\left[\Vert H\Vert_{L^{p}}\sup_{|t|<b_0}\Vert P_t^{3{k_n}}(G_{({k_n})}(t))\Vert_{\mathcal B}\right.\right.\\
	&
	\left.\left.+{k_n} \Vert G\Vert_{L^1}\Vert H\Vert_{L^{p}} +
	\Vert H\Vert_{L^1}
	\Vert \widehat\Psi_{2{k_n}}.G\Vert_{L^1(\mu)}\right]\right)\, .
	\end{align}
Thus, due to~\eqref{P3kGk}, the above formula~\eqref{ERR3} is bounded by
\[
\mathcal O\left(\max(\beta,\theta)^{k_n} \Vert G\Vert_{Holder}.\Vert H\Vert_{Holder}+\mathfrak a_{n}^{-d-2}\left[{k_n} \Vert G\Vert_{L^1}\Vert H\Vert_{L^{p}} +
\Vert H\Vert_{L^1}
\Vert \widehat\Psi_{2{k_n}}.G\Vert_{L^1(\mu)}\right]\right)\, .
\]
	It remains to estimate
	\[
	\int_{[-b_0  \mathfrak a_{n},b_0  \mathfrak a_{n}]^{d+1}} e^{-i\mathfrak a_{n}^{-1}\langle t,L\rangle}
	\widehat h(t/\mathfrak a_{n})
	\lambda_ {t/\mathfrak a_{n}}^{n-3{k_n}}\, dt \, .
	\]
	To this end, let us notice that, due to~\eqref{logx/logy}, for $t\in[-\mathfrak a_n b_0,\mathfrak a_n b_0]^d$,
	\begin{align}
	\label{lambdan-k}\lambda_{t/\mathfrak a_{n}}^{n-3{k_n}} 
	&
	=\lambda_{t/\mathfrak a_{n}}^{n} \lambda_{t/\mathfrak a_{n-3{k_n}}}^{-3{k_n}} 
	=  e^{-\frac {n}{ \mathfrak a_{n}^2} \langle \Sigma_{d+1} t,t\rangle\, \left(\log (\mathfrak a_{n}/|t|)\right)+\mathcal O(n|t|^2/\mathfrak a_{n}^2)}  
	e^{\mathcal O(\frac {k_n}n\max(|t|^{2-\epsilon},|t|^{2+\epsilon}))}
	\\
	\nonumber&=  e^{-\frac 1{2 \log (\mathfrak a_{n})} \langle \Sigma_{d+1} t, t\rangle\, \left(\log \mathfrak a_{n}-\log(|t|)\right)
		+\mathcal O(\max(|t|^{2-\epsilon},|t|^{2+\epsilon})\eta_n)
	}\\
	\nonumber&=e^{-\frac 1{2} \langle\Sigma_{d+1} t, t\rangle\, \left(1
		-\frac{\log |t|}{\log \mathfrak a_{n}}\right)}
	+\mathcal O\left(e^{-\frac{a'\min(|t|^{2-\epsilon},|t|^{2+\epsilon})}2}\max(|t|^{2-\varepsilon},|t|^{2+\varepsilon})\eta_n\right)
\, ,
	\end{align}
with $\eta_n:=\frac 1{\log n}+\frac{k_n}n$. 
	Therefore, since $\widehat h$ is Lipschitz continuous and
	writing $[\widehat h]_{Lip}$ for its Lipschitz constant,
	it follows that
	\begin{align*}
	&\int_{[-b_0  \mathfrak a_{n},b_0  \mathfrak a_{n}]^{d+1}}e^{-i\mathfrak a_{n}^{-1}\langle t,L\rangle}\widehat h(t/\mathfrak a_{n})\lambda_{t/\mathfrak a_n}^{n-3{k_n}}\, dt\\
	=&\int_{\mathbb R^{d+1}}e^{-i\mathfrak a_n^{-1}\langle t,L\rangle}\left(\widehat h(0)+\mathcal O(t/\mathfrak a_n)\right)e^{
		-\frac 1{2}  \langle\Sigma_{d+1} t, t\rangle \left(1-\mathbf 1_{|t|<b_0  \mathfrak a_n}\frac{\log(|t|)}{\log (\mathfrak a_n)}\right)}\,  dt+\mathcal O\left(\left(|\widehat h(0)|+\frac{
	[\widehat h]_{Lip}}{\mathfrak a_n}\right)\eta_n\right)\nonumber\\
	=&\widehat h(0)\int_{\mathbb R^{d+1}}e^{-i\mathfrak a_n^{-1}\langle t,L\rangle}
	e^{-\frac 1{2}  \langle\Sigma_{d+1} t, t\rangle}
	\left(1
	+\frac {\mathbf 1_{|t|<b_0  \mathfrak a_n}}2   \langle\Sigma_{d+1} t, t\rangle\frac{\log( |t|)}{\log (\mathfrak a_n)}\right. \nonumber\\
	&\left. \ \ \ 
	+O\left(\mathbf 1_{|t|<b_0  \mathfrak a_n}\max\left(1,e^{\frac 12 \langle\Sigma_{d+1} t, t\rangle\frac{\log(|t|)}{\log \mathfrak a_n}}\right)
	\frac{
		|t|^4(\log(|t|))^2}{(\log n)^2}\right)\right)\, dt+O\left( |\widehat h(0)|\eta_n+\frac{[\widehat h]_{Lip}}{\mathfrak a_n}\right)\, ,\nonumber
	\end{align*}
	where we used 
	$
	e^{x}=1+x+O(\max(1,e^{x})x^2)$. Thus
	\begin{align}
	\int_{[-b_0  \mathfrak a_n,b_0  \mathfrak a_n]^{d+1}} &e^{-i\mathfrak a_n^{-1}\langle t,L\rangle}\widehat h(t/\mathfrak a_n)\lambda_{t/\mathfrak a_n}^{n-3{k_n}}\, dt=g_{d+1}\left(\frac L{\mathfrak a_n}\right)\int_{\mathbb Z^d\times\mathbb R}h\, d\lambda_{d+1}+\mathcal O\left(|\widehat h(0)|\eta_n+\frac{
	[\widehat h]_{Lip}}{\mathfrak a_n}\right)\, ,\label{intlambda}
	\end{align}
	with
	$
	g_{d+1}(z):=
	\frac{
		e^{-\frac 1{2}   \langle\Sigma_{d+1} z, z\rangle}}{\sqrt{(2\pi)^{d}\det \Sigma_{d+1}}}$.
\end{itemize}
This ends the proof of Lemma~\ref{lem:jointllt0}.
\end{pfof}

\begin{pfof}{Lemma~\ref{lem:jointllt}}
Let $A_0,B_0\subset M$ be measurable sets such that $\mu(\partial A_0)=\mu(\partial B_0)=0$. Let $K\subset\mathbb Z^d\times\mathbb R
$ be a bounded set with $\Lambda_{d+1}(\partial K)=0$ (boundary in $\mathbb Z^d\times\mathbb R$) and let $z\in
\mathbb R^{d+1}
$ and $(z_n)_n$ be a sequence of $\mathbb Z^d\times\mathbb R$ such that $\lim_{n\rightarrow +\infty}z_n/a_n=z$.
Let us prove that 
\begin{equation}\label{jointLLD1}
\lim_{n\rightarrow +\infty}a_n^{d+1}\mu\left(A_0\cap T^{-n}( B_0)\cap\{\widehat\Psi_n(x)\in z_n +K\}\right)= g_{d+1}(z)\mu(A_0)\mu(B_0)\lambda_{d+1}(K)
\, .
\end{equation}
	
	
We will approximate $A_0$ and $B_0$ by $A_n^\pm$ and $B_n^\pm$
respectively, where $A_n^-$ (resp. $A_n^+$) is the  union of all connected components of $M\setminus\bigcup_{k=-m_n}^{m_n}T^{-k}(\mathcal S_0)$ contained in (resp. intersecting) $A_0$
with $m_n\rightarrow +\infty$, analogously with $B_n^\pm$ with respect to $B_0$. Since the diameter of these connected components is smaller than $C\vartheta^n$ for some $C>0$ and some $\vartheta\in(0,1)$ and since $\mu(\partial A_0)=\mu(\partial B_0)=0$, we conclude that $\mu(A_n^+\setminus A_n^-)$
and $\mu(B_n^+\setminus B_n^-)$ vanishes as $n\rightarrow +\infty$.
Consider $h$ as in Lemma~\ref{lem:jointllt0} taking nonnegative values. 
We set
\[
\mathfrak M_n(A,B):= a_n^{d+1}\mathbb E_{\mu}[\mathbf 1_{A}.h(\widehat\Psi_n-z_n
).\mathbf 1_{T^{-n}(B)}]\]
and
\[
\mathfrak M(A,B):= \mu(A)\mu(B)g_{d+1}\left(
 z
\right)\int_{\mathbb Z^d\times\mathbb R}h\, d\Lambda_{d+1}\, .
\]
Since $h\ge 0$, these two quantities are increasing in $A$ and in $B$. Thus
\[
\mathfrak M_n(A_m^-,B_m^-)-\mathfrak M(A_m^+,B_m^+)\le 
\mathfrak M_n(A_0,B_0)-\mathfrak M(A_0,B_0)\le 
\mathfrak M_n(A_m^+,B_m^+)-\mathfrak M(A_m^-,B_m^-)\, .
\]
It then follows from Lemma~\ref{lem:jointllt0} (with $k_n=\log n$) applied to $\mathbf 1_{A_m^\pm},\mathbf 1_{B_m^-}$ since these observables are dynamically H\"older and since $\mathfrak a_n\sim a_n$ that, for all $m\ge 1$,
\begin{align*}
(\mu(A_m^-)\mu(B_m^-)-\mu(A_m^+)\mu(B_m^+))g_{d+1}(z)\int_{\mathbb Z^{d}\times\mathbb R} g\, d\lambda_{d+1}& =
\lim_{n\rightarrow +\infty}
\mathfrak M_n(A_m^-,B_m^-)-\mathfrak M(A_m^+,B_m^+)\\
&\le 
\liminf_{n\rightarrow +\infty}\mathfrak M_n(A_0,B_0)-\mathfrak M(A_0,B_0)\, , 
\end{align*}
from which we conclude that 
$\liminf_{n\rightarrow +\infty}\mathfrak M_n(A_0,B_0)-\mathfrak M(A_0,B_0)\ge 0$, we proceed analogously with the $\limsup$ exchanging exponents $+$ and $-$, and we conclude that
\begin{equation}\label{CVCV}
\lim_{n\rightarrow +\infty}\mathfrak M_n(A_0,B_0)=\mathfrak M(A_0,B_0)=0\, ,
\end{equation}
and extend this to the case of complex valued function $g$.
Consider the function $H_0$ appearing in \eqref{coboundPsi}. 
Let $\delta=1/L>0$, where $L$ is an integer such that $L>2\Vert H_0\Vert_\infty$ and 
$K\subset (-L+2\Vert H_0\Vert_\infty,L-2\Vert H_0\Vert_\infty)^d\times(-\frac {2\pi}{\delta},\frac{2\pi}\delta)$. 
Let us consider the family of functions $(g_{\delta,\theta}:\mathbb R^{d+1}\rightarrow \mathbb C)_\theta$ given by
\begin{align}\label{eq:polya}
 g_{\delta,\theta}(x)=
e^{i\langle \theta, x\rangle}h_\delta(x)\, ,
\end{align}
with $h_\delta:(x_1,...x_{d+1})\mapsto\frac{1-\cos(\delta x_{d+1})}{(2L-1)^2\pi \delta x_{d+1}^2}\mathbf  1_{|x_1|,...,|x_d|<L}$ (using the density of Polya's distribution). The Fourier transform of $g_{\delta,\theta}$ is $t\mapsto   \sum_{|k_1|,..,|k_d|<L}\frac{e^{i\sum_{i=1}^dk_it_i}}{(2L-1)}
\max(0,1-|(t_{d+1}+\theta)/\delta|)$.
The above convergence result~\eqref{CVCV} with  $h=g_{\delta ,\theta}$ for all $\theta$ implies the
convergence in distribution of $(\mathfrak m_n)_n$ to $\mathfrak m$
(since it ensures the convergence of characteristic functions),
where $\mathfrak m_n$ has density $\frac{h_\delta}{a_n^{d+1}\mathbb E_\mu[\mathbf 1_{A_0\cap T^{-n}(B_0)}h_\delta(\widehat\Psi_n-z_n)]}$
with respect to the image measure of $a_n^{d+1}\mathbf 1_{A_0\cap T^{-n}B_0}\mu$ by $\widehat\Psi_n-z_n$, and where 
 $\mathfrak m$ is the probability measure with density $h_\delta/g_{d+1}(z)$ with respect to $g_{d+1}(z)\Lambda_{d+1}$. Thus, since $K\subset (-L,L)^d\times(-\frac {2\pi}{\delta},\frac{2\pi}\delta)$, the previous distribution convergence implies that  $\lim_{n\rightarrow +\infty}\int_K\frac 1{h_\delta} \, d\mathfrak m_n=\int_K\frac 1{h_\delta} \, d\mathfrak m$, i.e.
\[
\lim_{n\rightarrow +\infty}\frac{a_n^{d+1}\mu\left(A_0, \widehat \Psi_n-z_n\in K, T^{-n}(B_0) \right)}{a_n^{d+1}\mathbb E_\mu[\mathbf 1_{A_0\cap T^{-n}(B_0)}h_\delta(\widehat\Psi_n-z_n)]}=\Lambda_{d+1}(K)\, ,
\]
and so, using again~\ref{CVCV} for the denominator,
\[
\lim_{n\rightarrow +\infty}\frac{a_n^{d+1}\mu\left(A_0, \widehat \Psi_n-z_n\in K, T^{-n}(B_0) \right)}{g_{d+1}(z)\mu(A_0)\mu(B_0)}=\Lambda_{d+1}(K)\, .
\]
This ends the proof of pointwise MLLT for $\widehat\Psi_n$~\eqref{jointLLD1}.

It remains to prove the uniformity in the convergence results.
Assume that~\eqref{eq:jointLLDunif} does not converge to 0 uniformly in $z\in\mathbb Z^d\times\mathbb R\, :\, |z|\le L a_n$, as $n\rightarrow +\infty$. 
Then, there would exist a sequence $(z_n)_n$ in $\mathbb Z^d\times\mathbb R$ such that $|z_n|<La_n$, a sequence of integers $m(n)$ and a real number $\eta>0$ such that
\[
\forall n\ge N_0,\quad \left|a_n^{d+1}\mu\left(A_0\cap T^{-n}(B_0)\cap\{\widehat\Psi_n\in z_n +K\}\right)- g_{d+1}\left(\frac{z}{a_n}\right)\mu(A_0)\mu(B_0)\lambda_{d+1}(K)\right|>\eta\, .
\]
This ends the proof of Lemma~\ref{lem:jointllt}.
\end{pfof}

\appendix

\section{Proof of joint LLD (Lemma~\ref{lem:lld})}\label{sec:jLLD}
In this appendix, we prove Lemma~\ref{lem:lld}. The proof is very similar to that of~\cite{MPT} except that the function $\overline\Psi$ is not constant on partition elements. 
For completeness, we explain in this appendix which adaptations have to be done to~\cite{MPT} to prove our joint LLD estimate stated in Lemma~\ref{lem:lld}.\\
We recall that optimal LLD for the cell change $\kappa$ (and so for the flight function $V$, due to~\eqref{coboundPsi}),
have been obtained in~\cite{MPT}. More precisely, by~\cite[Theorem 1.1 and Remark 1.2]{MPT},
for any $h>0$, there exists $C>0$ so that 
$\mu(V_n\in B(x,h)))\le C\frac{n}{a_n^d}\frac{\log|x|}{1+|x|^2}$, for any $n\ge 1$
and $x\in\R^d$. Here $B(x, h)$ denotes an open ball in $\R^d$ of radius $h$ centered at $x$.  Similarly, 
$\mu(\kappa_n=N)\le C\frac{n}{a_n^d}\frac{\log|N|}{1+|N|^2}$ for all $N\in\Z^d$ and all
$n\ge 1$.

The proof of LLD for $\kappa$ in~\cite{MPT} relies strongly on the fact that $\kappa$ goes to the quotient Young tower and that $\overline\kappa$ is constant on partition elements of the partition $\mathcal P$ for the Young tower $\Delta$ (as recalled in Section~\ref{sec:proofjCLT}); the statement on $V$ follows immediately since, up to 
a bounded coboundary, $V$ is the same as $\kappa$.
Due to~\eqref{coboundPsibis}, $\widehat\Psi\circ \pi$ can be written as 
$\overline \Psi\circ\overline\pi=\left(\overline\kappa,\overline\tau\right)\circ\overline\pi$ plus a bounded coboundary.
Thus, LLD for $\widehat\Psi$ will follow from LLD
for $\overline\Psi$. 
The function $\overline\tau$, and thus $\overline\Psi$, is not 
constant on partition elements. However, as argued below, the argument in~\cite{MPT}
goes through to provide LLD as in Lemma~\ref{lem:lld} for $\overline\Psi$ (and thus $\widehat\Psi$).

Throughout this section, let $d\in\{0,1,2\}$ and $U\subset\R^{d+1}$ be an open ball, as in the statement of Lemma~\ref{lem:lld}. To avoid a clash of notation below, the $z$ in the statement of Lemma~\ref{lem:lld} will be replaced by $x$. More precisely, here we shall prove
that, for any bounded set $U\subset \mathbb R^{d+1}$,
\begin{equation} \label{eq:MPT0}
\mu(\widehat\Psi_n\in x+U)\ll \frac{n}{a_n^{d+1}} \frac{(\log|x|)}{1+|x|^2}
\quad\text{uniformly in $n\ge1$, $x\in\R^{d+1}$\, .}
\end{equation}

As recalled in~\cite[Remark 1.3]{MPT}, the LLD in the range $ |x|\le a_n$
follows from the involved LLT, while the range $ |x|\ge a_n$ requires serious work.

\subsection{The range $ |x|\le a_n$}
It follows from the LLT estimate given in Lemma~\ref{lem:jointllt0}  that $\mu(\widehat\Psi_n\in B(x,h) )\ll a_n^{-d-1}\ll \frac n{a_n^{d+1}}\frac{|\log |x||}{1+|x|^2}$,
where the first inequality holds true uniformly in $n\ge 1,x\in\mathbb R^{d+1}$ and where the second one holds true for $n\ge 1$ and $x\in\mathbb R^d$ such that $|x|\le a_n$.\\
Indeed, $t\mapsto \frac{t^2}{|\log |t||}$ as limit $+\infty$ as $t\rightarrow +\infty$,  has derivative $t\mapsto\frac{2\log t-1}t$ and so it is increasing on $[e^{\frac 12},+\infty[$. Thus, for $n$ large enough, if $|x|\le a_n$, then $\frac{|x|^2}{|\log |x||}\le \frac{a_n^2}{|\log a_n|}=\frac{2n\log n}{\log n+\log\log n}$.

\subsection{The range $n\ll \log |x|$ } 

In this range we proceed similarly to~\cite[Lemma 3.1]{MPT}
obtain
\begin{lemma}
For any $\epsilon_1>0$ and any $q\ge 1$, there exists $C_q>0$ so that, for every $x\in\mathbb R^{d+1}$
and every $n\ge 1$ such that $\epsilon_1 n\le \log |x|$, 
 $\mu(\widehat\Psi_n\in x+U)\le \frac{C_q}{n^q |x|^2}$.
\end{lemma}
 \begin{proof}
There exists $x_0$ such that if $|x|>x_0$, then $|x|>2(\diam(U)+\epsilon_1^{-1}\log|x|)$. There exists a constant $C'_q$ such that $\mu(\widehat\Psi_n\in x+U)\le \frac{C'_q}{n^q |x|^2}$ for all
$x,n$ such that $\epsilon_1 n\le \log|x|\le \log x_0$.
It remains to treat the case $|x|\ge x_0$. One can observe that the proof of~\cite[Lemma 3.1]{MPT}
only uses the fact that $|\kappa|_\infty$ takes integer values, that $\mu(|\kappa|_\infty=p)\ll p^{-3}$ as $x\rightarrow +\infty$ and that $\mu(|\kappa|_\infty=p,\, |\kappa|_\infty\circ T^r\ge cp^{\frac 45})\ll p^{-3-\frac 2{45}}$. These properties are also satisfied
by $(\kappa,\lfloor \widetilde\tau\rfloor)$.
Thus, for every $q\ge 1$, there exists $C'_q>0$ such that, for all
$y\in\mathbb R^d$ and all $n\ge 1$ such that $\epsilon n\le \log|y|$, 
 $\mu((\kappa_n,(\lfloor \widetilde\tau\rfloor)_n)=y)\le  \frac{C''_q}{n^q |y|^2}$ and so
\[
\mu(\Psi_n\in x+U)\le \sum_{y\in (x+U+\{0\}^d\times[-n;0])\cap \mathbb Z^{d+1}} \frac{C''_q}{n^q |y|^2}\le (n+1)(\diam(U)+1)^{d+1} \frac{4C''_q}{n^q |x|^2}\, .\]
since $y$ in the first sum above satisfies
\[|y|\ge |x|-\diam(U)-n\ge |x|-\diam(U)-\frac{\log |x|}{\epsilon_1} \ge |x|/2\, .\]
We conclude by taking e.g. $C_q:=\max\left(C'_q,8C''_{q+1}(\diam(U)+1)\right)$. \end{proof}

\subsection{The range $a_n\le |x|\le e^{\epsilon_1 n}$ for a particular $\epsilon_1$}

 This $\eps_1$ is to be fixed so that it matches with the choice of $\eps_1$ in ~\cite[Proposition 6.2]{MPT}. As we shall explain below, it does not play a role in the current argument, but see~\eqref{vareps1} for a particular choice.

Since $\widehat\Psi\circ\pi$ is equal to $\overline\Psi\circ\overline\pi$
plus a bounded coboundary, the desired LLD~\eqref{eq:MPT0} for $\widehat\Psi$ will follow from LLD for $\overline\Psi$. So, in this range, we focus on LLD estimates for $\overline\Psi$.
As clarified in Appendix~\ref{sec:app}, $\overline\Psi$ is uniformly Lipschitz on Young's partition elements.

We work in the set up of Section~\ref{sec:proofjCLT}. As in Section~\ref{sec:proofjCLT}, for simplicity, we assume that\footnote{\label{footnote}This assumption is not essential. One could, as in~\cite{MPT} work
without, but in that case~\eqref{spgap-Sz} becomes slightly more complicated as there exists no longer a simple isolated eigenvalue at $1$, but finitely many eigenvalues of finite multiplicity.} the g.c.d. of $R$ is $1$.

Set $\delta=b_0/4$ with $b_0\in(0,\min(1,\beta_0))$ (with $\beta_0$ introduced before~\eqref{spgap-Sz}) and such that there is a constant $c>0$  such that\footnote{The existence of such a couple $(b_0,c)$ comes from Sublemma~\ref{sub:asl}.}
\begin{equation}\label{defc}
\forall t\in\mathbb R^d,\quad |t|<b_0\ \Rightarrow\ |\lambda_{t}|\le e^{-c|t|^2L_t}\, ,\quad\mbox{with }L(t):=\log(|t|^{-1})=|\log|t||\, .
\end{equation}
Using~\eqref{spgap-Sz} and~\eqref{spgap-Sz-bis}, we consider a function
$r:\R^{d+1}\to\C$ is $C^2$ with $\supp r\subset [-b_0,b_0]^{d+1}$ such that\footnote{The existence of such an $r$ is guarantied by a classical smoothing argument; see, for instance,~\cite[Footnote 1]{MTejp})}
\begin{align}\label{eq:fourlld}
 \mu(\overline\Psi_n\in x+U) &\le \int_{[-\delta,\delta]^{d+1}} e^{-i\langle t, x\rangle}r(t)P_t^n 1\,dt= A_{n,x}+ \mathcal O(\theta^n)\, ,
\end{align}
with
\[A_{n,x}:=\int_{[-\delta,\delta]^{d+1}} e^{-i\langle t, x\rangle}r(t)\lambda_{t}^n\Pi_{t}1\,dt\, .\]
The desired LLD~\eqref{eq:MPT0}
in this range will follow from~\eqref{eq:fourlld} together with the following
estimates on $\lambda_t$ and $\Pi_t$.
The following result corresponds to the hardest estimate in the set-up of~\cite{MPT}.
Let $\partial_j=\partial_{t_j}$ for $j=1,\dots,d+1$.
For $t,h\in\R^{d+1}$, $\mathbf{b}>0$, set 
\[
M_{\mathbf{b}}(t,h)=|h|L_h
\big\{1+ L_h\,|t|^2L_t +|h|^{-\mathbf{b}|t|^2L_t}L_h^2\,|t|^4L_t^2\big\}\, .
\]
\begin{lemma}{~Analogue of~\cite[Lemma 4.1]{MPT}.}
	\label{lem:key} 
	Let $j\in\{1,\dots,d\}$.
	The maps
	$t\mapsto \lambda_{t}$ and $t\mapsto\Pi_{t}:\cB_0\to L^1$ are $C^1$ on $[-b_0,b_0]^{d+1}$. 
	
	Furthermore, there exist $C>0$ and 
	$\mathbf{b}>0$ such that 
	for all $t,h\in [-b_0,b_0]^{d+1}$, 
	\[
	|\partial_j\lambda_{t+h}-\partial_j\lambda_{t}|  \le CM_{\mathbf{b}}(t,h),
	\qquad
	\|\partial_j\Pi_{t+h}-\partial_j\Pi_{t}\|_{\cB_0\mapsto L^1}  \le 
	CM_{\mathbf{b}}(t,h).
	\]
\end{lemma}
In \cite{MPT}, we obtained the same formula for $M_{\mathbf{b}}(t,h)$,
while working with $\kappa\in\Z^d$ instead 
$\overline\Psi\in\R^{d+1}$). 
Given Lemma~\ref{lem:key} in the range $a_n\le |x|\le e^{\varepsilon_1 n}$
for
\begin{equation}\label{vareps1}
 \varepsilon_1:=c/\mathbf{b}
\end{equation}
 with $\mathbf{b}>0$ as in Lemma~\ref{lem:key} and $c>0$ as in~\eqref{defc},
 the desired LLD for $\overline\Psi$ (as in~\eqref{eq:MPT0} with $\overline\Psi$ 
 instead of $\widehat\Psi$)
follows word for word as in \cite[Section 6 via Corollary 4.3]{MPT} (written for $\kappa$).
Indeed, the proofs therein just use the statement of Lemma~\ref{lem:key},~\eqref{defc} and the fact that $a_n=\sqrt{n\log n}$.

\subsection{Proof of Lemma~\ref{lem:key}}
\label{sec:pfkey}
We follow the proofs of~\cite{MPT}. We focus on the changes
that are needed since $\overline\Psi$ is not constant on the atoms of the Young partition. We will indicate which part of proofs of~\cite{MPT} can be followed line by line.
Throughout this section let  $Q:L^1(\overline Y)\to L^1(\overline Y)$ be the transfer operator corresponding to the Gibbs-Markov map $\overline F:\overline Y\to \overline Y$ (the base map of the quotient tower $(\overline\Delta,f_{\overline\Delta})$). We recall that 
$R$ is the return time of $f_{\overline \Delta}$ to $\overline Y$, so that
$F(\cdot)=f_{\overline{\Delta}}^{R(\cdot)}(\cdot)$. 
The proof of~\cite[Lemma 4.1]{MPT} of which Lemma~\ref{lem:key}
is an analogue (in the set up of Lemma~\ref{lem:lld})  starts from the following 
renewal equation 
\begin{equation}\label{eq:ren}
 \widehat P(z,t)=\sum_{n\ge 0}z^nP_t^n=\widehat A(z,t)\widehat T(z,t)\widehat B(z,t)+\widehat E(z,t), z\in \C, |z|\le 1, t\in [-\delta_0,\delta_0]^{d+1}
\end{equation}
where the operators $\widehat A,\widehat T, \widehat B,\widehat E$ and $\delta_0>0$ are to be defined/specified in the subsections to follow. In particular $\widehat T$ will be given by  
\[
\widehat T(z,t)=(I-\widehat Q(z,t))^{-1}\, ,
\]
with
\begin{equation}\label{def:hatQ}
\widehat Q(z,t):=Q\left(z^{R(\cdot)}e^{i\langle t,\overline\Psi_{R}\rangle}\cdot \right)=\sum_{n\ge 1}z^nQ\left(\mathbf 1_{\{R=n\}}e^{i\langle t,\overline\Psi_{n}\rangle}\cdot\right)\, ,
\end{equation}
where we write $\overline\Psi_R$ for the function defined by
 \[\forall y\in \overline Y,\quad  \overline\Psi_{R(y)}(y)=\sum_{k=0}^{R(y)-1}\overline\Psi(y,k )\, . \]
Following the approach in~\cite{MPT}, the proof of Lemma~\ref{lem:key} consists in using~\eqref{eq:ren} to clarify that 
\[\lambda_t=(g_0(t))^{-1}\]
where $t\mapsto g_0(t)$ is continuous, satisfies $g_0(0)=1$ and is  such that 1 is the dominating eigenvalue of $\widehat Q(g_0(t),t)$ and that
\[
\Pi_t=\lambda_t \widehat A(g_0(t),t)\widetilde\pi_0(t)\widehat B(g_0(t),t)
\]
where $\widetilde\pi_0(t)$ is such that $\widehat H(z,t):=\widehat T(z,t)-(g_0(t)-z)^{-1}\widetilde\pi(t)$ is analytic in $z$. 
As in~\cite{MPT}, 
the regularity (in terms of $M_b$) of the derivatives of $\lambda$ (stated in Lemma~\ref{lem:key}) will follow, via the use of the implicit function theorem, from the study of $\widehat Q(z,t)$ for $z$ close to 1, and from the properties satisfied by $\overline\Psi$. The analogous property for $\Pi$ will follow from the properties satisfied by $\lambda=1/g_0$ and also from the study of the  derivatives in $t$ of  $\widehat A,\widehat T, \widehat B,\widehat E$ for $z$ close to 1.
\subsubsection{Renewal operators}
\label{sec:Y}
 Throughout, we write $\alpha$ for the  partition corresponding the Gibbs Markov map $(\overline F,\overline Y,\alpha,\mu_{\overline Y})$. Also, let $s'(y,y')$ be the usual separation time of points $y,y'\in \overline Y$
 (see, for instance,~\cite[Section 2]{MPT} for definition) and for $\beta\in (0,1)$, 
 let $d_\beta(y,y'):=\beta^{s'(y,y')}$.
 Let $\cB_1(\overline Y)$ be the Banach space of bounded observables $v:\overline Y\to\R$ 
 Lipschitz with respect to the metric $d_\beta$.
It follows from the fact that $R$ has exponential tail probability that 
$\mu_{\overline Y}(|\overline\Psi|_R>n)=\mathcal O(n^{-2})$ as $\mu(|\overline\Psi|>n)$ (see e.g.~\cite[Section 2]{PeneTerhesiu21} for this argument).

\begin{prop} \label{prop:g}
	There exists $C_{\overline\Psi}>0$ such that
	\[	
\forall a\in\alpha,\quad\forall y,y'\in a,\quad	\sum_{\ell=0}^{R(a)-1}|\overline\Psi(y,\ell)
	-\overline\Psi(y',\ell)|  \le C_{\overline\Psi}
	d_\beta(y,y')\, .
	\]
\end{prop}
\begin{proof}
	Recall that $\overline\Psi$
	 is-Lipschitz with respect to Young's metric $\beta^{s(\cdot,\cdot)}$. Let us write $C'_{\overline\Psi}$
	 for its Lipschitz constant.
	Thus, for $a\in\alpha$ and $y,\,y'\in a$,
	\begin{align*}
	|\overline\Psi_R(y)-\overline\Psi_R(y')|
	& \le \sum_{\ell=0}^{R(a)-1}|\overline\Psi(y,\ell)-\overline\Psi(y',\ell)|\\
	 & \le C'_{\overline\Psi} \sum_{\ell=0}^{R(a)-1}
	\beta^{R(a)-\ell +\overline s(y,y')}	\le \frac{C'_{\overline\Psi}d_\beta(y,y')}{1-\beta}\, ,
	\end{align*}
since $s'(y,y')\le \overline s(y,y')-R(a)+1$.
	\end{proof}

Recall that 
\begin{equation}\label{FormuleQ}
Q(u)(y)=\sum_{a\in\alpha} \xi(y_a)u(y_a)\, ,\end{equation}
where $y_a$ is the preimage of $y$ under $F$ that belongs to $a$, and with $\xi=e^{g_R}$ with $g_R(y):=\sum_{k=0}^{R(y)-1}g(y,k)$
where $g$ satisfies~\eqref{holderg}. Thus
\begin{equation} \label{eq:GM}
0<\xi(y_a)=e^{g_R(y_a)}\le C\mu_{\overline Y}(a), \qquad
|\xi(y_a)-\xi(y'_a)|\le C\mu_{\overline Y}(a)d_\beta(y,y'),
\end{equation}
for all $y,y'\in {\overline Y}$, $a\in\alpha$, where we define 
$g_R$ as we have defined $\overline\Psi_R$ 
considering the function $g$
instead of $\overline\Psi$.\\
The next result extends~\cite[Proposition 5.1]{MPT} which was stated, in the context therein, for a function $u$ constant on elements of $\alpha$ (for which the local Lipschitz constant $K_u$ is null).

\begin{prop} \label{prop:Ruv} There exists $C>0$ such that
	\[
\Vert 	Q(u)\Vert_{\mathcal B_1({\overline Y})}\le C\Vert  |u|+K_u\Vert_{L^1(\mu_{\overline Y})}\, ,
	\]
for every $u\in L^1({\overline Y})$ such that $K_u\in L^1(\mu_{\overline Y})$
with
\[
\forall a\in\pi,\ \forall x\in a,\quad K_u(x):=\sup_{y,y'\in a}\frac{|u(y)-u(y')|}{\beta^{s'(y,y')}}\, .
\]
In particular, for all $v\in\mathcal B_1({\overline Y})$ (set of Lipschitz functions on ${\overline Y}$),
\begin{equation}\label{Quv}
\Vert Q(uv)\Vert_{\mathcal B_1({\overline Y})}\le 
C\Vert  |uv|+K_{uv}\Vert_{L^1(\mu_{\overline Y})}\le 
C\Vert  |u|+K_{u}\Vert_{L^1(\mu_{\overline Y})}\Vert v\Vert_{\mathcal B_1({\overline Y})}\, .
\end{equation}
\end{prop}
\begin{pfof}{Proposition~\ref{prop:Ruv}}
	Let $y,y'\in {\overline Y}$. For $a\in\alpha$, we write $y_a,y_a'\in a$ for the respective preimages of $y,y'$ under $F$.
	It follows from~\eqref{FormuleQ} and from the first part of~\eqref{eq:GM} that
	\[
	\left\Vert Q(u)\right\Vert_\infty
	\ll \sum_{a\in\alpha} \mu_{\overline Y}(a)\,\sup_a|u| \le \left\Vert |u|+K_u\right\Vert_{L^1(\mu_{\overline Y})}\, ,
	\]
since, for all $y,x\in a\in\alpha$, $|u(y)|\le |u(x)|+K_u(x)$.\\
Next, let $a\in\alpha$ and $y,y'\in a$, using the second part of~\eqref{eq:GM}, we obtain
	\begin{align*}
\left| Q(u)(y)-Q(u)(y')\right| & \le \sum_{a\in\alpha}\mu_{\overline Y}(a)\,\left|u(y_a)-u(y'_a)\right|\\ 
 & \le \sum_{a\in\alpha}\mu_{\overline Y}(a)K_u(a)\beta^{s'(y_a,y'_a)}\\
 & \le \sum_{a\in\alpha}\mu_{\overline Y}(a)K_u(a)\beta^{s'(y,y')+1}\, ,
	\end{align*}
which ends the proof.
\end{pfof}

For $z\in\C$ with $|z|\le1$ and $t\in\R^d$, the operator
$\widehat Q(z,t)$ formally defined in~\eqref{def:hatQ} defines
an operator on $L^1(\mu_{\overline Y})$ and can be decomposed in
$\hQ(z,t)
=\sum_{n=1}^\infty z^n Q_{t,n}
$, where we set 
\[
Q_{t,n}:=Q\left(1_{\{R=n\}}e^{i\langle t,\overline \Psi_R\rangle}\cdot\right)\, .
\]
The next result replaces~\cite[Proposition 5.2]{MPT}. The conclusion is the same, but, in our context, we have to deal with $K_u$, so we include entirely its proof.
\begin{prop} \label{prop:Rt}
	There exists $\delta_0>0$ such that when
	regarded as  functions with values in the set of continuous linear operators on $\cB_1({\overline Y})$,
	\begin{itemize}
		\item[(a)] $z\mapsto \hQ(z,t)$
		is analytic on $B_{1+\delta_0}(0)$ for all $t\in \R^d$;
		\item[(b)] $(z,t)\mapsto (\partial_z^k\hQ)(z,t)$ is $C^1$ on $B_{1+\delta_0}(0)\times\R^d$ for $k=0,1,2$;
		\item[(d)]
		$z\mapsto (\partial_j\hQ)(z,t)$ is $C^1$ on $B_{1+\delta_0}(0)$ uniformly in $t\in B_1(0)$ for $j=1,\dots,d$.
	\end{itemize}
\end{prop}

\begin{proof}
	It suffices to show that there exist $a>0$, $C>0$ such that
	\[
	\|Q_{t,n}\|_{{\cB_1({\overline Y})}}\le C(|t|+1)e^{-an},  \qquad
	\|\partial_jQ_{t,n}\|_{{\cB_1({\overline Y})}}\le C(|t|+1)e^{-an}, 
	\]
	for all $t\in\R^d$, $j=1,\dots,d$, $n\ge1$.
	Since $R$ is constant on partition elements, it follows from Proposition~\ref{prop:Ruv}(a) that
	\[
	\left\Vert Q_{t,n}\right\Vert_{\mathcal L(\mathcal B_1({\overline Y}))}\ll
	\left\Vert 1_{\{R=n\}}(1+K_{e^{i\langle t,\overline\Psi_R\rangle}}) \right\Vert_{L^1(\mu_{\overline Y})}\, ,
	\]
and that
	\begin{align*}
\left\Vert	\partial_j Q_{t,n}\right\Vert_{\mathcal L(\mathcal B_1({\overline Y}))}&=\left\Vert iQ\left(1_{\{R=n\}}(\overline\Psi_R)_j e^{i\langle t, \overline\Psi_R\rangle}\right)\right\Vert_{\mathcal L(\mathcal B_1({\overline Y}))}\\
&\ll \left\Vert 1_{\{R=n\}}\left(|\overline\Psi_R|+ 
K_{(\overline\Psi_R)_je^{i\langle t,\overline\Psi_R\rangle}}\right) \right\Vert_{L^1(\mu_{\overline Y})}\, .
	\end{align*}
But it follows from Proposition~\ref{prop:g} that	$K_{e^{i\langle t,\overline\Psi_R\rangle}}\le C_{\overline\Psi} |t|$
and that 
\begin{equation}\label{Kpsiexp}
K_{(\overline\Psi_R)_je^{i\langle t,\overline\Psi_R\rangle}}
\le C_{\overline\Psi}(1+|t|\, |\overline\Psi_R|)\, .
\end{equation}
Therefore
\[
	\left\Vert
Q_{t,n}\right\Vert_{\mathcal L(\mathcal B_1({\overline Y}))}
\ll  (1+|t|)\mu_{\overline Y}(R=n)\, ,
\]
and
\[
\left\Vert	\partial_j Q_{t,n}\right\Vert_{\mathcal L(\mathcal B_1({\overline Y}))}\ll (1+|t|)\left\Vert 1_{\{R=n\}}(1+|\overline\Psi_R|) \right\Vert_{L^1(\mu_{\overline Y})}\, ,
\]
We complete the proof by noticing that, since $\overline\Psi\in L^r({\overline Y})$ for all $r<2$ and $R$ has exponential tails,
there exists $a>0$ such that
$\|1_{\{R=n\}} (1+\overline\Psi_R)\|_1\ll e^{-an}$.
\end{proof}

For $z\in\C$ with $|z|\le1$ and $t\in\R^d$, define
\[
\hA(z,t)  :L^1({\overline Y})\to L^1(\overline\Delta), \qquad \hA(z,t)(v)
=\sum_{n=1}^\infty z^n A_{t,n}(v)
\]
where 
\(
A_{t,n}(v)(y,\ell) =1_{\{\ell=n\}}P_t^n(v)(y,\ell)
=1_{\{\ell=n\}} e^{i\langle t,\overline\Psi_n(y,0)\rangle}v(y).
\)

\begin{prop} \label{prop:A}
	There exists $\delta_0>0$ such that 
	when regarded as functions with values in the set of continuous linear operators from $L^\infty({\overline Y})$ to $L^1(\overline\Delta)$,
	\begin{itemize}
		\item[(a)] $z\mapsto \hA(z,t)$
		is analytic on $B_{1+\delta_0}$ $t\in\R^d$;
		\item[(b)] $(z,t)\mapsto (\partial_z\hA)(z,t)$ is $C^1$ on $B_{1+\delta_0}(0)\times\R^d$.
	\end{itemize}
\end{prop}

\begin{proof} 
The proof goes  word for word as \cite[Proof of Proposition~5.3]{MPT} since it just uses the H\"older inequality combined with the fact that $\Vert 1_{R>n}\overline\Psi_R\Vert_{L^1({\overline Y})}$ decays exponentially fast in $n$ as $n\rightarrow +\infty$.
\end{proof}

For $z\in\C$ with $|z|\le1$ and $t\in\R^d$, define
\[
\hB(z,t)  :L^1(\overline\Delta)\to L^1({\overline Y}), \qquad \hB(z,t)(v)
=\sum_{n=1}^\infty z^n B_{t,n}(v)
\]
where 
\[
B_{t,n}(v) =1_{\overline Y}P_t^n(1_{D_n}v), \qquad
D_n=\{(y,R(y)-n):y\in \overline Y,\,R(y)>n\}.
\]

\begin{prop} \label{prop:B}
	There exists $\delta_0>0$ such that 
	when regarded a  functions with values in the set of continuous linear operators from $\cB_0$ to $\cB_1({\overline Y})$,
	\begin{itemize}
		\item[(a)] $z\mapsto \hB(z,t)$
		is analytic on $B_{1+\delta_0}(0)$ for all $t\in\R^d$;
		\item[(b)] $(z,t)\mapsto (\partial_z\hB)(z,t)$ is $C^1$ on $B_{1+\delta_0}(0)\times\R^d$.
	\end{itemize}
\end{prop}
The proof is analogous to the one of~\cite[Proposition 5.4]{MPT}, but, again, we have have to deal with the presence of $K_u$ in Proposition~\ref{prop:Ruv}. So we detail this proof.

\begin{pfof}{Proposition~\ref{prop:B}}
We observe that
	$B_{t,n}v=Q(1_{\{R>n\}}v_{t,n})$ where
	\[
	v_{t,n}(y):=e^{i\langle t,\overline\Psi_n(y,R(y)-n)\rangle}
	v(y,R(y)-n).
	\]
Since $R$ is constant on partition elements, it follows from Proposition~\ref{prop:Ruv} that
	\begin{equation}\label{BBB1}
	\left\Vert B_{t,n}(v)\right\Vert_{\mathcal B_1({\overline Y})}\ll \left\Vert 1_{\{R>n\}}(|v_{t,n} |+K_{v_{t,n}})\right\Vert_{L^1(\mu_{\overline Y})}\le (|t|+1)\|1_{\{R>n\}}\|_{L^1(\mu_{\overline Y})}
	\end{equation}
	and
	\begin{equation}\label{dBBB1}
	\|\partial_j B_{t,n}(v)\|_{\mathcal B_1({\overline Y})}\ll \left\Vert 1_{\{R>n\}}(|\partial_jv_{t,n} |+K_{\partial_j v_{t,n}})\right\Vert_{L^1(\mu_{\overline Y})} \, .
	\end{equation}
But on $\{R>n\}$,
\begin{equation}\label{BBB1bis}
K_{v_{t,n}}\le (1+|t|C_{\overline\Psi})\Vert v\Vert_{\mathcal B_0}\quad\mbox{and}\quad 
K_{\partial_j v_{t,n}}\le |\overline\Psi|_R K_{v_{t,n}} + C_{\overline\Psi}\Vert v\Vert_\infty
,
\end{equation}
where $C_{\overline\Psi}$ is the constant appearing in Proposition~\ref{prop:g}. 
Indeed, for any $a\in\alpha$ and any $y,y'\in a$, writing $[v]_{\mathcal B_0}$ for the Lipschitz constant of $v$ and using the fact that
\begin{align*}
\overline s((y,R(a)-n),(y',R(a)-n))&=\overline s((F(y),0),(F(y'),0))+n\ge s'(F(y),F(y'))+n\\
&=s'(y,y')+n-1\ge s'(y,y')\, ,
\end{align*}
we observe that
\begin{align*}
&\left|e^{i\langle t, \overline\Psi_n(y,R(y)-n))\rangle}v(y,R(y)-n)-e^{i\langle t, \overline\Psi_n(y',R(y')-n))\rangle}v(y',R(y')-n)\right|\\
&\le \left(|t| C_{\overline\Psi}\Vert v\Vert_\infty+[v]_{\mathcal B_0}\right) \beta^{s'(y,y')}\, ,
\end{align*}
which ends the proof of the first part of~\eqref{BBB1bis}. The second
comes from the standard bound of the Lipschitz constant of a product. 
It follows from~\eqref{BBB1},~\eqref{dBBB1} and~\eqref{BBB1bis} that
\begin{equation*}
	\left\Vert B_{t,n}(v)\right\Vert_{\mathcal B_1({\overline Y})}\ll (1+|t|)\|1_{\{R>n\}}\|_{L^1(\mu_{\overline Y})}\Vert v\Vert_{\mathcal B_0}
\end{equation*}
and
\begin{equation*}
\|\partial_j B_{t,n}(v)\|_{\mathcal B_1({\overline Y})}\ll (1+|t|)\left\Vert 1_{\{R>n\}}|\overline\Psi|_R\right\Vert_{L^1(\mu_{\overline Y})} \Vert v\Vert_{\mathcal B_0}\, .
\end{equation*}
The result follows from the fact that $\mu_{\overline Y}(R>n)$ decays exponentially fast in $n$ as $n\rightarrow +\infty$ and from 
the fact that $|\overline\Psi|_R$ is $L^{2-\epsilon}$ for any $\epsilon>0$.
\end{pfof}

For $z\in\C$ with $|z|\le1$ and $t\in\R^d$, define
\[
\hE(z,t)  :\cB_0\to L^1(\overline\Delta), \qquad \hE(z,t)(v)
=\sum_{n=1}^\infty z^n E_{t,n}(v)
\]
where 
\(
E_{t,n}(v)(y,\ell) =1_{\{\ell>n\}}P_t^n(v)(y,\ell).
\)

\begin{prop} \label{prop:E}
	There exists $\delta_0>0$ such that 
	regarded as operators from $\cB_0$ to $L^1(\overline\Delta)$,
	\begin{itemize}
		\item[(a)] $z\mapsto \hE(z,t)$
		is analytic on $B_{1+\delta_0}(0)$ for all $t\in\R^d$;
		\item[(b)] $(z,t)\mapsto \hE(z,t)$ is $C^0$ on $B_{1+\delta_0}(0)\times\R^d$;
	\end{itemize}
\end{prop}

\begin{proof}
The proof goes word for word as~\cite[Proof of Proposition~5.5]{MPT} since it just uses the H\"older inequality 
combined with the fact that $\Vert R 1_{R>n}\Vert_{L^1({\overline Y})}$
decays exponentially fast in $n$ as $n\rightarrow +\infty$.
\end{proof}

\subsubsection{Further estimates}
The results contained in this subsection are the analogue of~\cite[Proposition 5.6--5.9]{MPT}. Since, we will have to deal with $K_u$ coming from~\eqref{Quv}, some modifications are required in these proofs. We detail the parts corresponding to these modifications and indicate
which parts of the proofs of~\cite{MPT} remain the same.

\begin{prop} \label{prop:furtherR}
	There exist $C>0$, $\delta_0>0$ and $b>0$ such that
	\[
	\|\partial_j\partial_z \hQ(z,t+h)
	-\partial_j\partial_z \hQ(z,t)\|_{{\cB_1({\overline Y})}}
	\le C |h| L_h^2
	\big\{1
	+ |h|^{-b\log |z|}L_h(|z|-1)\big\},
	\]
	for all
	$t,h\in B_{\delta_0}(0)$, all $z\in\C$ with $1\le|z|\le 1+\delta$,
	and all $j=1,\dots,d$.
\end{prop}

\begin{proof}
We observe that
	\[
	\partial_j\partial_z Q(z,t+h)(v)-\partial_j\partial_z Q(z,t)(v) = i Q\left( w_{t,h}
	R z^{R-1} v)\right)\, ,
	\]
with $w_{t,h}:=(\overline\Psi_R)_j e^{i\langle t, \overline \Psi_R\rangle}\left(e^{i\langle h, \overline \Psi_R\rangle}-1\right)$.
It follows from Propositions~\ref{prop:Ruv} and~\ref{prop:g} that
	\begin{align*}
&	\|\partial_j\partial_zQ(z,t)\|_{\cB_1({\overline Y})} \\
& \ll \sum_{n=1}^\infty n|z|^{n-1}
	\left\Vert  1_{\{R=n\}}\left(|\overline\Psi_R(e^{i\langle h,\overline\Psi_R\rangle}-1)|+K_{w_{t,h}}\right)\right\Vert_{L^1(\mu_{\overline Y})} ,
	\\
 & \ll \sum_{n=1}^\infty n|z|^{n-1}
\left\Vert 1_{\{R=n\}}\left(|\overline\Psi_R(e^{i\langle h,\overline\Psi_R\rangle}-1)|(1+(1+|t|)C_{\overline\Psi})+C_{\overline\Psi} |h|\, (|\overline\Psi_R|+C_{\overline \Psi})\right)\right\Vert_{L^1(\mu_{\overline Y})} ,
	\end{align*}
Indeed, for all partition element $a$ and for all $y,y'\in a$,
\begin{align*}
&\left|w_{t,h}(y)-w_{t,h}(y')\right|
\\
&\le\left|(\overline\Psi_R(y))_je^{i\langle t, \overline \Psi_R(y)\rangle}-(\overline\Psi_R(y'))_je^{i\langle t, \overline \Psi_R(y')\rangle}\right|\, 
\left|e^{i\langle h, \overline \Psi_R(y)\rangle}-1\right|+ \left|\overline\Psi_R(y')\right||e^{i\langle h, \overline \Psi_R(y)\rangle}-e^{i\langle h, \overline \Psi_R(y')\rangle}|\\
&\le  C_{\overline\Psi}(1+|t|)|\overline\Psi_R|\beta^{s'(y,y')}\left|e^{i\langle h, \overline \Psi_R(y)\rangle}-1\right|+ |h| \left|\overline\Psi_R(y')\right|C_{\overline\Psi}\beta^{s'(y,y')}\, ,
\end{align*}
due to~\eqref{Kpsiexp} and Proposition~\ref{prop:g}; this gives the
required domination of $K_{w_{t,h}}$.
Thus, we have proved that
	\begin{align}
\nonumber	\|\partial_j\partial_zQ(z,t)\|_{\mathcal L(\cB_1({\overline Y}))} &\ll 
\left\Vert \left(|\overline\Psi_R(e^{i\langle h,\overline\Psi_R\rangle}-1)|+|h|(1+|\overline\Psi_R|)\right) Rz^R\right\Vert_{L^1(\mu_{\overline Y})} \\
&\ll  \left( \left\Vert|\overline\Psi|_R\min(|h|\, |\overline\Psi|_R,1)Rz^R\right\Vert_{L^1(\mu_{\overline Y})}+|h|\right)\, ,\label{majodjdz}
\end{align}
provided $\delta$ is small enough since $\overline\Psi_R\in L^{2-\varepsilon}$ for all $\varepsilon\in(0,2)$ and since
$\mu_{\overline Y}(R>n)$ decays exponentially fast in $n$ as $n\rightarrow +\infty$. It remains to estimate the first term of the right hand side of~\eqref{majodjdz}. 
For any $x\in {\overline Y}$, let us write $\psi(x)$ for the supremum of the upper integer part of $|\overline\Psi|_R$ on the partition atom containing $x$. 
Then
\begin{align}
\left\Vert|\overline\Psi|_R\min(|h|\, |\overline\Psi|_R,1)Rz^R\right\Vert_{L^1(\mu_{\overline Y})} 
\le \left\Vert \psi\min(|h|\, \psi,1)Rz^R\right\Vert_{L^1(\mu_{\overline Y})} 
\label{majodjdzbis}
\end{align}
provided $\delta$ is small enough.
But
	\[
	\left\Vert \psi\min(|h|\, \psi,1)Rz^R\right\Vert_{L^1(\mu_{\overline Y})}\ll
	 \sum_{m,n=1}^\infty r_{m,n}
	\]
	where
	\begin{align}\label{eq:rmn}
	 r_{m,n}=\mu_{\overline Y}(\psi=m,R=n) mn \min\{|h|m,1\}|z|^n.
	\end{align}
The rest of the proof then follows the same lines as~\cite[Proposition~5.6]{MPT}.
\end{proof}

\begin{rmk} \label{rmk:furtherR}
	Similarly to~\cite[Remark 5.7]{MPT}, a simplified version
	of the argument used in the proof of Proposition~\ref{prop:furtherR}
	(only for the derivative in $j$)
	gives
	\[
	\|\partial_j\hQ(z,t+h)
	-\partial_j\hQ(z,t)\|_{{\cB_1({\overline Y})}}
	\le C |h| L_h\big\{1
	+ |h|^{-b\log |z|}L_h(|z|-1)\big\}.
	\]
\end{rmk}

\begin{prop} \label{prop:furtherA}
	There exist $C>0$, $\delta>0$ and $b>0$ such that
	\[
	\|\partial_j \hA(z,t+h)
	-\partial_j \hA(z,t)\|_{\cB_1({\overline Y})\mapsto L^1(\overline\Delta)}
	\le C |h| L_h\big\{1 + |h|^{-b\log |z|}L_h(|z|-1)\big\},
	\]
	for all
	$t,h\in B_\delta(0)$, all $z\in\C$ with $1\le|z|\le 1+\delta$,
	and all $j=1,\dots,d$. 
\end{prop}

\begin{proof}
	We have
	\[
	(\hA(z,t)v)(y,\ell)=\sum_{n=1}^\infty z^n1_{\{\ell=n\}}e^{i\langle t,\overline\Psi_n(y,0)\rangle}v(y)
	=z^\ell e^{i\langle t,\overline\Psi_\ell(y,0)\rangle}v(y),
	\]
	for $\ell=0,\dots,R(y)-1$.
	Hence
	we can proceed as in the proof of
	Proposition~\ref{prop:furtherR}, except that
	there is one less factor of $n$
	 (and so one less factor of $L_h$).
\end{proof}

\begin{prop} \label{prop:furtherB}
	There exist $C>0$, $\delta>0$ and $b>0$ such that
	\[
	\|\partial_j \hB(z,t+h)
	-\partial_j \hB(z,t)\|_{\cB_0\mapsto \cB_1({\overline Y})}
	\le C |h| L_h \big\{1
	+ |h|^{-b\log |z|}L_h(|z|-1)\big\},
	\]
	for all
	$t,h\in B_\delta(0)$, all $z\in\C$ with $1\le|z|\le 1+\delta$,
	and all $j=1,\dots,d$, 
\end{prop}

\begin{proof} 
Let $v\in\mathcal B_0$. 
	In the notation of Proposition~\ref{prop:B},
	\begin{equation}\label{djBth1}
	\partial_j \hB(z,t+h) (v)
	-\partial_j \hB(z,t)(v)=i\sum_{n\ge 1}z^nQ(1_{\{R>n\}}\left(\partial_jv_{t+h,n}-\partial_ jv_{t,n}\right))
	\end{equation}
with
	\[
	v_{t,n}(y):=e^{i\langle t,\overline\Psi_n(y,R(y)-n)\rangle}
	v(y,R(y)-n).
	\]
Therefore
\begin{equation}\label{diffdjv}
\partial_jv_{t+h,n}-\partial_ jv_{t,n}
=  \left(\overline\Psi_n(\cdot,R(\cdot)-n)e^{i\langle t,\overline\Psi_n(\cdot,R(\cdot)-n)\rangle}\right)_j
\left(e^{i\langle h,\overline\Psi_n(\cdot,R(\cdot)-n)\rangle}-1\right)
v(\cdot,R(\cdot)-n)\, .
\end{equation}
It follows from Proposition~\ref{prop:Ruv} that
	\begin{equation}\label{djBth2}
	\left\Vert Q(1_{\{R>n\}}\partial_jv_{t+h,n}-\partial_ jv_{t,n})\right\Vert_{\mathcal B_1({\overline Y})} \ll \left\Vert 1_{\{R>n\}}\left(|\partial_jv_{t+h,n}-\partial_ jv_{t,n})|+K_{\partial_jv_{t+h,n}-\partial_ jv_{t,n}}\right)\right\Vert_{L^1(\mu_{\overline Y})}\, .
	\end{equation}
We proceed as in the proof of Proposition~\ref{prop:furtherR}.
It follows from~\eqref{diffdjv} that
\begin{equation}\label{djBth3}
\left\Vert 1_{\{R>n\}}(\partial_jv_{t+h,n}-\partial_ jv_{t,n})\right\Vert_{L^1(\mu_{\overline Y})}
\le \left\Vert 1_{\{R>n\}}\overline\Psi_n(\cdot,R(\cdot)-n)
\left(e^{i\langle h,\overline\Psi_n(\cdot,R(\cdot)-n)\rangle}-1\right)\right\Vert_{L^1(\mu_{\overline Y})}
\Vert v\Vert_\infty\, .
\end{equation}
Furthermore, due to Proposition~\ref{prop:g}, for all $a\in\pi$ and all $y,y'\in a$,
\begin{align*}
&\left|(\partial_jv_{t+h,n}-\partial_ jv_{t,n})(y)-(\partial_jv_{t+h,n}-\partial_ jv_{t,n})(y')\right|\\
&\ll \beta^{s(y,y')}
\left\Vert 1_{R>n}\left(1+|\overline\Psi_n(\cdot,R(\cdot)-n)|\right)\left(e^{i\langle h,\overline\Psi_n(\cdot,R(\cdot)-n)\rangle}-1\right)\right\Vert_{L^1(\mu)}
(1+C_{\overline\Psi})\Vert v\Vert_{\mathcal B_0}\\
&+|h|\beta^{s(y,y')}C_{\overline\Psi}\left\Vert 1_{R>n}(\overline\Psi_n(\cdot,R(\cdot)-n))_j\right\Vert_{L^1(\mu)}\Vert v\Vert_\infty\, .
\end{align*}
This combined with~\eqref{djBth2} and~\eqref{djBth3} ensures that
\begin{align}
\nonumber&\left\Vert Q(1_{R>n}\partial_jv_{t+h,n}-\partial_ jv_{t,n})\right\Vert_{\mathcal B_1({\overline Y})} \\
&\ll \left\Vert 1_{\{R>n\}}(1+\overline\Psi_n(\cdot,R(\cdot)-n))
\left(e^{i\langle h,\overline\Psi_n(\cdot,R(\cdot)-n)\rangle}-1\right)\right\Vert_{L^1(\mu_{\overline Y})}
\Vert v\Vert_{\mathcal B_0}+|h|\Vert v\Vert_\infty\nonumber\\
&\ll \left\Vert 1_{\{R>n\}}\psi
\min(|h|\, |\psi|)\right\Vert_{L^1(\mu_{\overline Y})}
\Vert v\Vert_{\mathcal B_0}+|h|\Vert v\Vert_\infty\, .
\label{djBth4}
\end{align}
But
\[
\sum_{n\ge 1}\left\Vert z^n 1_{\{R>n\}}\psi
\min(|h|\, |\psi|)\right\Vert_{L^1(\mu_{\overline Y})}  =\sum_{m,n\ge 1}	 \mu_{\overline Y}(\psi=m,R=n) m \min\{|h|m,1\}|z|^n\, ,
\]
which can be estimated as in the proof of Proposition~\ref{prop:furtherR}.
We conclude by combining this estimate with~\eqref{djBth1} and~\eqref{djBth4}.
~\end{proof}

The rest of the proofs of~\cite{MPT} (corresponding to Section 5.2 therein that provide all the required spectral properties for $\hQ(z,t)$) go through unchanged.

\section{Smoothness of $\overline\tau$ and $\chi$}
\label{sec:app}
Recall that $\overline\Delta\subset\Delta$,
Using the fact that $T\circ \pi=\pi\circ f_{\Delta}$ on $\Delta$, and that $f_{\overline\Delta}=\overline\pi\circ f_\Delta$ on $\overline\Delta$, $\overline\tau$ and $\chi_0$ defined in Section~\ref{sec:proofjCLT} can be rewritten as follows
\[ 
\overline\tau:=\widetilde\tau\circ\pi +\sum_{n\ge 1}\left(\tau\circ T^{n}\circ \pi -\tau\circ T^{n-1}\circ\pi\circ\overline\pi\circ f_{\Delta}\right)\quad\mbox{on }\overline\Delta
\]
and
\[
\chi_0:=\sum_{n\ge 0}\chi_{0,n},\quad\mbox{with }\chi_{0,n}:=\left(\tau\circ T^n\circ \pi-\tau\circ T^n\circ\pi\circ \overline\pi\right)\quad\mbox{on }\Delta\, .
\]
First, observe that, for every $x\in\Delta$, $\pi(x)$ and $\pi(\overline\pi(x))$ are in the same stable manifold, thus $d(T^n(\pi(x)),T^n(\pi(\overline\pi(x))))\le C_1\beta_1^n$, and so
\begin{align}
\forall x\in\Delta,\quad |\chi_{0,n}|\le 2C_1\beta_1^{n} \, .\label{majochi0}
\end{align} 
Analogously, for all $x\in\overline\Delta$,
\begin{align}
| \tau(T^n(\pi(x)))-\tau(T^{n-1}(\pi(\overline\pi(f_\Delta(x)))))|
&=| \tau(T^{n-1}(\pi(f_\Delta(x))))-\tau(T^{n-1}(\pi(\overline\pi(f_\Delta(x)))))|\\
&\le 2C_1\beta_1^{n-1}\, .
\end{align}
This ensures that $\chi_0$ and $\overline\tau$ are well defined and we have proved the identity
\[
\widetilde\tau\circ\pi-\chi_0+\chi_0\circ f_\Delta=\widetilde\tau\circ\pi\circ \overline\pi+\sum_{n\ge 0}\left(\tau\circ T^{n+1}\circ\pi\circ \overline\pi-\tau\circ T^n\circ \pi\circ \overline\pi\circ f_\Delta\right)=\overline\tau\circ\overline\pi
\]
since $\overline\pi\circ f_\Delta= f_{\overline\Delta}\circ\overline\pi$.

For any $x,y\in\overline\Delta$ such that $\overline s(x,y)= N\ge 2k$, 
then $T^N(\pi(x))$ and $T^N(\pi(y))$ are in a same unstable  manifold and
 the same holds true for $T^{N-1}(\overline\pi(f_\Delta(x)))$ and
 $T^{N-1}(\overline\pi(f_\Delta(y)))$.
This implies that, for every $n=0,...,N$, 
$d(T^n(\pi(x)),T^n(\pi(y)))\le C_1\beta_1^{N-n}$ and
$d(T^{n-1}(\pi(\overline \pi(f_\Delta(y)))),T^{n-1}(\pi(\overline \pi(f_\Delta(x))))\le C_1\beta_1^{N-n-1}$. Therefore
\begin{align*}
&\left|\overline\tau(x)-\overline\tau(y)\right|\le \left|\widetilde\tau(\pi(x))-\widetilde\tau(\pi(y))\right|
+2\sum_{n\ge\lceil\overline s(x,y)/2\rceil+1}\Vert
\tau\circ T^n\circ\pi-\tau\circ T^{n-1}\circ\pi\circ\overline\pi\circ f_\Delta\Vert_\infty\\
&\ \ \ \ \ +\sum_{n=1}^{\lceil\overline s(x,y)/2\rceil}
\left|  \tau(T^n(\pi(x)))-\tau(T^n(\pi(y)))-[\tau(T^{n-1}( \pi(\overline\pi( f_\Delta(x)))))-\tau(T^{n-1}( \pi(\overline\pi( f_\Delta(y)))))]\right|\\
&\le 2 \left(\beta_1^{\overline s(x,y)}+2\sum_{n\ge\lceil\overline  s(x,y)/2\rceil+1}\beta_1^{n}+2\sum_{n=1}^{\lceil\overline  s(x,y)/2\rceil}\beta_1^{\overline s(x,y)-n-1} \right)\\
&\le  2 \beta_1^{\frac{\overline s(x,y) }2}\left(1+4 \frac{\beta_1^{-1}}{1-\beta_1}\right)\, .
\end{align*}
Now, let us prove that $\chi_0$ satisfies
\[
\sup_{k\ge 1}
\sup_{x,y:s(x,y)>2k}\frac{|\chi_0(f_\Delta^k(x))-\chi_0(f_\Delta^k(y))|}{\beta^k}<\infty\, .
\]
Observe that
\begin{equation}\label{chi0fk}
\chi_0\circ f_\Delta^k=\sum_{n\ge 0}\chi_{0,n}\circ f_\Delta^k,\quad\mbox{and }\chi_{0,n}\circ f_\Delta^k=\tau\circ T^{n+k}\circ\pi-
     \tau\circ T^n\circ\pi\circ\overline\pi\circ f_\Delta^k\, .
\end{equation}
Let $x,y\in\Delta$ be such that $s(x,y)=N>2k$.
\begin{itemize}
\item
for $n> k/2$, we observe that 
it follows from~\eqref{majochi0} that
\begin{equation}
|\chi_{0,n}(f_\Delta^k(x))-\chi_{0,n}(f_\Delta^k(y))|\le 2\Vert \chi_{0,n}\Vert_\infty
\le 4C_1\beta_1^{n}\, .\label{DIFF0}
\end{equation}

\item for $n=0,...,k/2$, we observe that
$\pi(x),\pi(\overline\pi(x))$ are in a same stable manifold, $\pi(y),\pi(\overline\pi(y))$ are also in a same stable manifold, 
and that $T^N(\pi(\overline\pi(x)),T^N(\pi(\overline\pi(y))))$ are in a same unstable manifold. Therefore
\[
|\tau(T^{n+k}(\pi(x)))-\tau(T^{n+k}(\pi(\overline\pi(x))))|\le 2C_1\beta_1^{n+k}\, ,
\]
\[
|\tau(T^{n+k}(\pi(y)))-\tau(T^{n+k}(\pi(\overline\pi(y))))|\le 2C_1\beta_1^{n+k}\, ,
\]
\[
|\tau(T^{n+k}(\pi(\overline\pi(x))))-\tau(T^{n+k}(\pi(\overline\pi(y))))|\le 2C_1\beta_1^{N-n-k}\le 2C_1\beta_1^{k-n}
\]
Thus
\begin{equation}\label{DIFF1}
|\tau\circ T^{n+k}(\pi(x))-\tau\circ T^{n+k}(\pi(y))|\le 6C_1\beta_1^{k-n}\, .
\end{equation}
Furthermore $x_k'=\overline\pi(f_{\Delta}^k(x)),y_k'=\overline\pi(f_{\Delta}^k(y))\in\overline \Delta$ and $\overline s(x_k',y'_k)=N-k$. Thus  $T^{N-k}(\pi(\overline\pi(f_{\Delta}^k(x))))$ and $T^{N-k}(\pi(\overline\pi(f_{\Delta}^k(y))))$ are in the same unstable manifold and so
\begin{equation}\label{DIFF2}
|\tau(T^{n}(\pi(\overline\pi(f_{\Delta}^k(x)))))-\tau(T^{n}(\pi(\overline\pi(f_{\Delta}^k(x)))))|\le 2C_1\beta_1^{N-k-n}\le  2C_1\beta_1^{k-n}
\end{equation}
It follows from~\eqref{chi0fk} combined with~\eqref{DIFF0},~\eqref{DIFF1} and~\eqref{DIFF2} that
\end{itemize}
\begin{align*}
\left|\chi_0(f_\Delta^k(x))-\chi_0(f_\Delta^k(y))\right|
&\le \sum_{n>k/2}4C_1\beta_1^n+\sum_{n=0}^{k/2}6C_1\beta_1^{k-n} \ll \beta_1^{\frac k2}\, .
\end{align*}
We conclude since $ \beta_1^{\frac 12}\le \beta$.\\

{\bf Acknowledgements:}
We would like to thank CIRM, the Erd\"os Center, Brest and also to
Leiden and Utrecht universities for their hospitality.
We are grateful to Ian Melbourne for interesting discussions and for the important role he played in the LLD result obtained in~\cite{MPT} which is one of the crucial ingredients used in the present work.
FP conduced this work within the framework of the Henri
Lebesgue Center ANR-11-LABX-0020-01 and is supported by the ANR projects GALS (ANR 23-CE40-0001) and RAWABRANCH (ANR-23-CE40-0008).


\begin{thebibliography}{10}

\bibitem{AN17}
J.~Aaronson and H.~Nakada, On multiple recurrence and other properties of ‘nice’
infinite measure-preserving transformations, \emph{Erg. Th. Dynam. Systems} \textbf{37} (2017),
1345--1368.



\bibitem{AT20} J.~Aaronson and D.~Terhesiu.
Local limit theorems for fibred semiflows.
\emph{Discrete Contin. Dyn. Syst.}  \textbf{40} (2020) 6575--6609.

\bibitem{BalintGouezel06}
P.~B{\'a}lint and S.~Gou{\"e}zel. Limit theorems in the stadium billiard.
  \emph{Commun. Math. Phys.} \textbf{263} (2006) 461--512.

\bibitem{BBM19}
P.~B{\'a}lint, O.~Butterley, and I.~Melbourne. Polynomial decay of correlations
  for flows, including Lorentz gas examples. \emph{Comm. Math. Phys.}
  \textbf{368} (2019) 55--111.
  
  \bibitem{BBY} Q.~Berger, M.~Birkner, L.~Yuan.
Collective vs.\ individual behaviour for sums of i.i.d.\ random variables: appearance of the one-big-jump phenomenon.
To  appear in
Ann. Fac. Sci. Toulouse.
  
\bibitem{BS81}L.~A.~Bunimovich and Ya.~G.~Sinai, Statistical properties of Lorentz gas with
periodic configuration of scatterers, \emph{Comm. Math. Phys.} \textbf{78} (1981), No. 4, 479–497.

\bibitem{BCS91} L. A. Bunimovich, Ya. G. Sinai and N. I. Chernov, Statistical properties of two-dimensional hyperbolic billiards, (Russian) \emph{Uspekhi Mat. Nauk} \textbf{46} (1991), No. 4 (280), 43–92,
192; translation in Russian Math. Surveys 46 (1991), No. 4, 47–106.

  
  \bibitem{Chernov99}
N.~Chernov. Decay of correlations and dispersing billiards. \emph{J. Statist.
  Phys.} \textbf{94} (1999) 513--556.

\bibitem{ChDo09} N.\ Chernov, D.\ Dolgopyat,
\emph{Anomalous current in periodic Lorentz gases with
	infinite horizon.}
Russian Math. Surveys\  {\bf 64-4} (2009),  651--699


\bibitem{ChernovMarkarian} 
N.~Chernov, and R.~Markarian.
Chaotic billiards. 
Mathematical Surveys and Monographs 127. Providence, American Mathematical Society (AMS) xii, 316 p. (2006).


\bibitem{DN16} D.~Dolgopyat and P.~N{\'a}ndori.
 Non equilibrium density profiles in Lorentz tubes with thermostated
boundaries. Commun. Pure Appl. Math. \textbf{69} (2016) 649--692.




\bibitem{DN20} D.~Dolgopyat and P.~N{\'a}ndori.
On mixing and the local central limit theorem for hyperbolic flows.
\emph{Annales l'Institut Henri Poincar\'e, Prob. et Stat.} \textbf{40} (2020) 142--174.

\bibitem{DNL21} D.~Dolgopyat, P.~N{\'a}ndori and M~Lenci.
 Global observables for random walks: law of large numbers.
\emph{} \textbf{57} (2021) 94--115.

\bibitem{DN22} D.~Dolgopyat and P.~N{\'a}ndori.
 Infinite measure mixing for some mechanical systems
\emph{Adv. in Math.}\textbf{410 B} (2022) 108--757

\bibitem{DNP22}D.~Dolgopyat, P.~N{\'a}ndori and  F. P{\`e}ne.
Asymptotic expansion of correlation functions for $\mathbb Z^d$-covers of hyperbolic ﬂows.
\emph{Annales l'Institut Henri Poincar\'e, Prob. et Stat.} \textbf{58} (2022), No. 2, 1244--1283.


\bibitem{Gouezel07}
S.~Gou{\"e}zel, {Statistical properties of a skew product with a curve of
  neutral points}, \emph{Ergodic Theory Dynam. Syst.} \textbf{27} (2007),
  123--151.
  
\bibitem{Gouezel11} S.~Gou{\"e}zel,
\emph{Correlation asymptotics from large deviations in dynamical systems with infinite measure,}
Colloquium Math.\  \textbf{125} (2011), 193--212.


 \bibitem{Krickeberg67}
K.~Krickeberg, Strong mixing properties of {M}arkov chains with infinite
   invariant measure, \emph{Proc. {F}ifth {B}erkeley {S}ympos. {M}ath.
   {S}tatist. and {P}robability ({B}erkeley, {C}alif., 1965/66), {V}ol. {II}:
   {C}ontributions to {P}robability {T}heory, {P}art 2}, Univ. California Press,
  Berkeley, Calif., 1967, pp.~431--446.

\bibitem{Lenci10}M.~Lenci, On infinite-volume mixing, Comm. Math. Phys. 298 (2010), no. 2, 485--514.

\bibitem{Lorentz05}H.~A.~Lorentz. 
The motion of electrons in metallic bodies.
\emph{Koninklijke Nederlandse Akademie van Wetenschappen (KNAW), proceeding of the section of sciences} \textbf{7} (1905), No. 2, 438--593.

\bibitem{MT04}
I.~Melbourne and A.~T{\" o}r{\" o}k, {Statistical limit theorems for suspension
  flows}, \emph{Israel J. Math.} \textbf{144} (2004), 191--209.


\bibitem{MPT}
I. Melbourne, F. P{\`e}ne and D. Terhesiu. {Local large deviations for periodic infinite horizon Lorentz gases.} To appear in J. d'Analyse.

\bibitem{MT12}
 I.\ Melbourne, D.\ Terhesiu. 
{Operator renewal theory and mixing rates for dynamical systems with infinite measure},  
\emph{Invent.\ Math.\ } \textbf{189}, (2012) no.\ 1, 61--110.

\bibitem{MT17} I.~Melbourne, D.~Terhesiu. 
{Renewal theorems and mixing for non Markov flows with infinite measure.} 
\emph{Ann.\ Inst.\ H.\ Poincar{\'e} (B) Probab.\ Stat.\ }\textbf{56} (2020) 449--476.


\bibitem{MTejp} I.~Melbourne and D.~Terhesiu. 
Analytic proof of multivariate stable local large deviations and application to deterministic dynamical systems.
\emph{Electr.\ J.\ Probab.\ } \textbf{27} (2022) 1--17. 




\bibitem{FP09AIHP} F.~P{\`e}ne.
Planar Lorentz process in random scenery.
\emph{Ann.\  Inst.\ Henri Poincar{\'e}, Prob.\ and Stat.\ }\textbf{45} (2009) 818--839.

\bibitem{Pene18b}F.~P\`ene. Mixing in infinite measure for Zd-extensions, application to the periodic Sinai billiard. 
\emph{Chaos Solitons Fractals} \textbf{106} (2018), 44--48.

\bibitem{Pene18} F.~P{\`e}ne.
Mixing and decorrelation in infinite measure: the case of the periodic Sinai Billiard.
\emph{Annales de l'Institut Henri Poincar{\'e}, Prob. and Stat.} \textbf{55}, (2019) 378--411




\bibitem{PeneTerhesiu21}
F.  P{\`e}ne and  D. Terhesiu.
Sharp error term in local limit theorems and mixing for Lorentz gases with infinite horizon.
\emph{Commun. Math. Phys.} \textbf{382} (2021) 1625--1689.

\bibitem{Resnick} S. Resnick, Extreme Values, Regular Variation, and Point Processes, Springer Series in
Operation Research and Financial Engineering, Springer, New York, 2008.

\bibitem{Roz90} L.~V.~Rozovskii.
Probabilities of large deviations of sums of independent random variables with common distribution function in the domain of attraction
of the normal law.
 Teor.\ Veroyatnost.\ i Primenen.\ \textbf{34} (1989), no. 4, 686–-705; translation in Theory Probab.\ Appl.\ \textbf{34} (1989), no. 4, 625


\bibitem{Sinai70}Ya. G. Sinai.
 Dynamical systems with elastic reflections.
\emph{Russ. Math. Surv.} \textbf{25} (1970), No.2, 137--189.

\bibitem{SV04} D.~Sz\'{a}sz, T.~Varj\'{u}, 
 Local limit theorem for the Lorentz process and its recurrence in the plane,
\emph{Ergodic Theory Dynam. Systems} \textbf{24} (2004) 254--278.


\bibitem{SV07}
D. Sz{\'a}sz and T. Varj{\'u}. Limit Laws and Recurrence for the Planar Lorentz Process with Infinite Horizon. 
\emph{J. Statist. Phys. } \textbf{129} (2007) 59--80.  

\bibitem{Thaler00}
M.~Thaler. The asymptotics of the {P}erron-{F}robenius operator of a class of
  interval maps preserving infinite measures. \emph{Studia Math.} \textbf{143}
  (2000) 103--119.
  
  
 \bibitem{T22} 
 D. Terhesiu. Krickeberg mixing for Z extensions of Gibbs Markov semiflows. 
\emph{Monatsh. Math.}\textbf{198} (2022)  859--893 


\bibitem{Young98} L.-S.\ Young.
Statistical properties of dynamical systems with some hyperbolicity.
\emph{Ann.\ of Math.} {\bf 147} (1998) 585--650.


\end{thebibliography}
 \end{document}